\documentclass[11pt,a4paper]{article}
\usepackage[utf8]{inputenc}
\usepackage[english]{babel}
\usepackage[T1]{fontenc}
\usepackage{amsmath}
\usepackage{amsfonts}
\usepackage{amssymb,amsthm,esint}
\usepackage{fullpage}
\usepackage{graphicx}
\usepackage{color}
\usepackage{bm}
\usepackage{hyperref}
\usepackage{textcomp}
\usepackage{mathtools, bm}
\usepackage{bigints}
\usepackage{dsfont}
\usepackage{colortbl}
\usepackage[normalem]{ulem}
\usepackage{empheq}
\usepackage{dsfont}
\usepackage{placeins}
\usepackage{comment}
\usepackage{authblk}
\usepackage{enumitem}

\DeclareMathAlphabet\mathbfcal{OMS}{cmsy}{b}{n}

\hypersetup{urlcolor=blue, colorlinks=true}  

\hypersetup{colorlinks=true,linkcolor=blue}

\usepackage{tikz}
\newcommand*\circled[1]{\tikz[baseline=(char.base)]{
            \node[shape=circle,draw,inner sep=2pt] (char) {#1};}}

 \def\N{\mathbb{N}}
 \def\R{\mathbb{R}}
 
 \def\S{\mathbb{S}}
 
 \def\cK{\mathcal{K}}

\def\bcK{{\mathbfcal{K}}}
\def\bG{{\bm{G}}} 
\def\bz{{\bm{z}}} 
\def\bcM{{\mathbfcal{M}}}
\def\bcN{{\mathbfcal{N}}}
\def\bcC{{\mathbfcal{C}}}
\def\bcD{{\mathbfcal{D}}}

\newcommand{\cM}{\mathcal M}

\newcommand{\bc}{\mathbf c}
\newcommand{\be}{\mathbf e}
\newcommand{\bp}{\mathbf p}
\newcommand{\br}{\mathbf r}
\newcommand{\bv}{\mathbf v}
\newcommand{\bw}{\mathbf w}
\newcommand{\bN}{\mathbf N}
\newcommand{\bQ}{\mathbf Q}
\newcommand{\bR}{\mathbf R}

 \newcommand{\dps}{{\displaystyle }}

 \newtheorem{theorem}{Theorem}
 
 \newtheorem{proposition}[theorem]{Proposition}
 
 \newtheorem{lemma}[theorem]{Lemma}
 \newtheorem{remark}[theorem]{Remark}
 
\newcommand{\mybox}[1]{
\fbox{%
\parbox{\textwidth}{
#1
}%
}\medskip

}

\definecolor{cadmiumgreen}{rgb}{0.0, 0.42, 0.24}

\DeclareMathOperator*{\argmin}{arg\,min}

\author[1,2]{Robert Benda}
\author[1]{Eric Canc\`es}
\author[1]{Virginie Ehrlacher}
\author[3]{Benjamin Stamm}

\affil[1]{CERMICS, Ecole des Ponts and Inria Paris, 6 \& 8 avenue Blaise Pascal, 77455~Marne-la-Vall\'ee, France}
\affil[2]{LPICM, CNRS, Ecole Polytechnique, Institut Polytechnique de Paris, Route~de~Saclay, 91128 Palaiseau, France}
\affil[3]{ACoM, Department of Mathematics, RWTH Aachen University, Schinkelstrasse 2, D-52062 Aachen, Germany}

 \title{Multi-center decomposition of molecular densities: \\ a mathematical perspective}

\begin{document}

\maketitle
\begin{abstract} 
The aim of this paper is to analyze from a mathematical perspective some existing schemes to partition a molecular density into several atomic contributions, with a specific focus on Iterative Stockholder Atom (ISA) methods. We provide a unified mathematical framework to describe the latter family of methods and propose a new scheme, named L-ISA (for linear approximation of ISA). We prove several important mathematical properties of the ISA and L-ISA minimization problems and show that the so-called ISA algorithms can be viewed as alternating minimization schemes, which in turn enables us to obtain new convergence results for these numerical methods. Specific mathematical properties of the ISA decomposition for diatomic systems are also presented. We also review the basis-space oriented Distributed Multipole Analysis method, the mathematical formulation of which is also clarified. Different schemes are numerically compared on different molecules and we discuss the advantages and drawbacks of each approach.   
\end{abstract}
\tableofcontents

\section{Introduction}

A successful study of intermolecular interactions is closely interwoven with the quality of the representation of the electronic density and subsequently of the electrostatic potential (ESP) around molecules. When a continuous representation of the density, such as in Density Functional Theory (DFT) \cite{hohenberg1964inhomogeneous,Kohn1965} or Wave-Function Theory (WFT) \cite{jensen2017introduction}, cannot be used because of a too high computational burden, the density has to be summarized by a finite number of parameters thanks to \textit{localization schemes}. 
Classical force fields, such as GROMOS \cite{Scott1999}, CHARMM \cite{MacKerell1998} or AMBER \cite{Pearlman1995,Cornell1995}, among others, use such a representation, with (fixed) so-called partial charges assigned to all atomic sites, to represent the electronic density in a simplified way and account for electrostatic interactions with neighboring molecules. Molecular dynamics (MD) simulations based on these force fields are commonly used to gain insight into the structural or energetic properties of biologic or organic compounds.
The physical relevance of MD studies heavily relies on their quality, and thus on the specific localization scheme used to compute them. 
The purpose of the present paper is precisely to review several localization schemes from a mathematical perspective, providing a sound mathematical basis to the latter. Let us first review the existing localization schemes that have been used to derive atomic partial charges or atomic multipole moments in the context of force field parameterization. 

First, a popular method to derive partial charges is to fit the ESP generated by the set of atomic point charges to the exact quantum ESP (derived from DFT or WFT calculations) at a given geometry \cite{Singh1984}, possibly imposing additional constraints to lower the conformational dependence of the obtained optimal partial charges, such as in the RESP method \cite{Bayly1993}. Adding constraints to reproduce the total dipole moments at different conformations is also possible \cite{Reynolds1992}. So is to fit the ESP generated by the atomic charges to the approximate ESP generated by a series of distributed multipole moments \cite{Stone1981,Stone1985}, as implemented \textit{e.g.} in the Mulfit code \cite{ferenczy1991charges}. In both cases, the partial charges are chosen as those leading to the smallest error on the ESP for a suitable norm. 
This family of ESP-derived charges methods amounts to find indirectly the partial charges from the quantum calculation, using the electrostatic potential as a proxy -- and possibly the molecular dipole or quadrupole moments -- as relevant target. However, they are associated to several drawbacks, such as conformational-dependent charges -- although strategies to mitigate this undesired feature have been designed \cite{Reynolds1992} -- and lack of transferability. Chipot \textit{et al.} have for instance shown that ESP derived partial charges are not transferable from one to another similar molecule \cite{chipot1993}. In other words, the partial charge assigned to a carbon atom in a given specific local chemical environment of a given molecule cannot be used as partial charge for a carbon atom in a similar local chemical environment of another molecule.
It is also possible to derive atomic partial charges directly from a quantum calculation (namely, using the coefficients of the converged density matrix) using \textit{e.g.} Mulliken \cite{Mulliken1955} or L\"owdin \cite{Loewdin1953} population analysis, although these methods are not commonly used for partial charges parameterization -- or, if they are, charges are readjusted afterwards by fitting to the ESP, to minimize the deviation between the reference QM ESP and the atomic charges-generated ESP. Indeed, Mulliken and L\"owdin schemes adopt a basis space approach, which results in significant basis-set dependency of partial charges and lack of convergence with increasing basis set size. 

Real-space approaches, using a Voronoi-like partitioning of the space, such as the Bader Atom In Molecules \cite{Bader1994}, the Becke space-decomposition (fuzzy Voronoi cells) approach \cite{Becke1988a}, or stockholder partitionning schemes originating from the Hirshfeld method \cite{hirshfeld1977xvii} (see below), have also been used to derive partial charges, and do not suffer as much from basis set dependence -- but may suffer other drawbacks \cite{heidar2017information}. 

Finally, thanks to the recent availability of large databases of molecules already parameterized for classical force fields, modern machine-learning inspired methods have also been used to estimate partial charges, \textit{e.g.} formulating the charge assignment problem as a multiple-choice knapsack problem \cite{Engler2018}. 
Yet, although these methods allow to avoid explicit charge calculations, they may end up with atomic charges suffering from the same drawbacks as those of the database on which they have been fitted.

To correct for the too approximate nature of the ESP represented by atomic partial charges only \cite{chipot1993} (among other drawbacks), higher order moments of the charge distribution have been introduced within localization schemes, and consequently in force fields used for MD simulations. At their early times, multipole moments of the electronic density were introduced as a generalization of Mulliken population analysis, to analyse and interpret molecular wave-functions with a larger number of local descriptors \cite{Stone1981,Stone1985,Vigne-Maeder1988}. Since the advent of modern polarizable force fields, such as AMOEBA \cite{Ren2002,Ren2003,Ren2011,Wu2012}, these local descriptors have also been used successfully to improve the description of biomolecular and organic molecules, and of their mutual interactions, in particular thanks to an improved treatment of the water solvent. 
Local multipole moments -- which are used \textit{e.g.} up to quadrupole moments in AMOEBA force field -- allow to capture the local anisotropy of the charge density (contrary to atomic charges alone) and are thus a key ingredient of polarizable force fields. 

In a similar manner as for partial charges, local multipole moments can be derived either \textit{indirectly} (by fitting the multipoles-derived ESP to the quantum ESP) or \textit{directly} (\textit{e.g.} using the density matrix corresponding to a certain basis expansion, or the values of the density on a grid) from the result of a quantum calculation.
To the former rationale belongs the fitting method implemented \textit{e.g.} in the Mulfit code \cite{Ferenczy1997,Winn1997}, which allows to derive atomic charges \cite{Winn1997} or multipoles \cite{Ferenczy1997} (up to user-specified orders, for every atom) by fitting on the ESP generated by higher order local multipoles -- typically obtained previously from distributed multipole analysis (DMA) \cite{Stone1981,Stone1985}. 
To the latter family of methods belong Stone's DMA \cite{Stone1981,Stone1985}, performing a redistribution of the density matrix coefficients and elementary local multipole moments to final expansion sites (\textit{e.g.} all the atoms), and real-space partitioning schemes in the continuity of Bader concepts of atoms in molecules \cite{Bader1972}. The family of Hirshfeld-like partitionning schemes \cite{Hirshfeld1977,maslen1985atomic,ayers2000atoms,parr2005atom}, namely its iterative version Hirshfeld-I \cite{bultinck2007uniqueness,bultinck2007critical}, the iterative stockholder approach (ISA) \cite{lillestolen2008redefining,Lillestolen2009} and its variants GISA \cite{verstraelen2012conformational}, MB-ISA \cite{verstraelen2016minimal}, and more recent basis-space implementation (BS-ISA) \cite{Misquitta2014}, belong to this family of real-space oriented methods.  
This family of methods yields multipole moments which have been claimed to be rather transferable \cite{heidar2017information}.
Note that in principle, real-space partitioning schemes end up with local multipole moments located at atomic sites only (to enable force computations in force field applications), although this is not a requirement. The DMA method can also naturally allocate contributions to atoms and to other non-atomic sites, such as bond centers.

In this article, we review and clarify several existing schemes to partition a molecular density into several (\textit{e.g.} atomic) contributions or to derive directly distributed multipole moments. In Section~\ref{sec:AIM_methods_based_objective_functionals}, we first focus on the real-space oriented methods of the Hirshfeld / Iterative Stockholder Atom (ISA) family. We provide a unified mathematical framework to describe the existing ISA methods, through the formulation (\ref{eq:opt_AIM}). 
This formulation is based on information theory, a link that was first introduced by Nalewajski and Parr (see~\cite{nalewajski2000information}) and then subsequently used in the works of Ayers, Bultinck, Heidar-Zadeh and co-workers (see \cite{bultinck2007uniqueness,bultinck2007critical,heidar2017information} and references therein), but is new to our knowledge. 
It is also worthwhile to stress the mathematical properties established in~\cite{ayers2000atoms}.
Driven by formulation (\ref{eq:opt_AIM}), we propose a new AIM scheme, named L-ISA (linear approximation of ISA), and presented in Section~\ref{sec:L-ISA}. We also formulate the Minimal-Basis ISA (MB-ISA) method as a minimization of the Kullback-Leibler entropy between atomic-\textit{shell} densities and the corresponding pro-atomic-\textit{shell} densities, under equality constraints involving zero and first order moments of the atomic densities (see section \ref{sec:MB_ISA_iterations}). 
We then focus in Section \ref{sec:maths} on the mathematical analysis of the ISA and L-ISA methods. We first prove the existence and uniqueness of the solutions to the ISA and L-ISA optimization problems (Theorem~\ref{thm:WPcontPb}). While uniqueness easily follows from the strict convexity properties of the Kullback-Leibler entropy, establishing the existence of a minimizer is more difficult due to possible loss of compactness.
We then prove that the ISA and L-ISA algorithms can be interpreted as alternating minimization methods, that the associated entropy is a Lyapunov function of the algorithm, and that the $L^2$ norm between two successive iterates converges to $0$. We also prove that, in the case of the L-ISA method, the L-ISA algorithm converges toward the unique minimizer of the L-ISA optimization problem. All these results are collected in Theorem~\ref{th:convergence}. We finally prove specific mathematical properties of ISA decomposition for diatomic systems in Propositions~\ref{prop:spheric} and~\ref{prop:continuity}. 
In Section \ref{sec:DMA}, we review a popular basis-space oriented method, the Distributed Multipole Analysis (DMA), whose mathematical formulation is also clarified. 
Finally, we present in Section \ref{sec:numerical_results} some numerical results (with a focus on diatomic systems) of the different ISA schemes, namely GISA, L-ISA (section \ref{sec:L-ISA}), MB-ISA and (historical) ISA (section \ref{sec:ISA_iterations}), which have been implemented numerically. The DMA method, which has also been implemented in a more modern and modular form, enabling for more redistribution strategies to be explored (section \ref{sec:DMA}), is also compared to the ISA method in terms of convergence of local multipole moments with increasing basis set size, for some test systems.

\section{AIM methods based on objective functionals}
\label{sec:AIM_methods_based_objective_functionals}

We denote by
$$
X:=\left\{ f \in L^1(\R^3) \cap L^\infty(\R^3) \; \middle| \; \lim_{|\br| \to \infty} f(\br) = 0, \; \int_{\R^3} |\br| |f(\br)| \, d\br < \infty \right\}
$$
the Banach space of bounded, integrable, real-valued functions on the physical space $\R^3$ with finite first moments, and vanishing at infinity.  In the sequel, we consider a given nonnegative density $\rho \in X$ such that
$$
N:= \int_{\R^3} \rho(\br) \, d\br > 0,
$$
and a collection $\bR=(\bR_a)_{1 \le a \le M} \in (\R^3)^M$ of $M$ points in $\R^3$. Ground and excited state electronic densities of molecules and clusters are known to be continuous, positive, functions on $\R^3$ decaying exponentially fast at infinity (see e.g.~\cite[Theorem 1.3]{Fournais_et_al_2002}), and are therefore elements of $X$. 

\subsection{Set of admissible AIM decompositions}

The goal is to decompose $\rho$ as a sum of translated nonnegative densities $\rho_a^{\rm opt}\in X$, i.e.
\begin{equation}\label{eq:AIM}
\rho(\br) = \sum_{a=1}^M \rho_a^{\rm opt}(\br-\bR_a),
\end{equation}
each $\rho_a^{\rm opt}$ being localized around the origin and optimized in some sense. In quantum chemistry, the $\bR_a$'s are most often the positions of the $M$ nuclei of a molecular system containing $N \in \N^*$ electrons, and $\rho$ an approximation of its ground-state (or $k$-th excited-state) electronic density obtained by a given electronic structure calculation method in a given basis set. We will then denote by $z_a \in \N^*$ the charge of nucleus $a$ for $1\leq a \leq M$. However, in some applications, it can be useful to choose expansion centers away from the nuclei, for instance on a chemical bond, or at the center of mass of a functional group. We will therefore call the points $(\bR_a)_{1\leq a \leq M}$ the expansion centers (rather than the atomic positions); for convenience, we will however use the standard terminology {\it atoms-in-molecules} (AIM) to refer to the family of functions $(\rho_a^{\rm opt})_{1\leq a \leq M}$.

For later purposes, we introduce the atomic charge linear map $\bcN: X^M \to \R^M$ defined by
\begin{equation}\label{eq:atomic_charge_map}
\forall \bm\rho:=(\rho_a)_{1 \le a \le M} \in X^M, \quad [\bcN(\bm\rho)]_a = \int_{\R^3} \rho_a(\br) \, d\br, \; \forall 1\leq a\leq M.
\end{equation}

\begin{remark} We use the decomposition \eqref{eq:AIM} instead of the (equivalent) usual decomposition $\rho(\br) = \displaystyle \sum_{a=1}^M \rho_a^{\rm opt}(\br)$ because the former is more convenient for the analysis of the dependence of the AIM decomposition on the atomic positions. For instance, it is expected that the $\rho_a^{\rm opt}$ (centered at the origin) in \eqref{eq:AIM} converge to the ground (or excited) state density of the isolated atom or ion in the dissociation limit.
\end{remark}

\medskip

The $\rho_a^{\rm opt}$'s can be used as such, or transformed into a collection of descriptors, typically their charges, and their first (dipolar) and second-order moments:
$$
q_a := \int_{\R^3} \rho_a^{\rm opt}(\br) \, d\br, \quad \bp_a := \int_{\R^3} \br \rho_a^{\rm opt}(\br) \, d\br, \quad \bQ_a :=\int_{\R^3} \br \otimes \br \, \rho_a^{\rm opt}(\br) \, d\br.
$$

\medskip
We denote by $X_+:= \{f \in X \; | \; f \ge 0 \; \mbox{a.e.} \}$ the convex cone of bounded integrable nonnegative densities vanishing at infinity and introduce
\begin{equation}\label{eq:KrhoR}
\boxed{\bcK_{\rho,\bR}:=\left\{ \bm \rho=(\rho_a)_{1 \le a \le M} \in X_+^M \; \middle| \; \sum_{a=1}^M \rho_a(\,\cdot\,-\bR_a) = \rho \right\}}, 
\end{equation}
the set of admissible AIM decompositions of the density $\rho$ with respect to the expansion centers~$\bR=(\bR_a)_{1 \le a \le M} \in (\R^3)^M$.

The set $\bcK_{\rho,\bR}$ is non-empty and convex. It has also some nice topological properties which will allow us to prove the existence of an optimal decomposition for the usual AIM decomposition methods (see the mathematical results in Section~\ref{sec:maths} and their proofs in Section~\ref{sec:proofs}).

\subsection{Main ingredients of an AIM decomposition method}

Let us denote by $X^r \subset X$ (respectively $X^r_+ \subset X_+$) the subset of functions of $X$ (respectively $X_+$) that are radially symmetric. 

An objective-functional-based AIM decomposition method consists in seeking an optimal decomposition $\bm\rho^{\rm opt} :=(\rho_a^{\rm opt})_{1 \le a \le M} \in \bcK_{\rho,\bR}$, for a specific criterion. It is characterized by two main ingredients: 
\begin{itemize}
\item [(i)] a collection of sets of (radially symmetric) proatom densities $(\cK^0_z)_{z\in \mathbb{N}^*}$, so that for each $z\in \mathbb{N}^*$, $\cK^0_z \subset X^r_+$ contains reasonable approximations of the ground state density of an isolated atom with nuclear charge $z$. The set of proatom densities is then used to build a set of admissible proatoms-in-molecule (or promolecule) densities
    \begin{equation}\label{eq:def_K0}
    \boxed{\bcK^0:= \cK^0_{z_1} \times \cdots \times \cK^0_{z_M} = \left\{ \bm\rho^0:=(\rho_a^0)_{1 \le a \le M}, \; \rho_a^0 \in \cK_{z_a}^0 \right\} \subset (X^r_+)^M}
    \end{equation}
\item[(ii)] an objective functional $\mathcal E: X_+^M \times (X_+^r)^M \to \mathbb{R} \cup \{+\infty\}$, usually in the form of a relative entropy. 
\end{itemize}

\medskip

The set $\cK^0_z$ may, or not, depend explicitly on $z$. It may have the structure of a vector space but does not need to in general.  For all $1\leq a \leq M$, it is expected that, $\rho_a^{\rm opt}$ should be close in some sense to one of the elements of $\cK_{z_a}^0$, but let us emphasize that, in general, $\rho_a^{\rm opt}$ will not be a radially symmetric function, hence will not belong to $\cK^0_{z_a}$. 

\medskip

The quantity $\mathcal E(\bm\rho, \bm\rho^0)$ measures the discrepancy between a trial AIM decomposition $\bm\rho:=(\rho_a)_{1 \le a \le M}\in X_+^M$ and a trial promolecule density $\bm\rho^0:=(\rho_a^0)_{1 \le a \le M}\in \bcK^0$. Most objective-functional-based AIM decomposition methods (including ISA, Hirschfeld, Hirschfeld-I, L-ISA, NL-ISA) consist in finding $\left( \bm\rho^{\rm opt}, \bm\rho^{0, \rm opt}\right) \in \bcD_{\rho,\bR}$ solution to the minimization problem
\begin{equation}\label{eq:general}
\boxed{\left( \bm\rho^{\rm opt}, \bm\rho^{0, \rm opt}\right) \in \mathop{\rm argmin}_{\left( \bm\rho, \bm\rho^{0}\right) \in\bcD_{\rho,\bR} } \mathcal E(\bm\rho, \bm\rho^0)}
\end{equation}
where $\bcD_{\rho,\bR}$ is a well-chosen subset of $\bcK_{\rho,\bR} \times \bcK^0$.

\medskip

In Sections~\ref{sec:proatom} and~\ref{sec:objective}, we detail various possible choices of sets of proatom densities $\cK^0_{z}$ and objective functionals which are encountered in the literature. The aim of Section~\ref{sec:KL} is to present in more details the methods based on the Kullback-Leibler divergence, including Hirschfeld, Hirschfeld-I, ISA, L-ISA and NL-ISA. The GISA and MB-ISA methods cannot be expressed under the form (\ref{eq:general}); their mathematical structures are presented in Sections~\ref{sec:GISA} and~\ref{sec:MB-ISA} respectively.

\subsection{Proatom densities}\label{sec:proatom}
As anticipated above, we now discuss several choices for the proatom densities $\cK^0_z$.
\\

\noindent\boxed{1} 
In the original \emph{Hirschfeld method}~\cite{Hirshfeld1977}, the set $\cK^0_{z}$  is a singleton and only contains the ground state density of the neutral atom
$$
\cK^0_{z,{\rm Hirschfeld}} = \left\{ \rho_{z,z}^0 \right\},
$$
where $\rho_{z,n}^0$ denotes (an approximation of) the radially symmetric ground state density of the atomic system consisting of a single nucleus of charge $z$ and $n$ electrons\footnote{If the ground state is degenerate and the radial symmetry broken, $\rho_{z,n}^{0}$ is chosen equal to the radially symmetric mixed-state ground-state density obtained by averaging the pure-state ground-state densities (with respect to the Haar measure of the rotation group SO(3)).}, being understood that 
$\rho_{z,n}^0=\rho_{z,N_z^{\rm max}}^0$ for $n$ greater than $N_z^{\rm max}$, the maximum number of electrons that a nucleus of charge $z$ can bind. However, choosing for $\rho_a^0$ electronic densities of {\em neutral} atoms does not seem appropriate for strongly polarized molecules where significant charge transfers occur. This and other drawbacks are detailed in~\cite{bultinck2007critical}, such as the fact that the Hirshfeld method is limited to neutral molecules, and that Hirshfeld atomic charges are on average too small to accurately describe molecular polarization. A desirable feature of an AIM method is its ability to detect charge transfer in an automatic way, without requiring {\it a priori} knowledge from the user on the chemical system under consideration. \\

\noindent\boxed{2} 
In the so-called \emph{Hirschfeld-I method} \cite{bultinck2007critical} (where \emph{I} stands for iterative), this is achieved to some point by the following choice of proatom densities which contains not only the neutral ground-state atomic density, but also the ground-state densities of the ionized forms of the atom, as well as those of the mixed states corresponding to fractional (in the sense of non-integer) numbers of electrons:
$$
\cK^0_{z,{\rm Hirschfeld-I}} = \left\{ \rho_{z,n}^0, \, n \in \R_+ \right\}
$$
with
$$
\forall n \in \R_+, \quad \rho_{z,n}^0(\br) :=  \left( \lceil n \rceil - n\right) \rho^0_{z,\lfloor n \rfloor}(\br) + \left( n - \lfloor n \rfloor\right) \rho_{z,\lceil n \rceil}^0(\br) ,
$$
where for $x \in \R$, $\lfloor x \rfloor$ and $\lceil x \rceil$ respectively denote the largest integer smaller than $x$, and the smallest integer larger than $x$. \\

\noindent\boxed{3} In the \emph{iterative stockholder approach (ISA)}~\cite{lillestolen2008redefining,Lillestolen2009}, this set is independent of the chemical element and is chosen as the whole convex cone of bounded integrable nonnegative radially symmetric functions vanishing at infinity:
$$
 \cK^0_{\rm ISA} = X_+^r.
$$\\

\noindent\boxed{4} In the \emph{finite-dimensional linear approximations of ISA (L-ISA)}, the set $\cK^0_{z}$ is chosen to be a non-empty closed convex subset of $X_+^r$. More precisely, 
\begin{equation}\label{eq:L-ISA}
\cK^0_{z,\rm L-ISA} = \left\{ \rho^0(\br) =\sum_{k=1}^{m_{z}} c_{k} \, g_{z,k}(\br), \; c_{k} \in \R_+ \right\}, 
\end{equation}
where the $g_{z,k}$'s, $1 \le k \le m_{z}$ are $z$-dependent given linearly-independent, positive, $L^1$-normalized functions of $\cK^0_{\rm ISA} = X_+^r$.\\

\noindent\boxed{5} In the \emph{finite-dimensional nonlinear approximations of ISA (NL-ISA)},
\begin{equation}\label{eq:NL-ISA}
\cK^0_{z}=\cK^0_{z,\rm NL-ISA} = \left\{ \rho^0(\br) =\sum_{k=1}^{m_{z}} c_{k} \, g_{z,k,\alpha_{k}}(\br), \; c_{k} \in \R_+, \;  \alpha_{k} >0 \right\}, 
\end{equation}
where the $g_{z,k,\alpha_{k}}$'s are $L^1$-normalized positive functions of $\cK^0_{\rm ISA} = X_+^r$ depending on a parameter $\alpha_{k}$ in a non-affine manner and which needs to be optimized. For the sake of simplicity, here, we assume that the parameters $\alpha_{k}$ are positive real numbers, but $\alpha_{k}$ could be a vector subjected to equality and/or inequality constraints.\\

\noindent\boxed{6} In \emph{GISA} \cite{verstraelen2012conformational}, the set $\cK^0_z$ is defined as the finite-dimensional convex cone
\begin{equation}\label{eq:GISA_set}
\cK^0_{z,{\rm GISA}} = \left\{ \rho_z^0(\br) = \sum_{k=1}^{m_z} c_{k} \, \zeta_{\alpha_{z,k}}(\br), \; c_{k} \in \R_+ \right\} \quad \mbox{where} \quad \zeta_{\alpha}(\br) = \left( \frac{\alpha}{\pi} \right)^{3/2} e^{-\alpha|\br|^2},
\end{equation}
 and where $m_z \in \N^*$ and $\alpha_{z,k} > 0$ are fixed empirical parameters.

\subsection{Objective functionals}\label{sec:objective}
 As mentioned above, the aim of the quantity $\mathcal E(\bm\rho, \bm\rho^0)$ is to measure the discrepancy (in some sense) between a trial AIM decomposition $\bm\rho:=(\rho_a)_{1 \le a \le M}\in X_+^M$ and a trial promolecule density $\bm\rho^0:=(\rho_a^0)_{1 \le a \le M}\in \bcK^0$.

Most of the popular information-theory-based AIM methods (Hirschfeld, Hirschfeld-I, ISA, L-ISA, NL-ISA -- see below for further details) make use of the relative entropy 
\begin{equation}\label{eq:EKL}
\mathcal E(\bm\rho, \bm\rho^0) := S(\bm\rho|\bm\rho^0),
\end{equation}
where
\begin{equation}\label{eq:def_S}
S(\bm\rho|\bm\rho^0):= \sum_{a=1}^M \int_{\R^3} \rho_a(\br) \log \left( \frac{\rho_a(\br)}{\rho_a^0(\br)} \right) \, d\br =  \sum_{a=1}^M s_{KL}(\rho_a|\rho_a^0),
\end{equation}
constructed from the Kullback-Leibler divergence commonly used in information theory:
\begin{equation}\label{eq:KL}
\forall f,g \in X_+, \quad s_{\rm KL}(f|g):=\int_{\R^3} f(\br) \log \left( \frac{f(\br)}{g(\br)} \right) \, d\br.
\end{equation}
In order to make the above definition for $L^1$ functions consistent with the general definition of the Kullback-Leibler divergence for bounded positive measures, the following conventions must be used:
\begin{equation}\label{eq:convention_entropy}
0 \log \left(\frac{0}{0}\right)=0 \quad \mbox{and for all } p > 0, \quad 0 \log \left(\frac{0}{p}\right)=0, \quad p \log \left(\frac{p}{0}\right)=+\infty.
\end{equation}
Note however that alternative methods have been proposed \cite{heidar2014deriving}, where the objective functional $\mathcal E(\bm\rho, \bm\rho^0)$ is chosen as 
$$
\mathcal E(\bm\rho, \bm\rho^0) = \sum_{a=1}^M s(\rho_a|\rho_a^0),
$$
where $s(\cdot|\cdot)$ is either the Hellinger distance 
\begin{equation}\label{eq:Hel}
s_{\rm Hel}(f|g):=\int_{\R^3} \left( \sqrt{f(\br)}- \sqrt{g(\br)}\right)^2 \, d\br = \int_{\R^3} f(\br)  \left( \sqrt{\frac{g(\br)}{f(\br)}}-1 \right)^2,
\end{equation}
or a more general divergence of the form
\begin{equation}\label{eq:relat_ent}
s_{\phi}(f|g):= \int_{\R^3} f(\br)  \phi\left(\frac{g(\br)}{f(\br)}\right),
\end{equation}
where $\phi:\R_+ \to \R \cap \{+\infty\}$ is a strictly convex function, smooth on $\R_+^*$ and such that $\phi(1)=0$.

\medskip

The GISA method \cite{verstraelen2012conformational} makes use of both the relative entropy $S(\bm\rho|\bm\rho^0)$ and the $L^2$-distance
$$
d_{L^2}(\bm\rho,\bm\rho^0):= \left( \sum_{a=1}^M \| \rho_a-\rho_a^0\|_{L^2}^2\right)^{1/2},
$$
in a more complex fashion as will be summarized below in Section~\ref{sec:GISA}.

Lastly, the Minimal Basis Iterative Stockholder algorithm (MB-ISA) \cite{verstraelen2016minimal} is in fact not based on a direct atomic decomposition on $\rho$. It rather provides an {\em atomic-shell} decomposition of $\rho$ on {\em proatomic-shell} densities, from which AIM densities can be derived. As will be seen in Section~\ref{sec:MB-ISA}, it nevertheless nicely fits in the unified mathematical framework described in this article.

\subsection{Kullback-Leibler entropy methods}\label{sec:KL}

In this section, we more specifically focus on the AIM decomposition methods the objective functional of which is based on the Kullback-Leibler divergence  $S(\cdot|\cdot)$. Let us first recall the well-known formula~\eqref{eq:G_function} below, a rigorous proof of which is provided in Section~\ref{sec:proofs}.

\begin{lemma}\label{lem:Step1} Let $\bm \rho^0 := (\rho^0_a)_{1\leq a\leq M}\in (X_+^r)^M$ and $\rho \in X_+$ be such that
\begin{equation}\label{eq:hyp_G0}
\quad s_{\rm KL}(\rho|\rho^0) < +\infty \quad \mbox{where} \quad \rho^0(\br) := \sum_{a=1}^M \rho^0_a(\br-\bR_a).
\end{equation}
Then, the constrained minimization problem 
\[
\inf_{\bm \rho\in \bcK_{\rho, \bR}} S( \bm \rho| \bm \rho^0)
\]
admits an unique minimizer 
$\bG_{\rho,\bR}({\bm\rho}^0) := ([\bG_{\rho,\bR}({\bm\rho}^0)]_a)_{1\leq a \leq M}$, given by
\begin{equation} \label{eq:G_function}
 \forall 1\leq a \leq M, \quad    [\bG_{\rho,\bR}({\bm\rho}^0)]_a(\bm r) = \left| \begin{array}{ll} \dps
 \frac{\rho_a^0(\bm r)}{\rho^0(\bm r + \bR_a)}\rho(\bm r + \bR_a) & \quad \mbox{if } \rho_a^0(\bm r) > 0, \\
 0 & \quad \mbox{if } \rho_a^0(\bm r) =0, \end{array} \right.
\end{equation}
and it holds that
\begin{equation}\label{eq:min_on_KrhoR}
    \min_{\bm \rho\in \bcK_{\rho, \bR}} S( \bm \rho| \bm \rho^0) 
    = s_{\rm KL}(\rho|\rho^0).
\end{equation}

\end{lemma}
Observe that if $\rho_a^0(\bm r) > 0$, then  $\rho^0(\bm r + \bR_a) \ge\rho_a^0(\bm r)   >0$, while if $\rho_a^0(\bm r) =0$, then we can have either $\rho^0(\bm r + \bR_a) > \rho_a^0(\bm r)$, in which case $\frac{\rho_a^0(\bm r)}{\rho^0(\bm r + \bR_a)}\rho(\bm r + \bR_a)$ is well-defined and equal to $0$, or $\rho^0(\bm r + \bR_a) = \rho_a^0(\bm r)=0$, in which case $\frac{\rho_a^0(\bm r)}{\rho^0(\bm r + \bR_a)}\rho(\bm r + \bR_a)$ is not well-defined.

\begin{remark} \label{rem:product_space}
In the case when $\mathcal E( \bm \rho, \bm \rho^0)=S( \bm \rho| \bm \rho^0)$ and $\bcD_{\rho,\bR}=\bcK_{\rho,\bR} \times \bcK^0$, we infer from Lemma~\ref{lem:Step1} that the minimization problem in \eqref{eq:general} is equivalent to
\begin{equation}\label{eq:caseH}
\inf_{\bm \rho^0 := (\rho^0_a)_{1\leq a\leq M} \in \bcK^0} s_{\rm KL}\left( \, \rho \;  \bigg| \sum_{a=1}^M \rho_a^0(\br - \bR_a) \right).
\end{equation}
Although this problem seems natural, it is not the one solved in most AIM methods, probably because it is computationally challenging in general. Only the basic Hirschfeld method fits into this framework, as will be seen in the next section, but in this case, \eqref{eq:caseH} is trivial since $\bcK^0_{\rm Hirschfeld} =\{{\bm\rho}^0_{\rm Hirschfeld}\}$ contains only one element.
\end{remark}

\subsubsection{The Hirschfeld method}\label{sec:Hirschfeld}

The Hirshfeld method consists in solving 
\begin{equation} \label{eq:opt_AIM0}
\inf_{\bm\rho \in \bcK_{\rho,\bR}} S(\bm\rho | {\bm\rho}^0_{\rm Hirschfeld}),
\end{equation}
where ${\bm\rho}^0_{\rm Hirschfeld}=(\rho_{z_a,z_a}^0)_{1 \le a \le M}$ is given. As already mentioned in Remark~\ref{rem:product_space}, this problem fits into the general framework (\ref{eq:general}) using the objective functional $\mathcal E$ defined in (\ref{eq:EKL}) and $\bcD_{\rho, \bR} = \bcK_{\rho, \bR} \times \bcK^0_{\rm Hirschfeld}$.

\medskip

The atomic densities $\rho^0_{z_a,z_a}$ are positive everywhere and decay as exponentials at infinity; in addition, exact molecular densities and their usual approximations decay (at least) exponentially fast at infinity. The assumptions in Lemma~\ref{lem:Step1} are thus satisfied, which implies that the basic Hirschfeld problem \eqref{eq:opt_AIM0} is well-posed and has an explicit solution:
\begin{equation}\label{eq:sol_Hirschfeld}
\rho_a^{\rm opt,Hirschfeld}(\br)
=[\bG_{\rho,\bR}({\bm\rho}^0_{\rm Hirschfeld})]_a(\br)= \frac{\rho^0_{z_a,z_a}(\br)}{\displaystyle \sum_{b=1}^N \rho_{z_b,z_b}^0(\br-\bR_b+\bR_a)} \rho(\bR_a+\br).
\end{equation}

\subsubsection{AIM iterative methods based on the Kullback-Leibler divergence}
\label{subsec:AIM_methods_based_KL_divergence}

The ISA, Hirschfeld-I, L-ISA and NL-ISA methods can be formulated using the following unified formalism: the optimal AIM decomposition is obtained by solving 
\begin{equation} \label{eq:opt_AIM}
\boxed{\inf_{(\bm\rho,\bm\rho^0) \in \bcC_{\rho,\bR}} S(\bm\rho | {\bm\rho}^0)}
\end{equation}
where
$$
\boxed{\bcC_{\rho,\bR}= \big\{  (\bm\rho,\bm\rho^0) \in \bcK_{\rho,\bR} \times  \bcK^0 \; | \; \bcN(\bm\rho)=\bcN(\bm\rho^0) \big\}}
$$
with $\bcN$ being the atomic charge map defined in \eqref{eq:atomic_charge_map}.
Note that the variation that determines the different methods (ISA, Hirschfeld-I, L-ISA, NL-ISA) lies in the definition of $\bcK^0$ defined from the $\cK^0_z$'s given in Section~\ref{sec:proatom}, i.e. $\cK^0_z=\cK^0_{z,{\rm Hirschfeld-I}}$, $\cK^0_z=\cK^0_{\rm ISA}$, $\cK^0_z=\cK^0_{z,\rm L-ISA}$ or $\cK^0_z=\cK^0_{z, \rm NL-ISA}$.
These methods hence fit into the general framework (\ref{eq:general}) using the objective functional $\mathcal E$ defined in (\ref{eq:EKL}) and $\bcD_{\rho, \bR} = \bcC_{\rho,\bR}$ with corresponding $\bcK^0$. The solution to \eqref{eq:opt_AIM} is not explicit and must be computed numerically, by an iterative algorithm. 
The constraints $\bcN(\bm\rho)=\bcN(\bm\rho^0)$ request that for all $1\leq a\leq M$, $\rho_a$ and $\rho_a^0$ have the same charge, i.e. $\int_{\R^3}\rho_a(\br) \, d\br = \int_{\R^3}\rho_a^0(\br) \, d\br$, which is not the case for the optimal AIM obtained with the basic (non-iterative) Hirschfeld method. 

\medskip

\begin{remark} \label{rem:assumptions_ISA} For this approach to make sense, the least one can ask is that 
\begin{equation}\label{eq:cond_inf_finite}
\mbox{there exists $(\bm\rho,\bm\rho^0) \in \bcC_{\rho,\bR}$ such that $S(\bm\rho|\bm\rho^0) < +\infty$.}
\end{equation}
This condition guarantees that the infimum in \eqref{eq:opt_AIM} is finite. Note that it is not satisfied, even for physical ground-state electronic densities $\rho$, if e.g. the $\cK^0_z$'s only contain compactly supported functions. For \eqref{eq:cond_inf_finite} to be satisfied, the sets $\cK^0_z$ must contain functions which do not decay extremely fast at infinity. This is always the case for ISA for which the sets $\cK^0_{\rm ISA}=X_+^r$'  contains the function $e^{- |\br|}$, so that taking $\rho_a(\br) = M^{-1} \rho(\br+\bR_a)$ and $\rho^0_a(\br)=\frac{N}{8\pi M} e^{- |\br|}$, we have
$$
S(\bm\rho|\bm\rho^0) = \int_{\R^3} \rho \log\rho -  \left(\frac NM \log M - N \log \frac{N}{8\pi M}  \right) +  \frac 1M \sum_{a=1}^M \int_{\R^3} |\br| \rho(\br+\bR_a) \, d\br,
$$
which is a finite quantity for all $\rho \in X_+$. In the Hirschfeld-I, L-ISA, or NL-ISA settings, the sets $\cK^0_{z_a}$ must contain functions decaying asymptotically not faster than $e^{-\alpha |\br|}$ (for some $\alpha > 0$) for \eqref{eq:cond_inf_finite} to be satisfied for all $\rho \in X_+$. The requirements on the sets $\cK^0_{z_a}$ can be weakened by imposing more conditions on $\rho$. For instance, if we have $\int_{\R^3} |\br|^2 \rho(\br) \, d\br < \infty$, then condition \eqref{eq:cond_inf_finite} is satisfied as soon as the sets $\cK^0_{z_a}$ contain functions decaying asymptotically not faster than a Gaussian.
 \end{remark}

\medskip

Most of the iterative numerical schemes for solving \eqref{eq:opt_AIM} that have been proposed in the literature amounts to minimizing $S(\bm\rho|\bm\rho^0)$ alternatively with respect to each of the variables $\bm\rho^0$ and $\bm\rho$, under suitable constraints. They read as follows. 

\medskip

\noindent
\mybox{
{\bf Generic iterative AIM algorithm}: 
\begin{itemize}
\item {\bf Initialization}:  
Choose $\bm \rho^{0,(0)} \in \bm\cK^0$ such that $s_{\rm KL}(\rho|\rho^{0,(0)}) < +\infty$.
\item \bfseries Iteration $m\geq 0$: \normalfont 
\begin{description}[leftmargin=1cm,labelwidth=0.5cm,labelsep=0.5cm,itemindent=1cm]
\item[{\bf{Step 1}:}] 
set
\begin{equation}
    \label{eq:step1}
    \bm\rho^{(m)} = \bG_{\rho,\bR}(\bm\rho^{0,(m-1)})  = \argmin_{\bm\rho \in \bcK_{\rho,\bR}} S(\bm\rho|\bm\rho^{0,(m-1)}),
\end{equation}
\item[{\bf{Step 2}:}] 
find 
\begin{equation}
    \label{eq:step2} 
    \bm\rho^{0,(m)} \in \argmin_{\bm\rho^0 \in \bcK^0, \, \bcN(\bm\rho^0)=\bcN(\bm\rho^{(m)})} S(\bm\rho^{(m)}|\bm\rho^0).
\end{equation}
\end{description}
\end{itemize}
}
Recall that the solution to \eqref{eq:step1} in Step 1 is explicit, since the expression of the function $\bG_{\rho,\bR}(\bm\rho^0)$ is given by~\eqref{eq:G_function}. 

\medskip

Due to the particular nature of $S(\cdot|\cdot)$, the solution to Step~2 can then be obtained by solving $M$ independent and local problems of the form
\begin{equation}
    \label{eq:comp_G}
    \inf_{\rho^0_a \in \cK^0_{z_a}, \; \int_{\R^3} \rho^0_a=\int_{\R^3} \rho_a} s_{\rm KL}(\rho_a|\rho_a^0)
\end{equation}
for $\rho_a=\rho_a^{(m-1)}$.
\\

At the noticeable exception of the ISA method, solving the optimization problem in Step~2 for the various AIM methods presented above require evaluating the integrals 
$$
N_a^{(m)}:=[\bcN(\bm\rho^{(m)})]_a=\int_{\R^3} \rho_a^{(m)}(\br) \, d\br, 
$$
with $\bm\rho^{(m)}=\bG_{\rho,\bR}(\bm\rho^{0,(m-1)})$. Step 2 therefore requires an efficient numerical quadrature scheme to evaluate integrals of the form
\begin{equation} \label{eq:integral_atomic_charge}
Q_{a,\rho,\bR}(\bm\rho^0)
:=
[\bcN(\bG_{\rho,\bR}(\bm\rho^{0}))]_a
=
\int_{\R^3} \frac{\rho^0_a(\br)}{\sum_{b=1}^N \rho_b^0(\br-\bR_b+\bR_a)} \rho(\br+\bR_a) \, d\br,
\end{equation}
(with the convention that the integrand is equal to zero when the denominator vanishes),
for $\bm\rho^0=(\rho_a^0)_{1 \le a \le M}$ in the set $\bcK^0$ of promolecule densities. 
The elements of the sets $\bcK^0_{\rm GISA}$, $\bcK^0_{\rm Hirschfeld-I}$ or $\bcK^0_{\rm MB-ISA}$ of promolecule densities constructed from the GISA, Hirschfeld-I, and MB-ISA proatom densities respectively, are smooth, fast-decaying, radially symmetric functions. The integrals of the form~\eqref{eq:integral_atomic_charge} can therefore be efficiently computed using a one-dimensional quadrature scheme for the radial part and a Lebedev quadrature scheme for the angular part.

\medskip

We successively present in details in the following sections how problems of the form (\ref{eq:comp_G}) are solved for four different methods:
the Hirschfeld-I method, the general ISA method, the finite-dimensional linear approximations of the ISA (L-ISA) method,  the finite-dimensional nonlinear approximations of the ISA (NL-ISA) method.

\medskip

Note that the abstract versions L-ISA and NL-ISA haven not been reported in the literature, although they are similar, but not identical, to the GISA and MB-ISA methods respectively. For L-ISA however, we can establish rigorous results (see Section~\ref{sec:maths}) such as existence and uniqueness of the solution as well as global convergence of the AIM iteration scheme.

\subsubsection{Hirschfeld-I iterations}

For the Hirschfeld-I method \cite{bultinck2007critical}, solving \eqref{eq:comp_G} is trivial since the set $\{\rho^0_a \in \cK^0_{z_a,{\rm Hirschfeld-I}} \, | \,  \int_{\R^3} \rho^0_a=N_a\}$ contains a single element, namely $\rho^0_{z_a,N_a}$.
At the continuous level (i.e., without any discretization of the proatom densities), the Hirschfeld-I iterative procedure corresponding to the generic iterative AIM algorithm for 
$$
\bcK^0=\bcK^0_{\rm Hirschfeld-I}=\cK^0_{z_1,\rm Hirschfeld-I} \times \cdots \times \cK^0_{z_M,\rm Hirschfeld-I}
$$
can therefore be formulated more explicitly as follows:

\medskip

\noindent
\mybox{
\bfseries Hirschfeld-I algorithm: \normalfont

\begin{itemize}
\item {\bf Initialization:} 
choose an initial guess of charges $\bN^{(0)}=(N_a^{(0)})_{1 \le a \le M}\in (\R_+)^M$ so that $\displaystyle \sum_{a=1}^N N_a^{(0)}=N$.
\item \bfseries Iteration $m\geq 0$:
\normalfont 
\begin{description}[leftmargin=1cm,labelwidth=0.5cm,labelsep=0.5cm,itemindent=1cm]
\item 
set
\begin{align}
  &\rho_{a}^{(m)}(\br):=\frac{\rho^0_{z_a,N_a^{(m)}}(\br)}{\displaystyle \sum_{b=1}^M \rho^0_{z_b,N_b^{(m)}}(\br-\bR_b+\bR_a)} \rho(\bR_a+\br), \label{eq:HF1}
\end{align}
where
$$
\rho_{z_a,N_a^{(m)}}^0 = \left(\lceil N_a^{(m)} \rceil - N_a^{(m)}\right) \rho_{z_a,\lfloor N_a^{(m)} \rfloor}^0 + \left(N_a^{(m)}-\lfloor N_a^{(m)} \rfloor\right) \rho_{z_a,\lceil N_a^{(m)} \rceil}^0;
$$
\item 
compute  
\begin{equation}
    N_a^{(m+1)}:= \int_{\R^3}  \rho_{a}^{(m)}(\br) \, d\br.  \label{eq:comp_Na} 
\end{equation}
\end{description}
\end{itemize}
}
The atomic charges $N_a^{(m)}$ at iteration $m$ are fractional (i.e. non-integer) in general. The global convergence property of this iterative scheme was observed numerically in \cite{bultinck2007critical} on a benchmark of 168 molecules. A convergence proof was proposed in~\cite{bultinck2007uniqueness}.

\medskip

From a numerical point of view, the Hirschfeld-I method can be seen as a fixed-point procedure of the form $\bN^{(m+1)}=F_{\rm Hirschfeld-I}(\bN^{(m)})$, where $F_{\rm Hirschfeld-I}: \R_+^M \to \R_+^M$ is the function defined by \eqref{eq:HF1}-\eqref{eq:comp_Na}. Using $\bN=(N_a)_{1 \le a \le M} \in \R_+^M$ as main variables, it is not necessary to store in memory the functions $\rho_a^{(m)}$: they can be evaluated using \eqref{eq:HF1} from $\rho$, $\bN^{(m-1)}$, and the $\rho_{z,n}^{0}$'s. The implementation of the Hirschfeld-I method thus requires:
\begin{enumerate}
\item the pre-computation and storage of the radially symmetric functions $\rho_{z,n}^{0}$, for all chemical elements contained in the molecular system of interest, and integer values $n \le N_{z}^{\rm max}$. Note that these functions can be pre-computed once and for all with high accuracy and then be used for any molecular system;
\item repeated evaluations of integrals \eqref{eq:integral_atomic_charge}, where the $\rho_a^0$'s are of the form $\rho^0_{z_a,N_a}$. 
\end{enumerate}

\subsubsection{ISA iterations}
\label{sec:ISA_iterations}

The ISA method was originally introduced in~\cite{lillestolen2008redefining,Lillestolen2009} as the iterative scheme described below, which was then interpreted in \cite{heidar2017information} as a specific instance of the generic iterative AIM algorithm for 
$$
\bcK^0=\bcK^0_{\rm ISA}:=\left\{ \bm\rho^0=(\rho^0_a)_{1 \le a \le M}  \in (X_+^{r})^M \right\}.
$$ 
For $\cK^0_{z_a}=\cK^0_{\rm ISA} = X_+^r$ and Step 2, the solution to problem \eqref{eq:comp_G}, when it exists, has a particularly simple form.
Indeed, denoting by 
\begin{align}
\langle\rho_{a}\rangle_s(\br) :&= \fint_{\S^2} \rho_a(|\br|\bm\sigma) \, d\bm\sigma  \label{eq:rhoas} \\
&= \frac{1}{4\pi} \int_0^\pi \sin\theta \,  \int_0^{2\pi}  \, \rho_a(r \sin\theta\cos\phi \, \be_x + r \sin\theta\sin\phi \, \be_y + r \cos\theta \, \be_z) \, d\phi \,d\theta, \nonumber
\end{align}
with $\fint_{\S^2} := \frac{1}{4\pi} \int_{\S^2}$, the spherical average of $\rho_a$, we have the following result.

\begin{lemma}\label{lem:Step2_ISA} Let $\rho_a \in X_+$. Then, 
the constrained optimization problem
\begin{equation}\label{eq:pb_Lemma_1}
\mathop{\inf}_{\rho^0_a \in \cK^0_{\rm ISA}, \; \int_{\R^3} \rho^0_a = \int_{\R^3} \rho_a} s_{\rm KL}(\rho_a |\rho_a^0)
\end{equation}
is well-posed and its solution is $\langle \rho_a\rangle_s$.
\end{lemma}

It is convenient to formulate the ISA method in terms of the radial functions
\begin{equation}\label{eq:def_wa}
w_a(r):=\fint_{\S^2} \rho_a(r \bm\sigma) \, d\bm\sigma,
\end{equation}
which belong to the closed convex cone
$$
Y_{+}:=\{ w \in L^1(\R_+;(r^2+r^3) \, dr) \cap L^\infty(\R_+) \; | \; w \ge 0 \mbox{ a.e.}, \; \lim_{r \to +\infty} w(r) = 0  \}
$$
of the Banach space $Y:= L^1(\R_+;(r^2+r^3) \, dr) \cap L^\infty(\R_+)$, and are such that
$$
\forall r > 0, \; \forall \bm\sigma \in \S^2, \quad \langle\rho_{a}\rangle_s(r \bm\sigma)=w_a(r).
$$
The functions $w_a$ are the main variables used in the practical implementation of the ISA method. At the continuous level, the ISA algorithm can thus be formulated as follows:

\medskip

\noindent
\mybox{
\bfseries ISA algorithm: \normalfont

\begin{itemize}
 \item {\bf Initialization:} choose an initial guess for $\bw^{(0)}:=(w_a^{(0)})_{1 \le a \le M} \in Y_{+}^M$ such that $\displaystyle \sum_{a=1}^M w_a(|\br-\bR_a|) > 0$ in $\R^3$.
\item \bfseries Iteration $m\geq 1$: \normalfont 
\begin{description}[leftmargin=1cm,labelwidth=0.5cm,labelsep=0.5cm,itemindent=1cm]
\item[{\bf{Step 1}:}] 
set 
\begin{align*}
 &\rho_{a}^{(m)}(\br):=\frac{w_a^{(m-1)}(|\br|)}{\sum_{b=1}^M w_b^{(m-1)}(|\br-\bR_b+\bR_a|)} \rho(\bR_a+\br), 
\end{align*}
\item[{\bf{Step 2}:}]  
compute 
\begin{equation}
    w_a^{(m)}(r) := \fint_{\S^2} \rho_a^{(m)}(r \bm\sigma) \, d\bm\sigma. \label{eq:iter_ISA}
\end{equation}
 \end{description}
\end{itemize}
}
Introducing the linear map $X^M \ni \bm\rho \mapsto \langle\bm\rho\rangle_s \in X^M$ where for all $\bm\rho=(\rho_a)_{1 \le a \le M} \in X^M$, $\langle\bm\rho\rangle_s:=(\langle \rho_{a}\rangle_s)_{1 \le a \le M}$ with $\langle \rho_{a}\rangle_s$ defined by \eqref{eq:rhoas}, the ISA optimization problem can be written as
\begin{equation} \label{eq:ISA_opt}
\min_{\bm\rho \in \bcK_{\rho,\bR}} S(\bm\rho|\langle\bm\rho\rangle_s) \qquad \mbox{(ISA optimization problem)}.
\end{equation}

A practical implementation of ISA can then be obtained by
\begin{enumerate}
\item discretizing the radially symmetric functions $w_a^{(m-1)}$ on a suitable one-dimensional grid;
\item computing the spherical averages \eqref{eq:iter_ISA} for each $a$ and each grid point of the radial grid, using e.g. Lebedev integration method.
\end{enumerate}
More details about the practical aspects of the implementation will be given in Section~\ref{sec:numerical_results}.

\subsubsection{Linear approximation of the ISA (L-ISA) method}
\label{sec:L-ISA}

Instead of working with the three-dimensional, but radially symmetric, functions given in $\cK^0_z=\cK^0_{z, {\rm L-ISA}}$, it is more convenient to work with their generating function defined on $\R_+$. For this purpose, let us denote by $\widetilde g_{z_a,k}:\R_+ \to \R_+$ the function such that $g_{z_a,k}(\br)=\widetilde g_{z_a,k}(|\br|)$ for all $\br \in \R^3$, $1\leq a \leq M$ and $1\leq k \leq m_{z_a}$. 
Let us also denote by 
$$V:={\rm Span}\left\{ \widetilde g_{z_a,k}, \; 1\leq a \leq M, \; 1\leq k \leq m_{z_a} \right\}, \quad V_+:= {\rm Span}_+\left\{ \widetilde g_{z_a,k}, \; 1\leq a \leq M, \; 1\leq k \leq m_{z_a} \right\}$$ (where the notation ${\rm Span}_+$ denotes the set of linear combinations with non-negative coefficients of the $\widetilde g_{z_a,k}$'s). Then, it holds that $V\subset Y$, $V_+ \subset Y_+$ and the L-ISA algorithm can be rewritten as follows:

\medskip

\noindent
\mybox{
\bfseries L-ISA algorithm: \normalfont

\begin{itemize}
 \item {\bf Initialization:} choose an initial guess for $\bv^{(0)}:=(v_a^{(0)})_{1 \le a \le M} \in V_{+}^M$ such that $\displaystyle \sum_{b=1}^M v_b(|\br-\bR_b|) > 0$ in $\R^3$.
\item \bfseries Iteration $m\geq 1$: \normalfont 
\begin{description}[leftmargin=1cm,labelwidth=0.5cm,labelsep=0.5cm,itemindent=1cm]
\item[{\bf{Step 1}:}] 
set 
\begin{align} 
    \label{eq:iter_L-ISA_Hirshfled_formula}
    &\rho_{a}^{(m)}(\br):=\frac{v_a^{(m-1)}(|\br|)}{\displaystyle \sum_{b=1}^M v_b^{(m-1)}(|\br-\bR_b+\bR_a|)} \rho(\bR_a+\br), 
\end{align}
\item[{\bf{Step 2}:}] 
compute 
\begin{equation}
    v_{a}^{(m)}\in \mathop{\rm argmin}_{\begin{array}{c}v_a \in V_+,\\ 4\pi\int_{\mathbb{R}_+}v_a(r)r^2\,dr  = N_a^{(m)}\\
    \end{array}}  
    \int_{\R^3} \rho_a^{(m)}(\br) \log \left(\frac{\rho^{(m)}_a(\br)}{\displaystyle v_a(|\bm r|)} \right) \, d\br, \label{eq:iter_L-ISA}
\end{equation}
with 
$$
    N_a^{(m)} := \int_{\mathbb{R}^3} \rho_a^{(m)}(\bm r)\,d\bm r.
$$
\end{description}
\end{itemize}
}

In this case, the subproblems \eqref{eq:comp_G} (and equivalently subproblems \eqref{eq:iter_L-ISA} with $\rho_a = \rho_a^{(m)}$) to be solved in Step~2 are of the form
\begin{equation}\label{eq:inf-L-ISA}
\mathop{\inf}_{\begin{array}{c}v_a \in V_+,\\ 4\pi\int_{\mathbb{R}_+}v_a(r)r^2\,dr  = N_a\\
          \end{array}}  \int_{\R^3} \rho_a(\br) \log \left(\frac{\rho_a(\br)}{\displaystyle v_a(|\bm r|)} \right) \, d\br,
\end{equation}
with $N_a:= \int_{\mathbb{R}^3} \rho_a$. The existence of minimizers to this problem is discussed in Lemma~\ref{lem:prelimG}. Problem (\ref{eq:inf-L-ISA}) can be rewritten equivalently, since the functions $g_{z_a,k}$ are $L^1$-normalized, as
\begin{equation}\label{eq:G0-L-ISA}
\inf_{\substack{\bm c_a = (c_{a,1}, \cdots, c_{a,m_{z_a}}) \in \R_+^{m_{z_a}} \\ \displaystyle \sum_{k=1}^{m_{z_a}} c_{a,k}=N_a}} \int_{\R^3} \rho_a(\br) \log \left(\frac{\rho_a(\br)}{\displaystyle \sum_{k=1}^{m_{z_a}} c_{a,k} \, g_{z_a,k}(\br)} \right) \, d\br.
\end{equation}

\noindent Note that the minimization problem \eqref{eq:G0-L-ISA} differs from the original GISA method \cite{verstraelen2012conformational} where an L$^2$ norm distance between $v_a=\displaystyle \sum_{k=1}^{m_{z_a}} c_{a,k} \, g_{z_a,k}$ and the spherical average of $\rho_a$ (namely $\rho_a^{(m)}$, at iteration $m$) is minimized, though under the same constraint $ \displaystyle \sum_{k=1}^{m_{z_a}} c_{a,k}=N_a$ (see Section \ref{sec:GISA}).\\

Using the function $w_a$ defined in \eqref{eq:def_wa}, we see that
\begin{align*}
\int_{\R^3} \rho_a(\br) \log \left(\frac{\rho_a(\br)}{\displaystyle \sum_{k=1}^{m_{z_a}} c_{a,k}\, g_{z_a,k}(\br)} \right) \, d\br &= \int_{\R^3} \rho_a(\br) \log \left(\rho_a(\br) \right) \, d\br \\
& \quad -
4 \pi \int_{0}^{+\infty} r^2 w_a(r) \log \left(\sum_{k=1}^{m_{z_a}} c_{a,k} \, \widetilde{g}_{z_a,k}(r) \right) \, dr.
\end{align*}
It follows that the minimizers of \eqref{eq:G0-L-ISA} can be obtained by solving the problem
\begin{equation}
    \label{eq:G0-L-ISA-2}
    \inf_{\substack{\bc_a=(c_{a,1}, \cdots, c_{a,m_{z_a}}) \in \R_+^{m_{z_a}} \\ \displaystyle \sum_{k=1}^{m_{z_a}} c_{a,k}=N_a}} F_a(\bc_a),
\end{equation}
where
\begin{align}
   F_a(\bc_a) :=  - \int_{0}^{+\infty} r^2 w_a(r) \log \left(\displaystyle \sum_{k=1}^{m_{z_a}} c_{a,k} \, \widetilde{g}_{z_a,k}(r) \right) \, dr.
\end{align}
For functions $g_{z_a,k}$ decaying extremely fast at infinity, it may happen that $F_a(\bc_a)=+\infty$ for $w_a$'s corresponding to physically admissible $\rho_a$'s. On the other hand, if the $\widetilde{g}_{z_a,k}$ decay at most exponentially fast (resp. not faster than a Gaussian function) and the first-order moments (resp. the second-order moments) of the function $\rho$ to be decomposed is finite, then the function $F_a$ is strictly convex and continuous on the simplex $\left\{\bc_a=(c_{a,1}, \cdots, c_{a,m_{z_a}}) \in \R_+^{m_{z_a}} \, \left| \, \displaystyle \sum_{k=1}^{m_{z_a}} c_{a,k}=N_a\right.\right\}$ for all $N_a > 0$, and therefore has a unique minimizer. In addition, it is smooth, so that the minimizer can be computed efficiently for small values of $m_{z_a}$ by standard numerical optimization algorithms. The components of the gradient and Hessian of $F_a$ are given respectively for all $1\leq i,j \leq m_{z_a}$ by 
\begin{align*}
\frac{\partial F_a}{\partial c_{a,i}}(\bc_a) &= - \int_{0}^{+\infty} r^2 w_a(r) \frac{ \widetilde{g}_{z_a,i}(r)}{\displaystyle \sum_{k=1}^{m_{z_a}} c_{a,k}  \widetilde{g}_{z_a,k}(r)} \, dr, \\
\frac{\partial^2 F_a}{\partial c_{a,i}\partial c_{a,j}}(\bc_a) &=  \int_{0}^{+\infty} r^2 w_a(r) \frac{ \widetilde{g}_{z_a,i}(r) \widetilde{g}_{z_a,j}(r)}{\left(\displaystyle \sum_{k=1}^{m_{z_a}} c_{a,k} \, \widetilde{g}_{z_a,k}(r)\right)^2} \, dr .
\end{align*}

\subsubsection{Nonlinear approximation of the ISA (NL-ISA)}

Proceeding as in the previous section, the subproblems \eqref{eq:comp_G} to be solved in Step 2 are now of the form
\begin{equation}\label{eq:G0-NLISA}
    \inf_{\substack{ 
    \bc_a \in \R_+^{m_{z_a}}, \, \bm\alpha_a \in (0,+\infty)^{m_{z_a}}
    \\ \displaystyle \sum_{k=1}^{m_{z_a}} c_{a,k}=N_a 
    }}
    F_a(\bc_a,\bm\alpha_a),
\end{equation}
denoting $\bc_a=(c_{a,1}, \cdots, c_{a,m_{z_a}})$, $\bm\alpha_a=(\alpha_{a,1},\cdots,\alpha_{a,m_{z_a}})$ and
where 
$$
F_a(\bc_a,\bm\alpha_a) :=  - \int_{0}^{+\infty} r^2 w_a(r) \log \left(\sum_{k=1}^{m_{z_a}} c_{a,k} \, \widetilde g_{z_a,k,\alpha_{a,k}}(r) \right) \, dr.
$$
Still for the sake of simplicity, we assume that the $\alpha_{a,k}$'s are positive real numbers, but more complicated settings can be considered as well.

\medskip

Not much can be said about the optimization \eqref{eq:G0-NLISA}
without additional assumptions on the functions $\widetilde g_{z_a,k,\alpha_{a,k}}$. It may have no solution for some admissible $w_a$'s, and multiple solutions for other admissible $w_a$'s.

\subsection{GISA iterations}
\label{sec:GISA}

The aim of this section is to describe in details the original GISA method. The GISA iterations can be formalized as follows. 

\medskip

\noindent
\mybox{
\bfseries GISA algorithm: \normalfont
\begin{itemize}
\item \bfseries Initialization: \normalfont Choose $\bm \rho^{0,(0)} \in \bcK^0_{\rm GISA}$. 
\item \bfseries Iteration $m\geq 1$: \normalfont
\begin{description}[leftmargin=1cm,labelwidth=0.5cm,labelsep=0.5cm,itemindent=1cm]
\item[{\bf{Step 1}:}]  Find $\bm \rho^{(m)} \in \bcK_{\rho,\bR}$ solution to
\begin{equation}\label{eq:GISA1}
     \bm\rho^{(m)} =\argmin_{\bm\rho \in \bcK_{\rho,\bR}} S(\bm\rho|\bm\rho^{0,(m-1)}) .
\end{equation}
\item[{\bf{Step 2}:}] Find $\bm\rho^{0,(m)} \in \bcK^0_{\rm GISA}$ solution to 
\begin{equation}\label{eq:GISA2}
         \bm\rho^{0,(m)} = \argmin_{\bm\rho^0 \in \bcK^0_{\rm GISA} \, | \, \bcN(\bm\rho^0)=\bcN(\bm\rho^{(m)})} \|\bm\rho^0-\bm\rho^{(m)}  \|_{L^2}^2 .
    \end{equation}
\end{description}
\end{itemize}
}

The solution to \eqref{eq:GISA1} is explicit, while \eqref{eq:GISA2} can be split into the $M$ independent minimization problems 
\begin{equation}\label{eq:GISA3}
\rho_a^{0,(m)}=\argmin_{\rho_a^0 \in \cK^0_{z_a,\rm GISA} \, | \, \int_{\R^3}\rho^0_a=\int_{\R^3}\rho_a^{(m)}} \|\rho_a^{(0)}-\rho_a^{(m)}\|_{L^2}^2. 
\end{equation}
It is convenient to reformulate the GISA iterations using as main variables the coefficients $\bc=(\bc_a)_{1 \le a \le M}\in \R_+^{p}$ with $\bc_a=(c_{a,k})_{1 \le k \le m_{z_a}} \in \R_+^{m_{z_a}}$ and $p=\sum_{a=1}^M m_{z_a}$  (see Eq. \eqref{eq:GISA_set}). The iterations \eqref{eq:GISA1}-\eqref{eq:GISA2} can be rewritten as
\begin{align}
        \rho_a^{(m)}(\br)&= \frac{\displaystyle \sum_{k=1}^{m_{z_a}}
    c_{a,k}^{(m-1)} \, \zeta_{\alpha_{a,k}}(\br)}{\displaystyle \sum_{b=1}^M \sum_{k=1}^{m_{z_b}}
    c_{b,k}^{(m-1)} \, \zeta_{\alpha_{b,k}}(\br+\bR_a-\bR_b)} \rho(\br+\bR_a),
    \label{eq:GISA4} \\ \nonumber
\end{align}
and
\begin{align}
    \bc_a^{(m)}&= \argmin_{\substack{\bm{c}_{a} \in \R_+^{m_{z_a}}, \displaystyle \dps\sum_{k=1}^{m_{z_a}} c_{a,k} = \int_{\R^3} \rho_a^{(m)}}}
    \left\|  \dps\sum_{k=1}^{m_{z_a}}
    c_{a,k} \, \zeta_{\alpha_{a,k}} - \rho_{a}^{(m)} \right\|_{L^2}^2 \nonumber \\
    &=\argmin_{\substack{\bm{c}_{a} \in \R_+^{m_{z_a}}, \displaystyle \dps\sum_{k=1}^{m_{z_a}} c_{a,k} = \int_{\R^3} \rho_a^{(m)}}} \left( \frac 12 \bm{c}_{a}^T \bm{S}_{z_a} \bm{c}_{a} - \bm{c}_{a}^T \bm{b}_a^{(m)} \right),  \label{eq:GISA5}\\
\end{align}
with 
\begin{align*}
    [S_{z_a}]_{k\ell} &= 2 \int_{\R^3} \zeta_{\alpha_{a,k}}(\br) \, \zeta_{\alpha_{a,\ell}}(\br) \, d\br =  \frac{ 2}{ \pi \sqrt{\pi}} \frac{\left( \alpha_{a,k} \alpha_{a,\ell} \right)^{\frac{3}{2}} }{\left(\sqrt{\alpha_{a,k}+\alpha_{a,\ell}}\right)^3}  , 
    \\
    {b}_{a,k}^{(m)} &=\int_{\R^3} \zeta_{\alpha_{a,k}}(\br) \, \rho_a^{(m)}(\br) \, d\br.  
\end{align*}
Again, the functions $\rho_a^{(m)}$ do not need to be stored in memory. The integrals $\int_{\R^3} \rho_a^{(m)}$ and the entries of the vectors $\bm{b}_a^{(m)}$ can be computed using quadrature formulas from the vector $\bc_a^{(m-1)}$ and the function $\rho$. 
In practice, $m_{z_a}$'s are small so that \eqref{eq:GISA5} can be solved very easily once these integrals have been computed, using standard routines for quadratic programming problems \cite{goldfarb1982dual,goldfarb1983numerically}.

\medskip

 The existence of a fixed point to the GISA iterations can be proved using Brouwer fixed-point theorem. On the other hand, the convergence of the GISA iterations for any initial guess is not guaranteed.

\subsection{MB-ISA iterations}\label{sec:MB-ISA}
\label{sec:MB_ISA_iterations}

 The MB-ISA method \cite{verstraelen2016minimal} was originally defined through a Lagrangian that does not seem to be canonically associated to a well-defined constrained optimization problem. 
 It is more satisfactory from a mathematical point of view, to reformulate it as follows:
 \begin{enumerate}
     \item the density $\rho$ is decomposed as a sum of atomic-shell densities. The set of admissible atomic-shell-in-molecule (ASIM) decompositions of the density $\rho$ is
     $$
     \bcK_{\rho,\bR,\bz} = \left\{ \bm\rho=(\rho_{a,k})_{1 \le a \le M, \, 1 \le k \le m_{z_a}} \in X_+^M \; \bigg| \sum_{a=1}^M \sum_{k=1}^{m_{z_a}} \rho_{a,k}(\cdot-\bR_a)=\rho  \right\}.
     $$
     The integer $m_{z_a}$ is the number of electronic shells in atom $a$; it is a function of the atomic number $z_a$ (hence the subscript $\bz=(z_1,\ldots,z_M) \in (\N^\ast)^M$ collecting the atomic numbers of the atoms contained in the molecular system of interest);
     \item the set of proatomic-shell densities is defined as
     $$
     \cK^0_{\rm MB-ISA}=\left\{\rho^0(\br)= c \frac{\alpha^3}{8\pi} e^{-\alpha |\br|}, \; c \ge 0, \; \alpha > 0 \right\},
     $$
     (note that this set is independent of the chemical element and the shell) and the set of admissible pro-ASIM densities as
     $$
     \bcK^0_{\bz}=\left\{  \bm\rho^0=(\rho_{a,k}^0)_{1 \le a \le M, \, 1 \le k \le m_{z_a}}, \; \rho_{a,k}^0 \in \cK^0_{\rm MB-ISA} \right\};
     $$
     \item the atomic-shell relative entropy is defined as
     $$
     S_{\rm sh}(\bm\rho|\bm\rho^0):= \sum_{a=1}^M \sum_{k=1}^{m_{z_a}} s_{\rm KL}(\rho_{a,k}|\rho_{a,k}^0);
     $$
     \item the optimal MB-ISA decomposition is finally obtained by solving the constrained minimization problem 
     \begin{equation} \label{eq:MS-ISA}
     \inf_{(\bm\rho,\bm\rho_0) \in \bcC_{\rho,\bR,\bz}} S_{\rm sh}(\bm\rho|\bm\rho^0),
     \end{equation}
     where 
     $$
     \bcC_{\rho,\bR,\bz}= \left\{(\bm\rho,\bm\rho^0) \in \bcK_{\rho,\bR,\bz} \times \bcK^0_{\bz}  \; | \; \bcM_{0}(\bm\rho)=\bcM_{0}(\bm\rho^0), \; \bcM_1(\bm\rho)=\bcM_1(\bm\rho^0) \right\}.
     $$
     Denoting by $K=\displaystyle \sum_{a=1}^M m_{z_a}$ and 
     $$
     X_n:=\left\{ \rho \in X \; \left| \; \int_{\R^3} |\br|^n |\rho(\br)| \, d\br < \infty \right. \right\}, 
     $$
     the linear maps $\bcM_n:X_n^K \to \R^K$ are defined by
     $$
     \forall \bm\rho=(\rho_{a,k})_{1\le a \le M, \, 1 \le k \le m_{z_a}} \in X_n^K, \; [\bcM_n(\bm\rho)]_{a,k}=\int_{\R^3} |\br|^n \rho_{a,k}(\br) \, d\br.
     $$
 \end{enumerate}
 Note that $X_0=X$ and $\bcM_0=\bcN$. As for any $c\in \mathbb{R}$,
 $$
 \int_{\R^3}  c \frac{\alpha^3}{8\pi} e^{-\alpha |\br|} \, d\br = c \quad \mbox{and} \quad 
 \int_{\R^3}  |\br| c \frac{\alpha^3}{8\pi} e^{-\alpha |\br|} \, d\br = \frac{3c}\alpha,
 $$
 the function
 $$
 \bG^0_{\bcK^0_\bz}(\bm\rho) := {\rm argmin} \big\{S_{\rm sh}(\bm\rho|\bm\rho^0), \; \bm\rho^0 \in \bcK^0_{\bz}, \; \bcM_{0}(\bm\rho^0)=\bcM_{0}(\bm\rho), \; \bcM_1(\bm\rho^0)=\bcM_1(\bm\rho) \big\},
 $$
 has an explicit expression: for all $\bm\rho=(\rho_{a,k})_{1 \le a \le M, \, 1 \le k \le m_{z_a}} \in \bcK_{\rho,\bR,\bz}$, 
 $$
 [\bG^0_{\bcK^0_\bz}(\bm\rho)]_{a,k}(\br) = c(\rho_{a,k})
 \frac{\alpha(\rho_{a,k})^3}{8\pi} e^{-\alpha(\rho_{a,k}) |\br|},
 $$
 with
 $$
 c(\rho_{a,k}) = \int_{\R^3} \rho_{a,k}(\br) \, dr \quad \mbox{and} \quad
 \alpha(\rho_{a,k}) = 3 \, c(\rho_{a,k}) \left( \int_{\R^3} |\br| \rho_{a,k}(\br) \, dr \right)^{-1}.
 $$
 MB-ISA can be reformulated in an easily implementable iterative algorithm in which the main variable is ${\mathbf y}:=(c_{a,k},\alpha_{a,k})_{1 \le a \le M, \, 1 \le k \le m_{z_a}} \in \R_+^K$:
 
\medskip
 
\noindent
\mybox{
\bfseries MB-ISA algorithm: \normalfont
\begin{itemize}
    \item \bfseries Initialization: \normalfont choose $(c_{a,k}^{(0)})_{1 \leq a \leq M,1 \leq k \leq m_{z_a}}$ and $(\alpha_{a,k}^{(0)})_{1 \leq a \leq M,1 \leq k \leq m_{z_a}}$ -- see section \ref{subsec:implementation_details} for details.
    \item \bfseries Iteration $m\geq 1$: \normalfont
    set 
   \begin{align}
        c_{a,k}^{(m)} &= \bigintss_{\mathbb{R}^3}  \frac{ c_{a,k}^{(m-1)}\left(\alpha_{a,k}^{(m-1)}\right)^3 e^{-\alpha_{a,k}^{(m-1)}|\br|}}{\displaystyle \sum_{b=1}^M \left( \sum_{\ell=1}^{K_b} c_{b,\ell}^{(m-1)}\left(\alpha_{b,\ell}^{(m-1)}\right)^3 e^{-\alpha_{b,\ell}^{(m-1)}|\br+\bR_a-\mathbf{R_b}|} \right) } \rho(\bR_a+\br) d\br, \\
    \alpha_{a,k}^{(m)} &= 3 \, c_{a,k}^{(m)} \left( \bigintss_{\mathbb{R}^3} \frac{ c_{a,k}^{(m-1)}\left(\alpha_{a,k}^{(m-1)}\right)^3 e^{-\alpha_{a,k}^{(m-1)}|\br|}}{ \displaystyle \sum_{b=1}^M \left( \sum_{\ell=1}^{K_b} c_{b,\ell}^{(m-1)}\left(\alpha_{b,\ell}^{(m-1)}\right)^3 e^{-\alpha_{b,\ell}^{(m-1)}|\br+\bR_a-\mathbf{R_b}|} \right) } |\br|  \rho(\bR_a+\br) d\br  \right)^{-1}, \label{eq:alpha_ak}
\end{align}
\end{itemize}
}
which only requires repeated evaluations of integrals of the form \eqref{eq:integral_atomic_charge} (with $\rho(\bR_a+\br)$ replaced by $|\br|\rho(\bR_a+\br)$ in the evaluation of \eqref{eq:alpha_ak}).

\section{Mathematical analysis of AIM decomposition methods}
\label{sec:maths}

In this section, we make a specific focus on the AIM decomposition methods falling into the scope of the unified formulation~\eqref{eq:opt_AIM}:  find $\left(\bm \rho^{\rm opt}, \bm \rho^{0,\rm opt}\right) \in \bcK_{\rho, \bR} \times \bcK^0$ solution to
\begin{equation}\label{eq:ISAopt}
\inf_{
\begin{array}{c}\left(\bm \rho, \bm \rho^0\right) \in \bcK_{\rho, \bR}\times \bcK^0\\
\bcN(\bm \rho) = \bcN(\bm \rho^0)\\
\end{array}
} S\left( \bm \rho | \bm \rho^0 \right),
\end{equation}
where $\bcK^0$ is a non-empty subset of $(X_+^r)^M$. Note that the only freedom to change the method is the set $\bcK^0$ and the remaining mathematical structure remains identical.
\medskip

For the ISA and L-ISA method, the set $\bcK^0$ is convex and closed for the topology of $X^M$. Consequently, \eqref{eq:ISAopt} is a strictly convex problem, so that its solution, if it exists, is unique. On the other hand, the existence of a minimizer is not obvious {\it a priori} due to possible loss of compactness; this issue is addressed in Theorem~\ref{thm:WPcontPb}.

\medskip

For the Hischfeld-I method, the existence of a minimizer follows from the very simple structure of $\bcK^0_{\rm Hirschfeld-I}$ by a simple compactness argument, while the uniqueness of the AIM decomposition is not guaranteed due to the non-convexity of $\bcK^0_{\rm Hirschfeld-I}$. For NL-ISA, neither existence, nor  uniqueness, can be proved in the absence of further assumptions on $\rho$ and the functions $\widetilde g_{z_a,k,\alpha_{a,k}}$.

\medskip

In the rest of the section, we focus on ISA and L-ISA. Let us emphasize that although the sets $\bcK^0_{\rm L-ISA}$ and $\bcK^0_{\rm GISA}$ have the same mathematical structure, the arguments below cannot be applied to GISA since the latter approach does not fit into the framework~\eqref{eq:ISAopt}.  The existence and uniqueness of a minimizer to the corresponding optimization problem is established in Theorem~\ref{thm:WPcontPb}. The aim of Theorem~\ref{th:convergence} is to study the convergence properties of ISA and L-ISA.  We prove that the objective functional is non-increasing along the iterations and that the $L^2$-distance between to successive iterates tends to zero as the iteration number tends to infinity. In the L-ISA case, we prove the convergence of the L-ISA algorithm towards the unique minimizer of the L-ISA optimization problem. 
 Some additional mathematical properties of the minimizer of the ISA optimization problem in the diatomic case are collected in Propositions~\ref{prop:spheric} and~\ref{prop:continuity}. All the proofs are postponed until Section~\ref{sec:proofs}. 

\medskip

\begin{theorem}
\label{thm:WPcontPb}
Let $\rho \in X_+$. Then,
\begin{enumerate}
\item For $\bcK^0 = \bcK^0_{\rm ISA}$, problem~\eqref{eq:ISAopt} admits a unique minimizer $\left(\bm \rho^{\rm opt}, \bm \rho^{0,\rm opt}\right) \in \bcK_{\rho, \bR} \times \bcK^0_{\rm ISA}$. 
\item The same holds for $\bcK^0 = \bcK^0_{\rm L-ISA}$ under the following additional assumption
\begin{equation}\label{eq:cond_gzak}
 \forall 1\leq a \leq M, \quad \rho(\,\cdot + \bR_a) \log \left(\mathop{\min}_{1\leq k \leq m_{z_a}} g_{z_a,k}(\cdot)\right) \in L^1(\mathbb{R}^3).
\end{equation}
\item In both cases (ISA and L-ISA),  $\bm \rho^{\rm opt}:=\left(\rho_a^{\rm opt}\right)_{1\leq a \leq M}$ and $\bm \rho^{0,\rm opt}:=\left(\rho_a^{0,\rm opt}\right)_{1\leq a \leq M}$ satisfy the Hirschfeld relation
\begin{equation} \label{eq:EL1}
\forall 1\leq a \leq M, \quad \rho^{\rm opt}_a(\bm r) = \frac{\rho^{0,\rm opt}_a(\bm r)}{\displaystyle \sum_{b=1}^M  \rho^{0,\rm opt}_b(\bm r - \bR_b + \bR_a)} \; \rho(\bm r + \bR_a),
\end{equation}
with the convention that $\rho^{\rm opt}_a(\bm r) =0$ whenever $\rho^{0,\rm opt}_a(\bm r)=0$.
\item In the ISA case, we have in addition 
\begin{equation} \label{eq:EL11}
\rho^{0,\rm opt}_a = \langle \rho^{\rm opt}_a \rangle_{\rm s}.
\end{equation}
\end{enumerate}
\end{theorem}

\medskip

At first sight, it is not clear at all that the fixed-point iterations \eqref{eq:step1}-\eqref{eq:step2} for solving problem~\eqref{eq:ISAopt} should converge. Indeed, the constraint $\bcN(\bm \rho) = \bcN(\bm \rho^0)$ is not taken into account in the minimization subproblem \eqref{eq:step2}, and it is therefore not clear {\it a priori} that the algorithm \eqref{eq:step1}-\eqref{eq:step2} possesses a Lyapunov functional (i.e. a functional which decreases at each iteration, a very useful tool in convergence proofs). It turns out that the function $S$ itself is in fact a Lyapunov functional: $S\left( \bm \rho^{(m+1)} | \bm \rho^{0,(m+1)} \right) \le S\left( \bm \rho^{(m)} | \bm \rho^{0,(m)} \right)$ for all $m$. This is a consequence of the special properties of the Kullback-Leibler divergence, which gives rise to the following result.

\medskip

\begin{theorem}\label{th:convergence}
Let $\rho \in X_+$ be such that $\rho > 0$ almost everywhere, and $\bm\rho^{0,(0)}=(\rho^{0,(0)}_a)_{1 \le a \le M} \in \bcK^0$ (with $\bcK^0=\bcK^0_{\rm ISA}$ or $\bcK^0=\bcK^0_{\rm L-ISA}$) be an initial guess such that 
$\rho_a^{0,(0)} > 0$ almost everywhere for all $1 \le a \le M$, and $s_{\rm KL}(\rho|\rho^{0,(0)}) < +\infty$, where $\rho^{0,(m)}=\sum_{a=1}^M \rho^{0,(m)}_a(\,\cdot\,-\bR_a)$. In the L-ISA case, we assume in addition that \eqref{eq:cond_gzak} is satisfied and all the $g_{z_a,k}$'s are positive almost everywhere.
Then 
\begin{enumerate}
\item For both ISA and L-ISA, the AIM iterations \eqref{eq:step1}-\eqref{eq:step2} are well-defined and $S$ is a Lyapunov functional of the algorithm. More precisely, we have for all $m \ge 1$, $(\bm \rho^{(m)}, \bm \rho^{0,(m)}) \in \bcC_{\rho,\bR}$ and the estimates
\begin{equation}\label{eq:decrease}
S(\bm \rho^{(m-1)}| \bm \rho^{0,(m-1)}) - S(\bm \rho^{(m)}| \bm \rho^{0,(m)}) \geq \frac{1}{2\|\rho\|_{L^\infty}} \sum_{a=1}^M \left\| \rho_a^{(m)} -  \rho_a^{(m-1)}\right\|_{L^2}^2 \geq 0,
\end{equation}
and
\begin{equation}\label{eq:decrease_KL}
s_{\rm KL}(\rho| \rho^{0,(m-1)}) - s_{\rm KL}(\rho|  \rho^{0,(m)}) \geq \frac{1}{2\|\rho\|_{L^\infty}} \sum_{a=1}^M \left\| \rho_a^{(m+1)} -  \rho_a^{(m)}\right\|_{L^2}^2 \geq 0.
\end{equation}
As a consequence, it holds that
\begin{equation}\label{eq:diffiter}
\bm \rho^{(m)} - \bm \rho^{(m-1)} \mathop{\longrightarrow}_{m\to +\infty} 0 \quad \mbox{strongly in } L^2(\mathbb{R}^3)^M. 
\end{equation}
\item For L-ISA, the sequence $(\bm\rho^{0,(m)},\bm\rho^{(m)})_{m \in \N}$ converges to the unique minimizer $(\bm\rho^{\rm opt},\bm\rho^{0,\rm opt})$ of (\ref{eq:ISAopt}) in some sense. More precisely, $(\bm\rho^{0,(m)})_{m \in \N}$ converges strongly to $\bm\rho^{0,\rm opt}$ in $L^p(\mathbb{R}^3)$ for any $1\leq p \leq +\infty$, and $(\bm\rho^{(m)})_{m \in \N}$ converges to $\bm\rho^{\rm opt}$ for the weak-* topologies of $\cM_{\rm b}(\R^3)^M$ and $L^\infty(\R^3)^M$, for the weak topology of $L^p(\R^3)^M$ for any $1 < p < \infty$, and for the strong topology of $L^\infty_{\rm loc}(\R^3)$.
\end{enumerate}
\end{theorem}
Here, $\cM_{\rm b}(\R^3)$ denotes the Banach space of the bounded (signed) Radon measures on $\R^3$ endowed with the total variation norm.

\medskip

We finally establish additional mathematical properties of the ISA AIM decomposition $\bm \rho^{\rm opt} = (\rho^{\rm opt}_1, \rho^{\rm opt}_2)$ of diatomic systems ($\bcK^0=\bcK^0_{\rm ISA}$ and $M=2$). We first show that one of the atomic densities $\rho^{\rm opt}_a$ is identically equal to zero if and only the total density $\rho$ is a radial function centered on the other atom.

\medskip

 \begin{proposition}\label{prop:spheric}
 Let us assume that $M = 2$ and $\rho \in X_+$ is continuous and positive on $\R^3$. Let $\left(\bm \rho^{\rm opt}, \bm \rho^{0,\rm opt}\right)$  be the unique minimizer of problem (\ref{eq:ISA_opt}) for $\bcK^0 = \bcK^0_{\rm ISA}$. Then, $\rho^{\rm opt}_1 =0$ if and only if $\rho(\cdot + \bR_2) = \rho^{\rm opt}_2 = \langle \rho^{\rm opt}_2\rangle_s$. 
 \end{proposition}
 
 \medskip
 
 We also have the following result, which proves stronger regularity results on $\bm \rho$ under appropriate assumptions 
 on the regularity of $\rho$. 
 \begin{proposition} \label{prop:continuity}
 Assume that $M = 2$ and $\rho \in X_+ \cap C^{0,1}(\R^3)$. Let $(\bm \rho^{\rm opt},\bm \rho^{0,\rm opt})$ be
  the unique minimizer of problem \eqref{eq:ISAopt} for $\bcK^0 = \bcK^0_{\rm ISA}$.
For $a=1,2$, let $w_a\in Y_+$ be the function defined by
$$
\forall r>0, \quad \forall \bm\sigma\in \S^2, \quad w_a(r):= \langle \rho_a^{\rm opt}\rangle_s(r\bm\sigma).
$$
Then, 
\begin{enumerate}
    \item for almost all $r \ge 0$, $w_a(r) \in \{0,w_a^+(r)\}$, where the functions $w_a^+:\R_+ \to \R_+$ are bounded by $M_\rho:=\max_{\R^3}\rho$, continuous on $(0,+\infty)$, and Lipschitz in the neighborhood of any point of $(0,+\infty)$ at which they are positive;
    \item if in addition, the density $\rho$ is $C^k$ away from the centers $\bR_a$ for $a=1,2$ and the functions $w_a$ are essentially bounded away from zero  on any compact subset of $[0,+\infty)$, then the functions $w_a$ are Lipschitz on $[0,+\infty)$ and $C^k$ on $(0,R) \cup (R,+\infty)$ with $R:=|\bR_a - \bR_b|$.
\end{enumerate}
\end{proposition}

\section{Distributed Multipole Analysis}
\label{sec:DMA}

Other schemes allow to represent the electronic density as a sum of atom-centered contributions. Mulliken population analysis \cite{Mulliken1955}, and Distributed Multipole Analysis (DMA) \cite{Stone1981,Stone1985} -- which can be viewed as a generalization of Mulliken scheme to higher-order multipole moments -- belong to the latter. These schemes are fully basis-set oriented, in the sense that they use as a starting point a decomposition of the density of the form 
\begin{equation}
\rho(\mathbf{r}) = \sum_{k=1}^{M'} f_k(\mathbf{r}-\mathbf{a}_k)
\label{general_expression_density_for_possible_redistribution}
\end{equation}
where ${\cal A} :=\{{\mathbf a}_k \}_{1 \le k \le M'}$ is a set of points of the physical space $\R^3$, and $f_k : \R^3 \to \R^3$ functions for which the multipoles
\begin{equation}
(Q_{\ell m}^k) = K_\ell^m \int_{\mathbb{R}^3} \mathcal{R}_{\ell}^{m}(\mathbf{r}) f_k(\mathbf{r})  d\mathbf{r},
\label{eq:multipole_moments_f_k}
\end{equation}
with $\mathcal{R}_{\ell}^{m}(\mathbf{r})=|\mathbf{r}|^{\ell} Y_{\ell}^{m}\left( \frac{\mathbf{r}}{|\mathbf{r}|}\right)$ (real solid harmonics using $L^2(\S^2)$-orthonormal real spherical harmonics),
are easily computable. The coefficients $K_\ell^m$ are real normalization coefficients (for instance, $K_0^0=\sqrt{4\pi}$ and $K_1^m = \frac{\sqrt{4\pi}}{3}$). 

The purpose of DMA methods is to redistribute the multipoles $Q_{\ell m}^k$ carried by the points $\mathbf{a}_k$ to multipoles $\widetilde Q_{\ell m}^j$ carried by some chosen points ${\mathbf S}_j$,  called DMA expansion sites.  The set ${\cal S}=\{{\mathbf S}_j\}_{1 \le j \le J}$ can coincide with the set of atomic positions, but can include other points. For instance, Mulliken population analysis defines overlap populations (that can be attributed to the bonds and allocated to bond centers), and the DMA can also include charge contributions from \textit{e.g.} bond centers and from many other points. 

\medskip

The rationale of the distributed multipole analysis is thus the following: 
\begin{enumerate}
\item first,  the multipole moments $(Q_{\ell m}^k)$ are computed. The latter will be coined multipole moments \textit{at natural centers} in the following. In the case of an electronic density computed in a basis of Gaussian Type Orbitals (GTOs) -- as used in most quantum chemistry applications --, there is an analytical expression for these elementary moments, allowing for a very fast computation;
\item then, these multipoles are redistributed to the DMA expansion sites ${\mathbf S}_j$, giving rise to the redistributed multipoles $\widetilde Q_{\ell m}^j$ carried by the sites ${\mathbf S}_j$.
\end{enumerate}
These two aspects are discussed in the following two sections.

\subsection{Multipole moments with respect to natural centers}

We recall here the procedure for rewriting as \eqref{general_expression_density_for_possible_redistribution} an electronic density $\rho$ issued from a quantum chemistry calculation in a Gaussian type orbital (GTO) basis set, and computing the corresponding multipoles $Q_{\ell m}^k$. First, $\rho$ is decomposed as
\begin{equation}
\rho(\mathbf{r}) = \sum_{\mu,\nu=1}^{N_b} P_{\mu\nu}
\,
{\chi}_{\mu}(\mathbf{r}-\bR_{a_\mu})
{\chi}_{\nu} (\mathbf{r}-\bR_{a_\nu}),
\label{decomposition_density_n_density_matrix_pGTOs}
\end{equation}
where $({\chi}_{\mu})_{1 \le \mu \le N_b}$ is a basis of {\em primitive} Gaussian polynomial functions, $a_\mu$ the label of the nucleus $\chi_\mu$ is attached to, and $(P_{\mu\nu})_{1 \le \mu,\nu \le N_b}$ the density matrix in this basis. The matrix $(P_{\mu\nu})_{1 \le \mu,\nu \le N_b}$ is easily obtained from the density matrix in the {\em contracted} GTO basis set used for the computation of $\rho$ by a suitable linear transform. Each $\chi_\mu$ is of the form
\begin{equation}
{\chi}_{\mu}(\mathbf{r})= N_{\mu} \mathcal{R}_{\ell_\mu}^{m_\mu}(\mathbf{r}) e^{-\zeta_{\mu}|\mathbf{r}|^2},
\end{equation}
with $\ell_\mu$ and $m_\mu$ its angular momentum quantum numbers, $\zeta_{\mu}$ the Gaussian exponent controlling its decay rate, and $N_{\mu}$ a normalization constant. Since 
\begin{equation}
    e^{-\zeta_\mu \left|\mathbf{r}-\mathbf{R}_{a_\mu} \right|^2} e^{-\zeta_\nu \left|\mathbf{r}-\mathbf{R}_{a_\nu}\right|^2} = K_{\mu\nu} e^{-(\zeta_{\mu}+\zeta_{\mu}) \left|\mathbf{r}-\mathbf{R}_{\mu\nu}\right|^2},
    \label{product_two_gaussians_centered_different_points}
\end{equation}
with $K_{\mu\nu}=e^{-\frac{\zeta_{\mu} \zeta_{\mu}}{\zeta_{\mu}+\zeta_{\nu}}\left|\mathbf{R}_{a_\mu}-\mathbf{R}_{a_\nu}\right|^2}$, and $\mathbf{R}_{\mu\nu}=\frac{\zeta_{\mu} \mathbf{R}_{a_\mu} + \zeta_{\nu} \mathbf{R}_{a_\nu}}{\zeta_{\mu}+\zeta_{\nu}}$, it holds that 
\begin{equation}
    {\chi}_{\mu} (\mathbf{r}-\bR_{a_\mu}) 
    {\chi}_{j} (\mathbf{r}-\bR_{a_\nu}) 
    = 
    K_{\mu\nu} N_{\mu} N_{\nu} \, 
    \mathcal{R}_{\ell_\mu}^{m_\mu}(\mathbf{r}-\bR_{a_\mu}) \,
    \mathcal{R}_{\ell_\nu}^{m_\nu}(\mathbf{r}-\bR_{a_\nu}) \,
    e^{-(\zeta_{\mu}+\zeta_{\mu}) \left|\mathbf{r}-\mathbf{R}_{\mu\nu}\right|^2},
\label{expression_product_two_primitive_GTOs}
\end{equation}
so that the electronic density given by~\eqref{decomposition_density_n_density_matrix_pGTOs}
can be written as \eqref{general_expression_density_for_possible_redistribution},
where the functions $f_k:\mathbb{R}^3\to\mathbb{R}$ are Gaussian polynomials of the form \eqref{expression_product_two_primitive_GTOs}, and $\mathcal{A} := \left\lbrace \mathbf{a}_k \right\rbrace_{1 \le k \le M'}$ a set of points in the physical space, (strictly) containing the positions of the nuclei. 

\medskip

The key quantities for the calculation of the distributed multipole moments are thus the multipole moments with respect to natural expansion centers $\mathbf{R}_{\mu\nu}$:
\begin{equation}
   Q_{\ell m}^{\mu\nu}  =P_{\mu\nu} K_\ell^m \int_{\mathbb{R}^3} \mathcal{R}_{\ell}^{m}(\mathbf{r}) \chi_{\mu}(\br- \bR_\mu+\bR_{\mu\nu}) \chi_{\nu}(\mathbf{r}- \bR_\nu+\bR_{\mu\nu}) d\mathbf{r} .
    \label{eq:multipole_moment_center_overlap_center_P_i_j}
\end{equation}
Given the expression \eqref{expression_product_two_primitive_GTOs}, the series of multipole moments $\{Q_{\ell m}^{\mu\nu} \}_{\ell \in \mathbb{N},m=-\ell,\ldots,\ell}$ is finite and can be easily evaluated numerically, by expanding the product of the two solid harmonics in  \eqref{expression_product_two_primitive_GTOs} as a weighted sum of homogeneous polynomials, multiplying it with the solid harmonic $\mathcal{R}_{\ell}^{m}(\mathbf{r})$  (homogeneous polynomial) and finally using the series of tabulated integrals:
\begin{equation}
I_q= \int_{0}^{+\infty} u^q e^{-u^2}du = \frac{\Gamma\left( \frac{q+1}{2}\right)}{2},
\end{equation}
for all needed $q$ integer values, where $\Gamma$ is the Euler Gamma function. 

\subsection{Redistribution formula}

Once multipole moments at natural centers are computed, a lower number of expansion sites (than the total number of overlap centers \textit{i.e.} $|\mathcal{A}|$) may be desired to perform the distributed multipole expansion, as the number of overlap centers grows quadratically with the number of atoms or basis size. For instance, only nuclei positions may be retained as final centers (e.g. to use these local multipole moments in polarizable force fields such as AMOEBA \cite{Shi2013}), or all nuclei and bond centers \cite{Stone1981,Vigne-Maeder1988}. The multipole moments at natural centers, which are easy to compute (see above), have thus to be redistributed to the chosen final expansion sites. We recall that the set of final expansion sites is a choice taken by the user.

\medskip

This requires two ingredients. First, the introduction of redistribution weights in order to distribute the multipoles from natural sites not belonging to the set of final sites. 
Second, a transformation of a multipole corresponding to one expansion center to a multipole corresponding to a different expansion center which is a linear transformation that is known in the Fast Multipole Method (FMM) framework as $M2M$-operator~\cite{fortin2006algorithmique} (used here in the context of real-valued multipolar coefficients).\\

We first focus on the latter. We briefly present the resulting formula, and refer to \cite{these_Robert} for a detailed derivation:
\begin{align}
 Q_{\ell m}^{\mu\nu \to j} 
&= 
\frac{ \Re \left( \left(  Q_{\ell,-m}^{\mu\nu \to j}  \right)^{\mathbb{C}} \right)  + (-1)^m \Re \left( \left(  Q_{\ell m}^{\mu\nu \to j}  \right)^{\mathbb{C}} \right) }{\sqrt{2}} \quad \mbox{ if } m > 0, 
\label{redistributed_multipole_S_final_expression_m_larger_0}
\\
 Q_{\ell m}^{\mu\nu \to j} 
&= 
\frac{ (-1)^{|m|} \Im \left( \left(  Q_{\ell,-m}^{\mu\nu \to j}  \right)^{\mathbb{C}} \right) - \Im \left( \left(  Q_{\ell m}^{\mu\nu \to j}  \right)^{\mathbb{C}}\right) }{\sqrt{2}} \quad \mbox{ if } m < 0,
\label{redistributed_multipole_S_final_expression_m_smaller_0}
\\
 Q_{\ell,0}^{\mu\nu \to j} 
&= 
\Re \left( \left(  Q_{\ell,0}^{\mu\nu \to j}  \right)^{\mathbb{C}}\right),
\label{redistributed_multipole_S_final_expression_m_equal_0}
\end{align}
where the complex $M2M$-operator is given by 
\begin{align}
    \left(  Q_{\ell m}^{\mu\nu \to j}  \right)^{\mathbb{C}}
    = 
    \sum_{\ell'=0}^\ell \sum_{m'=-\ell'}^{\ell'} c_{\ell'm'}^{\ell m} \left| \mathbf{R}_{\mu\nu} - \mathbf{S}_j \right|^{\ell-\ell'} {\cal Y}_{\ell-\ell'}^{m-m'} \left( \frac{\mathbf{R}_{\mu\nu} - \mathbf{S}_j}{\left|\mathbf{R}_{\mu\nu} - \mathbf{S}_j\right|} \right)  \left( Q_{\ell',m'}^{\mu\nu} \right)^{\mathbb{C}},
    \label{formula_redistribution_COMPLEX_multipoles}
\end{align}
where ${\cal Y}_{\ell}^{m}$ denote the complex ($L^2(\S^2)$-orthonormal) spherical harmonics and 
$
    c_{\ell',m'}^{\ell m} =  \sqrt{\binom{\ell+m}{\ell'+m'} \binom{\ell-m}{\ell'-m'}}
$
a normalization factor. 
Complex arithmetic operations are used since it simplifies the presentation of the formulae.
Finally, the complex multipole moments are simply given by 
\begin{equation}
    \left( Q_{\ell m}^{\mu\nu} \right)^{\mathbb{C}}
    = P_{\mu\nu} K_\ell^m \int_{\R^3} \left|\mathbf{r}\right|^\ell {\cal Y}_\ell^m\left( \frac{\mathbf{r}}{\left|\mathbf{r}\right|} \right) \chi_{\mu}(\br- \bR_\mu+\bR_{\mu\nu}) \chi_{\nu}(\mathbf{r}- \bR_\nu+\bR_{\mu\nu}) d\mathbf{r} .
    \label{Q_l_m_i_j_center_S_complex}
\end{equation}

The multipole moments are redistributed to the final expansion sites $\mathbf{S}_j$ using a formula of the form 
\begin{equation}
\widetilde Q_{\ell m}^j = \sum_{(\mu,\nu)} C_{\mu\nu \to j} \left( Q_{\ell m}^{\mu\nu \to j} \right)^{\mathbb{R}},
\label{eq:total_multipole_moment_point_S_after_redistribution_FINAL}
\end{equation}
where the $C_{\mu\nu \to j}$'s are  {\em user-specified} redistribution weights, representing the part of the multipolar distribution centered at $\mathbf{R}_{\mu\nu}$ allocated to the expansion site $\mathbf{S}_j$. It holds 
$$
C_{\mu\nu \to j}  \ge 0, \quad \sum_{j=1}^p C_{\mu\nu \to j} =1, \quad C_{\mu\nu \to j}=1 \mbox{ if } {\mathbf S}_j=\bR_{\mu\nu}.
$$
In particular, if $\mathbf{R}_{\mu\nu}$ is one of the final expansion sites, then $Q_{\ell m}^{\mu\nu}$ is fully allocated to this site. As a consequence, $Q_{\ell m}^{\mu\nu \to j}$ needs to be computed only if  $\mathbf{R}_{\mu\nu}$ is {\em not} one of the final expansion sites. It is up to the user to choose both the final expansion sites and the values of the weight coefficients $C_{\mu\nu \to j}$. A complete flexibility is allowed in our DMA code \cite{DMA_NEW_CODE}. For instance, following Stone's \textit{neighbor-takes-it-all} strategy implemented in GDMA code \cite{stone2005distributed} would amount to take:
\begin{equation}
C_{\mu\nu\to j} = \left\{ \begin{array}{ll}
1  \mbox{ if } \mathbf{S}_j \mbox{ is the final expansion site nearest to } \mathbf{R}_{\mu\nu}, \\
0 \mbox{ otherwise},
\end{array} \right.
\label{eq:Stone_redistribution_strategy}
\end{equation}
or $\frac{1}{q}$ for each of the $q$ final sites equally close to $\mathbf{\bR}_{\mu\nu}$. This strategy means that multipole moments at natural centers are only redistributed \textit{locally}, at the nearest neighbor final site. An obvious drawback of this approach is that the redistributed moments $\widetilde Q_{\ell m}^j$ are clearly not differentiable, nor even continuous, with respect to the nuclear coordinates.
We have also implemented, and provide as a possible user-choice, 
the more balanced rule of Vigne-Maeder \textit{et al.} \cite{Vigne-Maeder1988}, which amounts to choose the following weight coefficients:
\begin{equation}
C_{\mu\nu \to j} = \frac{\frac{1}{|\mathbf{S}_j - \mathbf{\bR}_{\mu\nu}|}}{ \underset{k}{\sum} \frac{1}{|\mathbf{S}_k -  \mathbf{\bR}_{\mu\nu}|}},
\label{eq:Vigne_Maeder_redistribution_strategy}
\end{equation}
where the sum in the denominator runs over the $p$ chosen final DMA sites. This time, multipole moments at natural centers that are overlooked in the set of DMA expansion sites are redistributed to \textit{all} final DMA sites $\left\lbrace \mathbf{S}_j \right\rbrace_j$, with larger contributions to nearer sites.

\subsection{Multipolar expansion of the electrostatic potential}

Once the local distributed multipole moments are computed, the electrostatic potential generated by the charge density distribution (that the latter summarize) can be evaluated by a simple sum, instead of a three-dimensional integral needed to evaluate the exact, quantum potential:
$$V(\mathbf{r}) = \int_{\mathbb{R}^3} \frac{\rho(\mathbf{r'})}{\left| \mathbf{r} - \mathbf{r'} \right|} d \mathbf{r'}. $$
The multipolar expansion of the total potential with respect to the expansion centers $\mathbf{R}_{\mu \nu}$ then writes:
\begin{equation}
\displaystyle
     V_{\rm M}(\mathbf{r}) =\sum_{\mu,\nu=1}^{N_b}  \left( \sum_{\ell=0}^{+\infty} \frac{4\pi}{(2\ell+1)} \sum_{m=-\ell}^{\ell} \frac{ Q_{\ell m}^{\mu\nu} Y_\ell^m\left(\frac{\mathbf{r}-\mathbf{R}_{\mu\nu}}{\left|\mathbf{r}-\mathbf{R}_{\mu\nu} \right|} \right)}{K_\ell^m \left|\mathbf{r}-\mathbf{R}_{\mu\nu} \right|^{\ell+1}} \right),
     \label{eq:multiple_center_multipolar_expansion}
\end{equation}
where the sum over $\ell$ is actually finite and terminates for $\ell=\ell_\mu+\ell_\nu$, the sum of the degrees of the two real solid harmonics involved in the primitive GTOs $\chi_\mu$ and $\chi_\nu$ respectively. The multipolar expansion $V_{\rm M}$ is an excellent approximation of $V$ far enough from the molecule, and a very poor approximation in the region of space where the electronic density is large. In the intermediate region, it is valid \textit{up to a penetration term} \cite{liu2019amoeba+,rackers2017optimized}.

\medskip

\noindent After redistribution of the multipole moments to the DMA final expansion centers, the potential generated by the DMA distributed local multipole moments writes:
\begin{equation}
   V_{\rm M}^{\rm DMA}(\mathbf{r}) = \sum_{j=1}^J \left( \sum_{\ell=0}^{+\infty} \frac{4\pi}{(2l+1)} \sum_{m=-\ell}^{\ell} \frac{ \widetilde Q_{\ell m}^j Y_\ell^m\left(\frac{\mathbf{r}-\mathbf{S}_j}{\left|\mathbf{r}-\mathbf{S}_j \right|} \right)}{K_\ell^m \left|\mathbf{r}-\mathbf{S}_j \right|^{\ell+1}} \right)
   \label{eq:multiple_center_multipolar_expansion_DMA_sites_approx_1}
\end{equation}
where the total multipole moments $(\widetilde Q_{\ell m}^j)_{1 \le j \le J}$ with respect to final expansion centers $\mathbf{S}_j$ include the contributions of multipole moments naturally centered at $\mathbf{S}_j$ (if any), as well as of the reallocated (\textit{e.g.} neighboring) multipoles -- see Eq. \eqref{eq:total_multipole_moment_point_S_after_redistribution_FINAL}. The series of multipole moments $\left(\widetilde Q_{\ell m}^j\right)_{l,m}$ with origin $\mathbf{S}_j$ is in general infinite, and must be truncated in practice. For this reason, $V_{\rm M}^{\rm DMA}$ is always an approximation of $V_{\rm M}$ as soon as one of the natural centers $\mathbf{R}_{\mu\nu}$ is not kept as a final DMA expansion site, even away from the numerical support of the density $\rho$.

\medskip

Let us finally mention that the new version of DMA \cite{stone2005distributed} can be thought as a mix between this DMA basis-set oriented method and real-space oriented (partitionning) methods, as it associates a redistribution of the density matrix contributions for some (the least diffuse) pairwise products of GTOs with a real-space integration for the contributions of the others (the most diffuse) pairwise products of GTOs, thus allowing to stabilize local multipole moments in the limit of a very large Atomic Orbitals basis for the molecular density $\rho$.

\section{Numerical results}
\label{sec:numerical_results}

We have implemented the four methods of the ISA family described in Section \ref{sec:AIM_methods_based_objective_functionals}, namely ISA, GISA, L-ISA and MB-ISA, in Python language with a specific focus on accurate quadrature schemes for diatomic molecules. This serves as a proof-of-concept to highlight the theoretical results obtained in this article and can by far not be considered as an exhaustive computational study.

\medskip

The Distributed Multipole Analysis (DMA) method (Section \ref{sec:DMA}), has also been implemented in Python language in a new, more modular package, allowing for user-specified redistribution strategies (see Eq.\eqref{eq:total_multipole_moment_point_S_after_redistribution_FINAL} and e.g. the two possible choices of coefficients \eqref{eq:Stone_redistribution_strategy} and \eqref{eq:Vigne_Maeder_redistribution_strategy}). Our code has been validated by comparing the results of local multipole moments (up to hexadecapoles) obtained on test molecules (H$_2$O, NH$_3$, CH$_4$, benzene, and ions such as ClO$^{-}$) with those obtained by the GDMA code \cite{stone2005distributed}, using Stone \textit{neighbor-takes-it-all} redistribution strategy (equation \ref{eq:Stone_redistribution_strategy}). Several basis sets were tested for each molecule. We obtained identical distributed (atomic) local multipole moments up to hexadecapoles, up to numerical accuracy.

\medskip

All these codes are built upon the functionalities of {\tt cclib} library \cite{Oboyle2008}, which allows to parse in a generic way the output of most quantum chemistry packages (14 different packages are supported by this library), without any specific parsing required by the user. This new feature enables an easy comparison of DMA or ISA partitionning results among the most popular quantum chemistry codes,  providing more interoperability. The parsing by the existing Stone's GDMA code \cite{stone2005distributed} for DMA was indeed only implemented for Gaussian \cite{g16} or Psi4 \cite{turney2012psi4} outputs, and similarly for the existing implementations of the ISA methods in Horton code \cite{HORTON}.

\subsection{Implementational details}
\label{subsec:implementation_details}

\noindent
{\bf Solver.} 
The ISA AIM decomposition has been computed using two different methods: (i) the fixed-point algorithm described in Section \ref{sec:ISA_iterations}, and (ii) an alternative scheme based on the Newton method applied to a suitable formulation of the Euler-Lagrange equations of the problem (see~\cite{these_Robert} for further details). Both methods yielded the same results (up to numerical accuracy) in our test cases.

\medskip

\noindent
{\bf Discretization and numerical quadrature.}
The AIM methods rely either on a real-space discretization, or on a basis-space representation of the functions spanning the linear or non-linear approximation space. 

As already detailed in Section \ref{sec:AIM_methods_based_objective_functionals}, the GISA,  L-ISA and MB-ISA methods rely on basis-set representations of pro-atomic densities, through linear combination of normalized Gaussians (for GISA and L-ISA) or Slater functions (for MB-ISA). For GISA and L-ISA, only the coefficients of the linear combination are optimized. For MB-ISA, the coefficients of the expansion \textit{and} the exponents involved in the Slater basis functions are both optimized simultaneously. Typically, a small number of shells (e.g. 4 for H and 6 for C, O and N \cite{verstraelen2012conformational}) are assigned to each atom in GISA and MB-ISA~\cite{verstraelen2016minimal} methods. 

On the other hand, the original ISA algorithm introduced by Lillestolen \textit{et al.} \cite{lillestolen2008redefining,Lillestolen2009} (Section \ref{sec:ISA_iterations}) relies on a real space radial and angular discretization by point-evaluations.

In the family of ISA-methods, several integrals arise both when computing the total mass $N_a$ of the atomic densities (imposed as a constraint in the GISA, L-ISA and MB-ISA schemes), the first-order moments of the atomic densities (imposed as a constraint in the MB-ISA scheme), or when computing spherical averages in the original ISA method.
This requires a numerical quadrature which is given by a tensor product of a radial and spherical grid. 
The radial grid can be e.g. an 1D integration grid based on $N_{\rm r}$ Gauss quadrature points extending up to some maximal $R_{\max}$, while the angular grid can be e.g. a Lebedev grid \cite{lebedev1999quadrature} made of $N_{\rm s}$ points on the unit sphere $\mathbb{S}^2$ (with their associated weights). For each atom $a$, the discretization points are defined as follows:
\begin{equation}
    \forall a \in [\![1,M]\!], \;  \forall i \in [\![1,N_{\rm r}]\!], \;  \forall j \in [\![1,N_{\rm s}]\!],  \; \qquad \mathbf{r}_{i,j}^{a} := \mathbf{R}_a+r_i^a \bm \sigma_j,
\end{equation}
denoting $[\![1,M]\!]=\{1,2,\ldots,M\}$ and where $\mathbf{R}_a$ is the position of atom $a$, $(r_i^a)_{i \in [\![1,N_{\rm r}]\!]}$ are the (atom-dependent) radial grid points (which can be e.g. equally-spaced or logarithmic), and $(\bm\sigma_j)_{j \in [\![1,N_{\rm s}]\!]}$ are the Lebedev grid points grids used for the quadrature on the unit sphere. 
The molecular density $\rho$ is precomputed at these points $\mathbf{r}_{i,j}^{a}$, yielding tabulated values $\rho_{i,j}^{a} = \rho(\mathbf{r}_{i,j}^{a})$ used throughout the algorithm repeatedly. 

In the special case of symmetric systems, such as for some diatomic molecules (associated to an electronic density invariant by any rotation around the axis of the molecule), the spherical average on the unit sphere $\mathbb{S}^2$ can be replaced by a simple one-dimensional integration on the $\theta$ variable using $N_{\rm s}$ points, e.g. using a Gauss-Legendre quadrature by aligning the atoms on the $z$-axis. All integrals on the radial variable $r$ are computed likewise.

\medskip

\noindent
{\bf Interpolation.}
In the ISA-method, where (pro-)atomic densities are known only on a real space grid, we use a piecewise linear (sometimes referred to $\mathbb{P}^1$) interpolation between the nearest neighbors. We use such an interpolation for instance to evaluate the pro-atomic density $\rho_a^0$ of atom $a$ at a radial discretization value $r_i^b$ associated to a near atom $b$. Indeed, only the values $\rho_a^0(r_i^a)$ for $i\in \![1,N_{\rm r}]\!]$, at the radial grid point $r_i^a$, are  known. 
On the contrary, for the GISA, L-ISA and MB-ISA schemes, the expressions of the pro-atomic densities are (by construction) analytical due to the expansion on Gaussian or Slater basis functions and no interpolation is thus needed for these methods.

\medskip

\noindent
{\bf Optimization.}
For the resolution of the GISA problem, in particular for the minimization of the quadratic functional under linear (equality and inequality) constraints \eqref{eq:GISA5}, we use standard routines for quadratic programming problems \cite{goldfarb1982dual,goldfarb1983numerically} implemented in the QUADPROG Python function. For solving the convex minimization \eqref{eq:G0-L-ISA-2} arising in L-ISA, we use the CVXOPT Python package \cite{andersencvxopt}.
 
\medskip

\noindent
{\bf Parameterization.}
For GISA and L-ISA methods, the atomic shell exponents $\left(\alpha_{a,k}\right)$, with $a\in [\![1,M]\!]$, $k\in [\![1,m_{z_a}]\!]$,
are fixed once and for all. The solutions of the GISA and L-ISA schemes thus \textit{a priori} depend on the values of these exponents. 
For C, H, N and O atoms, we use the values of the exponents provided by Verstraelen \textit{et al.} for the GISA method \cite{verstraelen2012conformational}, obtained by fitting to the densities of isolated (neutral or ionized) atoms, so that the GISA Gaussian basis represents accurately enough all the pro-atomic densities associated to the quantum calculations on both isolated neutral atoms and ions. For other atoms, for which no values of the exponents are available, we apply the empirical formula provided by Verstraelen \textit{et al.} in the MB-ISA paper \cite{verstraelen2016minimal}:
\begin{equation}
    \forall k \in [\![1,m_{z_a}]\!], \qquad  
    \alpha_{a,k}=\frac{2 z_a^{1-\frac{k-1}{m_{z_a}-1}}}{a_0},
    \label{eq:empirical_rule_exponents}
\end{equation}
where $a_0$ is the value of 1 Bohr in Angstr\"om and $m_{z_a}$ the number of shells (derived from the atomic number $z_a$). For MB-ISA, the exponents $\left(\alpha_{a,k}\right)$, with $a\in [\![1,M]\!]$, $k\in [\![1,m_{z_a}]\!]$, of the Slater functions used to expand the pro-atomic densities are optimized instead of attributing fixed values.
As initial guess for these exponents, we simply take the values of the GISA (or L-ISA) exponents detailed above. These initial values of the exponents appear numerically not to matter (as regards the final fixed-point solution) in the MB-ISA scheme, which optimizes both the weight and the Slater exponent of every pro-atomic shell (see Section \ref{sec:MB-ISA}). Note that some choices of exponents lead to an ill-conditionned local overlap matrix $\mathbf{S}_{z_a}$, which can prove problematic in particular in the context of the GISA method (see Section \ref{sec:GISA}), where a linear system involving this matrix has to be solved.

For the expansion coefficients $\left(c_{a,k}\right)$, with $a\in [\![1,M]\!]$, $k\in [\![1,m_{z_a}]\!]$, which are optimized in the GISA, L-ISA and MB-ISA methods, several initial guesses that satisfy the property that the initial pro-atomic mass $N_a^{(0)} =\displaystyle \sum_{k=1}^{m_{z_a}} c_{a,k}$ equals the atomic number $z_a$, are possible, such as 
\[
    \forall k \in [\![1,m_{z_a}]\!], \qquad c_{a,k}^{(0)} = \frac{z_a}{m_{z_a}}
\]
for example. Other less-balanced initial guesses, e.g. putting all the initial weight on a specific Gaussian (or Slater) basis function:  
\[
    \forall k \in [\![1,m_{z_a}]\!], \qquad c_{a,k}^{(0)} =  \delta_{k,k_0} z_a, 
\]
with $k_0 \in [\![1,m_{z_a}]\!]$ and where $\delta_{i,j}$ denotes the Kronecker delta, are also possible. \\

\medskip

\noindent
{\bf Uniqueness and dependency on initial condition.}
For the GISA method, there is no guarantee of the uniqueness of the solution for the expansion coefficients, and local minima are possible (i.e. a dependence of the computed solution on the initial guess for the expansion coefficients). An example found numerically for a very simple test density is provided in the Appendix.
On the contrary, for L-ISA, the optimal expansion coefficients $\left(c_{a,k}^{\rm opt}\right)$ does not depend on the initial guess  $\left(c_{a,k}^{(0)}\right)$ due to the strict convexity. 
For MB-ISA, there is no guarantee of the uniqueness of the fixed-point solution, so that the numerically computed solution may depend on the initial guess. However, numerically, the MB-ISA solution appears quite robust to the guess (both in terms of expansion coefficients $c_{a,k}$ and of exponents $\alpha_{a,k}$), and no local minimum has been unveiled in this study on diatomic systems. 

\medskip

\noindent
{\bf Validation.}
Our implementation of the different existing ISA partitionning schemes has been validated on some test diatomic molecules, for several basis sets, by comparison to the Horton code \cite{HORTON}, specifically for the ISA and MB-ISA scheme. Both implementations (ours and Horton's) yielded the same local multipole moments (up to quadrupoles) up to numerical accuracy on the tested molecules. For GISA and L-ISA, our implementation was validated based on custom test atomic densities made of a sum of one or several normalized Gaussians, centered at both atomic positions (diatomic molecules only).

\subsection{Comparison of the different ISA schemes}

In this section, we compare typical pro-atomic densities profiles and the associated atomic moments (up to second order) for diatomic molecules or ions, at the equilibrium inter-atomic distance for  Hartree-Fock densities obtained with GAMESS~\cite{Schmidt1993}.\\

\paragraph*{CO molecule.\\}

Figure \ref{fig:comparison_ISA_schemes_CO_RHF_ACCD_atom_C} represents the numerically computed $r \mapsto \log\left(4 \pi r^2 \rho_a^0(r) \right)$ profiles solutions of GISA, L-ISA, MB-ISA and ISA for the carbon atom in the CO molecule, using a molecular density $\rho$ computed at the RHF/aug-cc-pVDZ level of theory. Qualitatively similar profiles are obtained for the oxygen atom. Note that the quantity $4 \pi r^2 \rho_a^0(r)$ represents the contribution of the atomic density $\rho_a$ (whose spherical average equals $\rho_a^0$) to the atomic charge $N_a = \bigintss_{0}^{+\infty} 4 \pi r^2 \rho_a^0(r) dr$ at a radial distance $r$.\\

\FloatBarrier
\begin{figure}[!htp]
    \centering
    \includegraphics[width=0.8\textwidth]{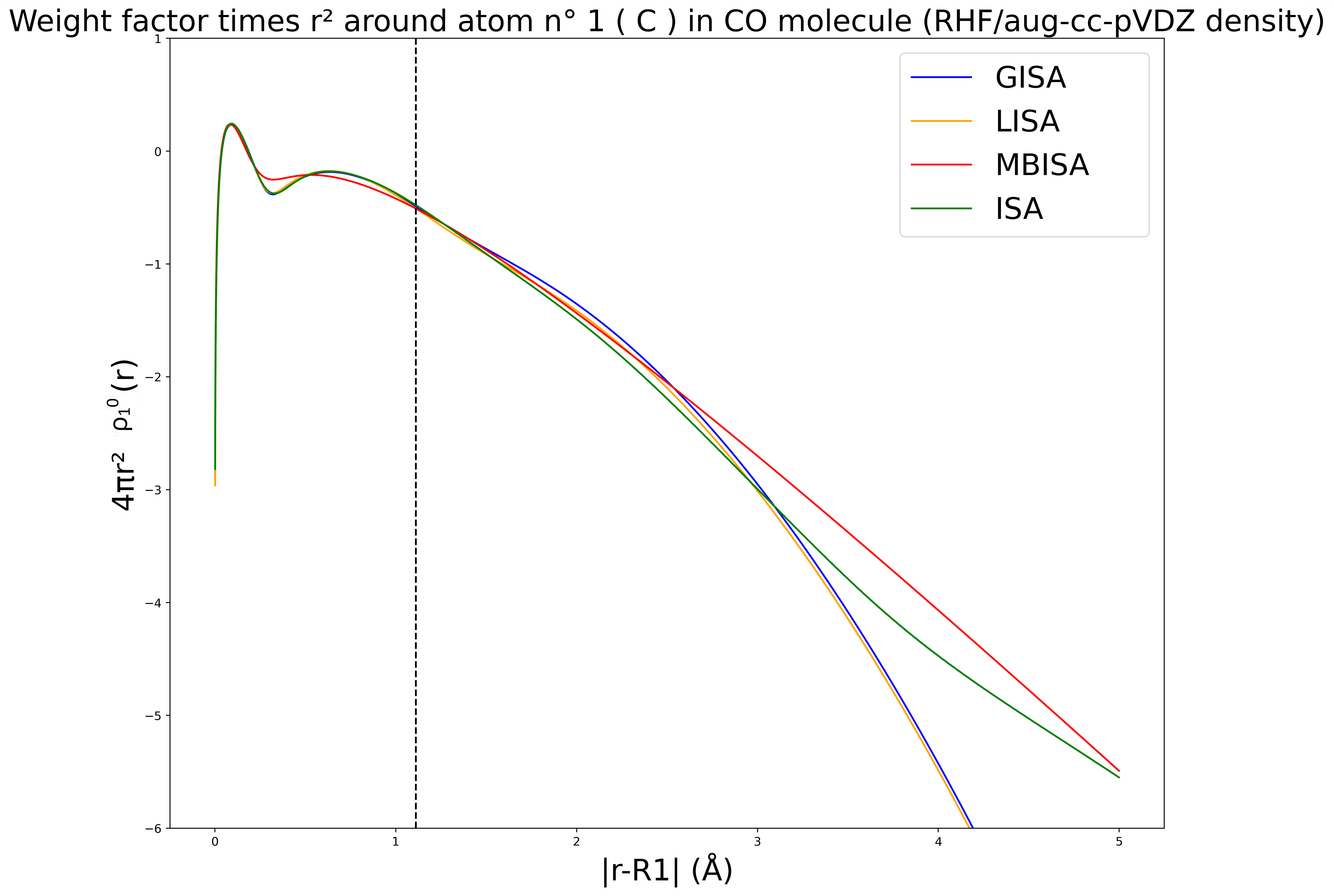}
    \caption{Profiles $r \longmapsto \log\left(4 \pi r^2 \rho_a^0(r) \right)$ for the carbon atom in CO (1.131 \AA{} interatomic distance, RHF/aug-cc-pVDZ level of theory), obtained by GISA, L-ISA, and MB-ISA and ISA methods. $N_{\rm r}=320$ Gauss-Legendre radial (up to R$_{\max}$=8 \AA) and $N_{\rm s}=150$ Gauss-Legendre angular points were used for both the quadratures and the ISA discretization. The black vertical dashed line indicates the position of the oxygen atom.}
   
    \label{fig:comparison_ISA_schemes_CO_RHF_ACCD_atom_C}
\end{figure}
\FloatBarrier

The GISA and L-ISA profiles look very similar but lead to slightly different atomic multipole moments (see Table \ref{tab:local_multipoles_CO_RHF_ACCD_equilibrium_ISA_ALL_METHODS}). The GISA, L-ISA and MB-ISA pro-atomic densities are (visually) very close to the ISA reference near the first maximum of the curve -- corresponding to the core electrons (1s) of the carbon atom. Near the second maximum -- corresponding to the 2s electrons shell -- the GISA and L-ISA profiles are also very close to the ISA reference profile, while the MB-ISA profile features a less pronounced local maximum. The GISA and L-ISA profiles then differ quite notably from the MB-ISA and ISA reference profile in the asymptotic regime.

This proximity of GISA and L-ISA to the ISA reference profile seems rather due to the choice of optimized exponents $\alpha_{a,k}$ for C and O atoms in GISA and L-ISA from Ref. \cite{verstraelen2012conformational}. For the N$_2$ molecule (also studied, results not reported), the results are indeed qualitatively similar to those obtained for CO. In the case of optimized exponents $\alpha_{a,k}$, the GISA and L-ISA Gaussian expansions of the pro-atomic densities variational space perform well to capture correctly the short-range features of $\rho_a^0$. However, this (visual) proximity of GISA and (quantitative) proximity of L-ISA to the reference ISA profile vanishes if the (fixed) exponents $\alpha_{a,k}$ used in GISA and L-ISA are not optimized for one of the atoms in the molecule. Moreover, in the case of non-optimized exponents, only L-ISA recovers qualitatively the reference ISA profile, but not GISA (see below the case of Cl atom in ClO$^{-}$ ion, Figure \ref{fig:comparison_ISA_schemes_ClO_RHF_631Gd_atom_Cl}, compared to Figure \ref{fig:comparison_ISA_schemes_ClO_RHF_631Gd_atom_O} for atom O in the same ion).\\

\FloatBarrier
\begin{table}[!htp]
    \centering
\begin{tabular}{|>{\centering}m{5.cm}||>{\centering}m{2.cm}|>{\centering}m{2.cm}|>{\centering}m{2.cm}|>{\centering}m{2.cm}|}
    \hline
     Multipole moment \\ (atomic units) &  GISA & L-ISA & MB-ISA & ISA  \tabularnewline
     \hline
     \hline
    Charge $q_1=q_C$ & 0.129 & 0.182 & 0.225 &  0.183 \tabularnewline
    \hline
    Dipole $d_z^{1}$ & 0.217 & 0.298  &  0.325 & 0.307 \tabularnewline
    \hline
    Dipole $d_z^{2}$ & -0.015 & 0.015  & 0.079 &  0.008 \tabularnewline
    \hline
     $Q_{xx}^{1}=Q_{yy}^{1}$ & 4.109 & 3.903 & 3.835 & 3.887 \tabularnewline
    \hline
     $Q_{xx}^{2}=Q_{yy}^{2}$ & 3.476 & 3.682 & 3.751 &  3.696 \tabularnewline
    \hline
     $Q_{zz}^{1}$ & 5.100 & 4.823 & 4.789 &  4.786 \tabularnewline
    \hline
     $Q_{zz}^{2}$ & 3.636 & 3.806 & 3.917 &  3.812 \tabularnewline
    \hline
       \end{tabular}
    \caption{Local charges, dipoles and second-order moments carried by C (atom 1) and O (atom 2) in CO (1.131 \AA{} interatomic distance, RHF/aug-cc-pVDZ level of theory) computed by GISA, L-ISA, MB-ISA and ISA methods. 320 Gauss-Legendre radial (up to R$_{max}$=8 \AA) and 150 Gauss-Legendre angular points were used for both the quadratures and for the ISA discretization.}
    \label{tab:local_multipoles_CO_RHF_ACCD_equilibrium_ISA_ALL_METHODS}
\end{table}
\FloatBarrier

The local multipole moments associated to the different ISA profiles of Figure \ref{fig:comparison_ISA_schemes_CO_RHF_ACCD_atom_C} are reported Table \ref{tab:local_multipoles_CO_RHF_ACCD_equilibrium_ISA_ALL_METHODS}. The results of the L-ISA scheme are overall the closest to the ISA reference, followed by MB-ISA and GISA, which appears less satisfactory (and for which the atomic monopole differs by more than 30 \% from the reference ISA monopole).

\paragraph*{ClO$^{-}$ ion.\\}

Figures \ref{fig:comparison_ISA_schemes_ClO_RHF_631Gd_atom_Cl} and \ref{fig:comparison_ISA_schemes_ClO_RHF_631Gd_atom_O} display the $r \mapsto \log\left(4 \pi r^2 \rho_a^0(r) \right)$ profiles of GISA, L-ISA, MB-ISA and ISA solutions for the chlorine and oxygen atoms respectively, using a molecular density $\rho$ computed at the RHF/6-31G(d) level of theory. Optimized exponents $\alpha_{a,k}$ for Cl atom have not been reported in the literature to our knowledge. We first tried the empirical rule \eqref{eq:empirical_rule_exponents} but finally assigned six shells (in GISA and L-ISA) to the chlorine atom with the same (optimized) exponents as for the oxygen atom, as it yielded slightly better (although not satisfactory, see below) results for the GISA profiles and multipole moments. The solution of the MB-ISA method was found not to depend on this initial choice of exponents for the Cl atom (empirical rule  \eqref{eq:empirical_rule_exponents} or same exponents as the O atom).

Contrary to the case of neutral molecules (CO, N$_2$) with optimized exponents for both atoms, the pro-atomic densities profiles of GISA, L-ISA and MB-ISA quantitatively differ from the reference ISA profile, with some common features (\textit{e.g.} three local minima for the L-ISA and ISA reference schemes, \textit{vs.} only two for MB-ISA). The GISA profile ends up with an nonphysical, too large (positive) charge (1.94 a.u.) on Cl atom (-2.84 a.u. on O atom) corresponding to a too quick (Figure \ref{fig:comparison_ISA_schemes_ClO_RHF_631Gd_atom_Cl}) or (resp.) too slow (Figure \ref{fig:comparison_ISA_schemes_ClO_RHF_631Gd_atom_O}) decrease in the Cl (resp. O) atomic density profile, despite a good short-range accordance with the ISA reference profile in the case of the O atom (using its optimized Gaussian exponents \cite{verstraelen2012conformational}).The other GISA atomic multipole moments are also either much too large or largely differing (more than 100 \%) from the ISA reference multipoles. The MB-ISA profile is associated to local multipole moments in fair accordance with the ISA reference multipoles (see Table \ref{tab:local_multippoles_ClO_RHF_6-31Gd_equilibrium_ISA_ALL_METHODS}) despite differing qualitatively quite notably from the ISA reference profile, in particular in the asymptotic regime. The latter is a feature rather due to the expansion on Slater functions, leading to a decrease at a smaller rate.

Quite notably, unlike the GISA scheme, the L-ISA method is also associated to local multipole moments in fair accordance with the ISA reference multipoles (Table \ref{tab:local_multippoles_ClO_RHF_6-31Gd_equilibrium_ISA_ALL_METHODS}), similarly to the MB-ISA method, although the latter optimizes on the (Slater) exponents while the L-ISA uses fixed (non-optimized) Gaussian exponents. The L-ISA scheme, which has been shown here to be mathematically more grounded than GISA, seems to be more robust with respect to the Gaussian exponent choices and yields more physical (and closer to the ISA reference profile) results, although this assumption would have to be confirmed on larger test sets.

\FloatBarrier
\begin{figure}[!htp]
    \centering
    \includegraphics[width=0.8\textwidth]{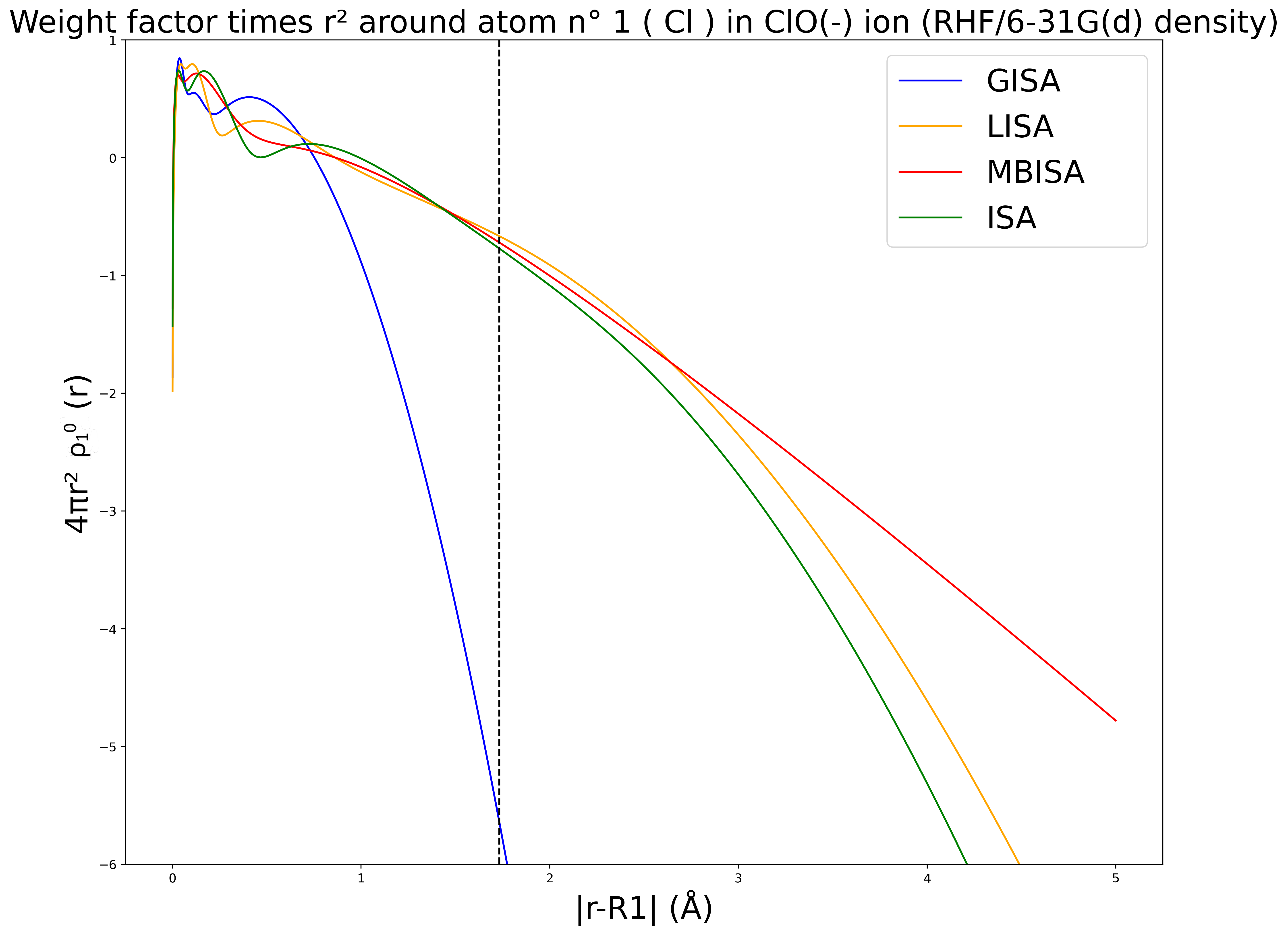}
    \caption{Profiles $r \longmapsto \log\left(4 \pi r^2 \rho_a^0(r) \right)$ for the chlorine atom in ClO$^{-}$ ion (1.733 \AA{} interatomic distance, RHF/6-31G(d) level of theory), obtained by GISA, L-ISA, and MB-ISA and ISA methods. $N_{\rm r}=500$ Gauss-Legendre radial (up to R$_{\max}$=8 \AA) and $N_{\rm s}=250$ Gauss-Legendre angular points were used for both the quadratures and the ISA discretization. The black vertical dashed line indicates the position of the oxygen atom.}
    \label{fig:comparison_ISA_schemes_ClO_RHF_631Gd_atom_Cl}
\end{figure}
\FloatBarrier

\FloatBarrier
\begin{figure}[!htp]
    \centering
    \includegraphics[width=0.8\textwidth]{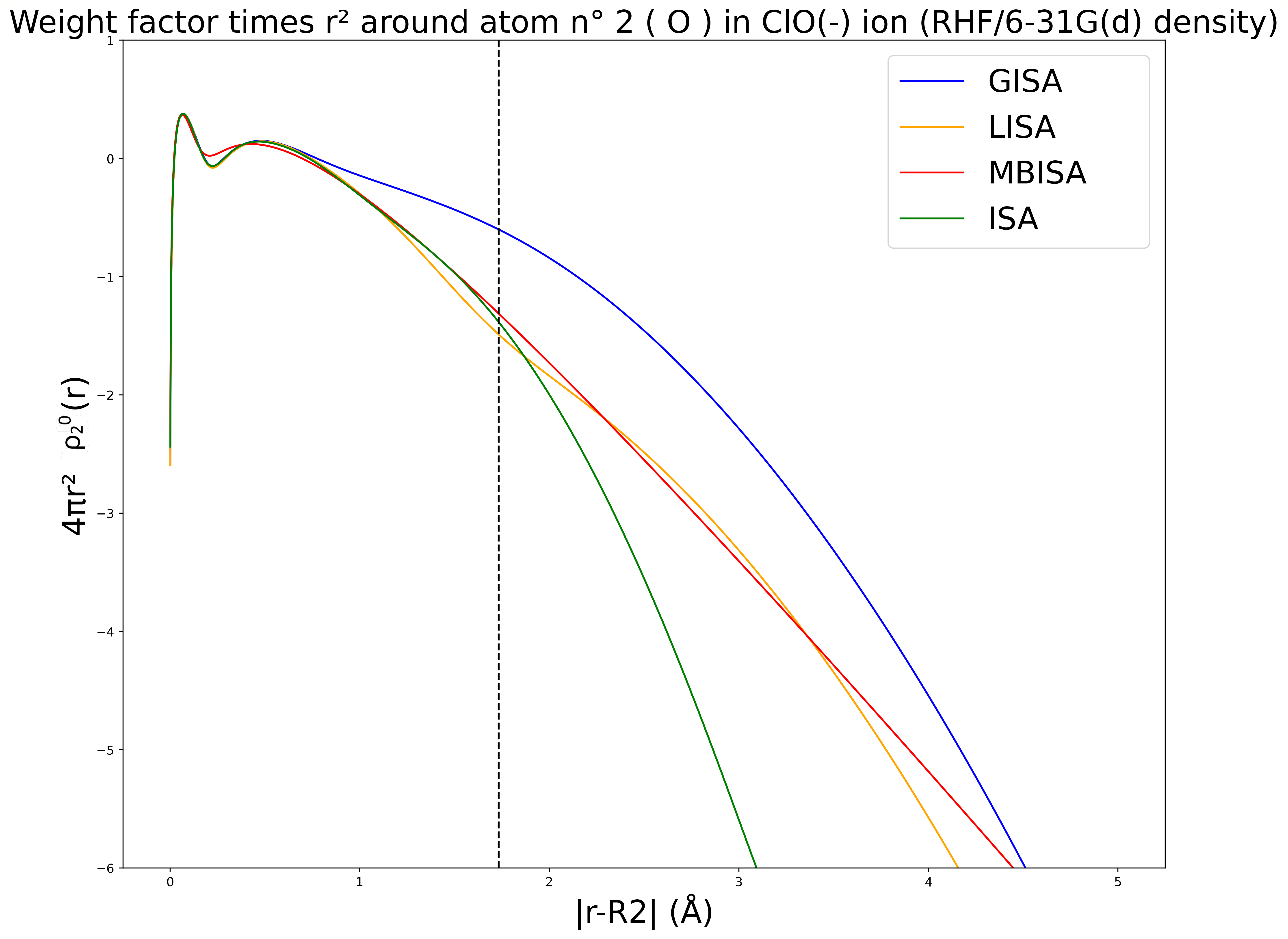}
    \caption{Profiles $r \longmapsto \log\left(4 \pi r^2 \rho_a^0(r) \right)$ for the oxygen atom of ClO$^{-}$ ion (1.733 \AA{} interatomic distance, RHF/6-31G(d) level of theory), obtained by GISA, L-ISA, and MB-ISA and ISA methods. $N_{\rm r}=500$ Gauss-Legendre radial (up to R$_{\max}$=8 \AA) and $N_{\rm s}=250$ Gauss-Legendre angular points were used for both the quadratures and the ISA discretization. The black vertical dashed line indicates the position of the chlorine atom.}
    \label{fig:comparison_ISA_schemes_ClO_RHF_631Gd_atom_O}
\end{figure}
\FloatBarrier

\FloatBarrier
\begin{table}[!htp]
    \centering
\begin{tabular}{|>{\centering}m{5.cm}||>{\centering}m{2.cm}|>{\centering}m{2.cm}|>{\centering}m{2.cm}|>{\centering}m{2.cm}|}
    \hline
     Multipole moment (atomic units) &  GISA & L-ISA & MB-ISA & ISA \tabularnewline
     \hline
     \hline
    Charge $q_1=q_{Cl}$ &  1.941 & -0.368  &  -0.328 & -0.342  \tabularnewline
    \hline
    Dipole $d_z^{1}$ &  -1.024 & -0.117  &  -0.198 & -0.232 \tabularnewline
    \hline
    Dipole $d_z^{2}$ &  -6.569 & 0.086  &  0.033 &  0.115 \tabularnewline
    \hline
     $Q_{xx}^{1}=Q_{yy}^{1}$ & 4.455  &  10.460 & 10.333  &  10.353 \tabularnewline
    \hline
     $Q_{xx}^{2}=Q_{yy}^{2}$ &  10.581 &  4.576 &  4.703 &  4.682 \tabularnewline
    \hline
     $Q_{zz}^{1}$ & 4.176 & 9.787 &  9.567 & 9.364 \tabularnewline
    \hline
     $Q_{zz}^{2}$ &   28.460 & 4.026  &  4.159 & 3.983 \tabularnewline
    \hline
       \end{tabular}
    \caption{Local charges, dipoles and second-order moments carried by Cl (atom 1) and O (atom 2) in ClO$^{-}$ (1.733 \AA{} interatomic distance, RHF/6-31G(d) level of theory), obtained by GISA, L-ISA, and MB-ISA and ISA methods. $N_{\rm r}=500$ Gauss-Legendre radial (up to R$_{\max}$=8 \AA) and $N_{\rm s}=250$ Gauss-Legendre angular points were used for both the quadratures and the ISA discretization.}
    \label{tab:local_multippoles_ClO_RHF_6-31Gd_equilibrium_ISA_ALL_METHODS}
\end{table}
\FloatBarrier

\subsection{Dissociation of diatomic molecules }

In this section, we apply the ISA multi-center decomposition schemes implemented in this work to the study of the dissociation of LiH. We consider ground-state densities obtained by Full CI calculations in a small basis sets (computed with PySCF \cite{sun2018pyscf,sun2020recent}), for different inter-atomic distances, and extract the corresponding one-body density matrices. We then compute the corresponding ISA atomic charges and dipoles.
The results are reported in Figure~\ref{fig:LiH_dissociation_curves}.

\FloatBarrier
\begin{figure}[htp]
  \centering
  \begin{tabular}{cc}
    \includegraphics[width=70mm]{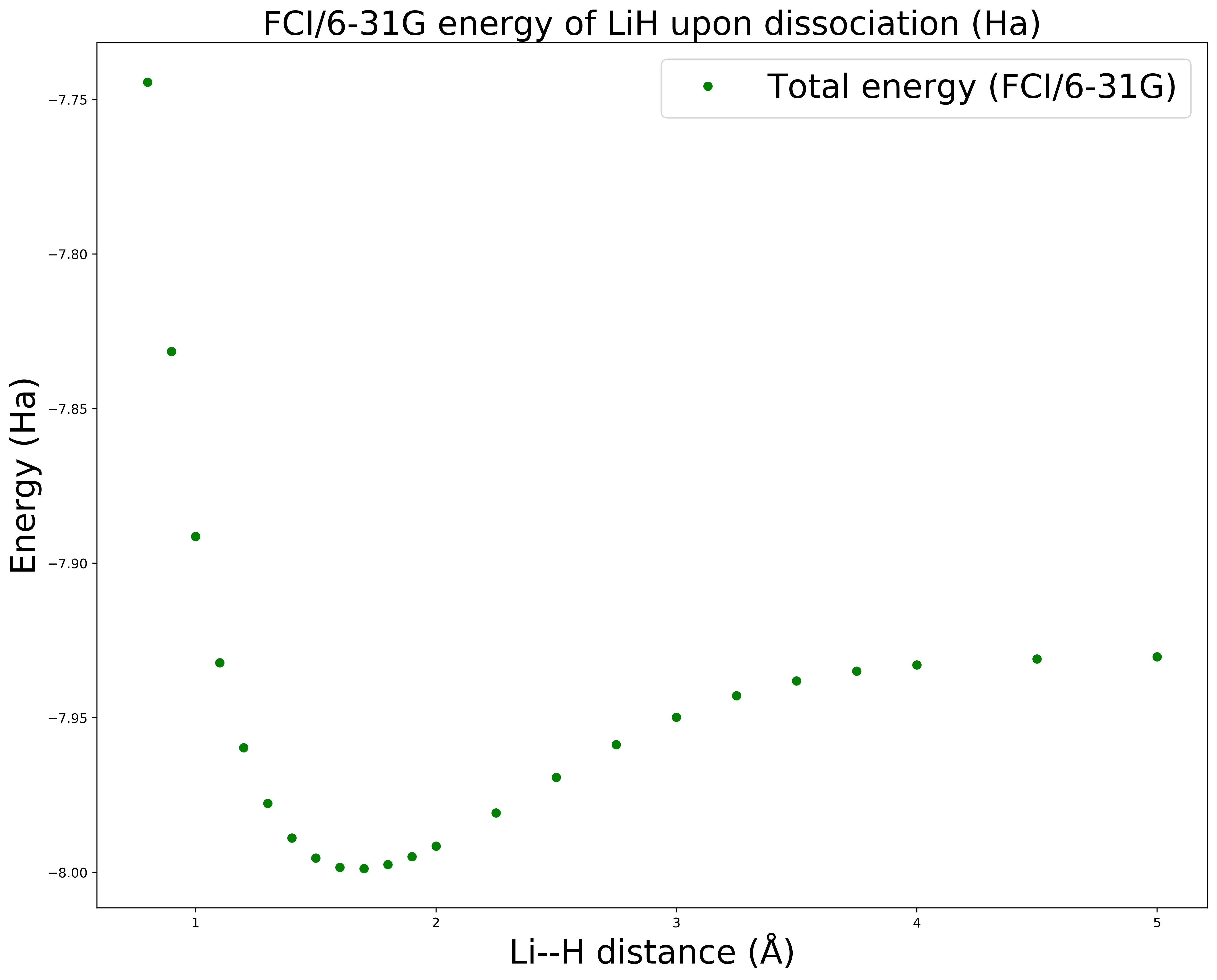}&
   \includegraphics[width=70mm]{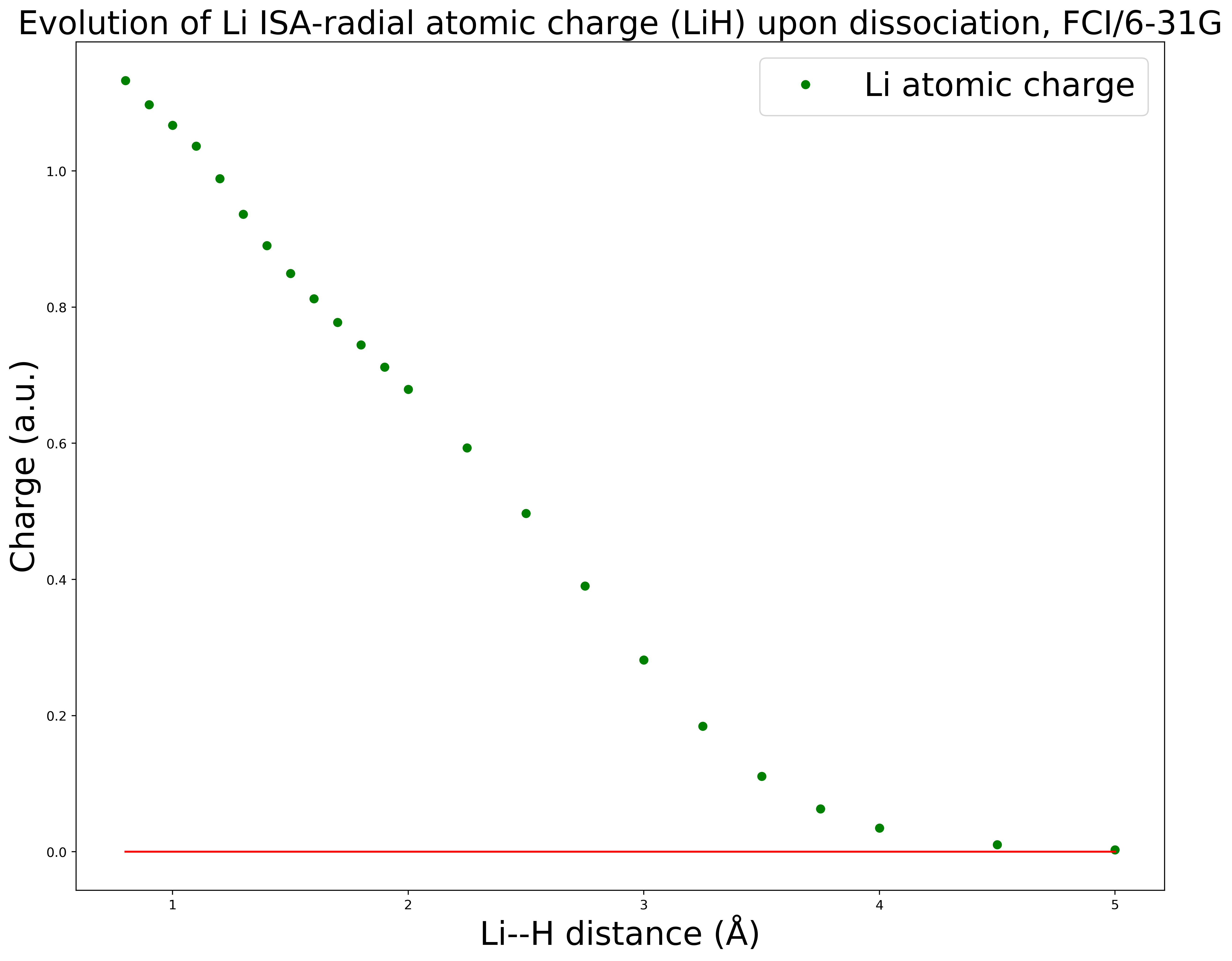}\\
    \includegraphics[width=76mm]{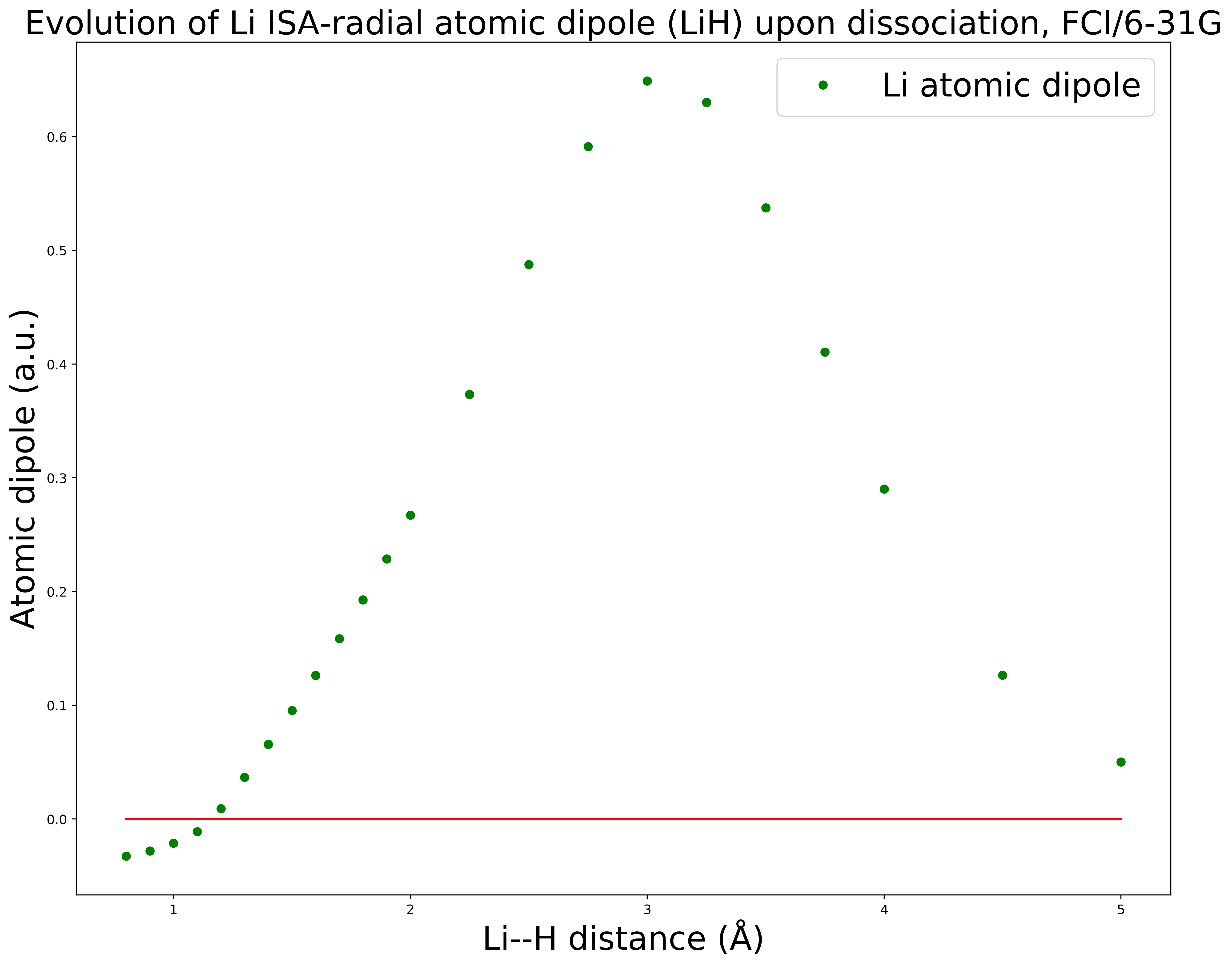}&
    \includegraphics[width=75mm]{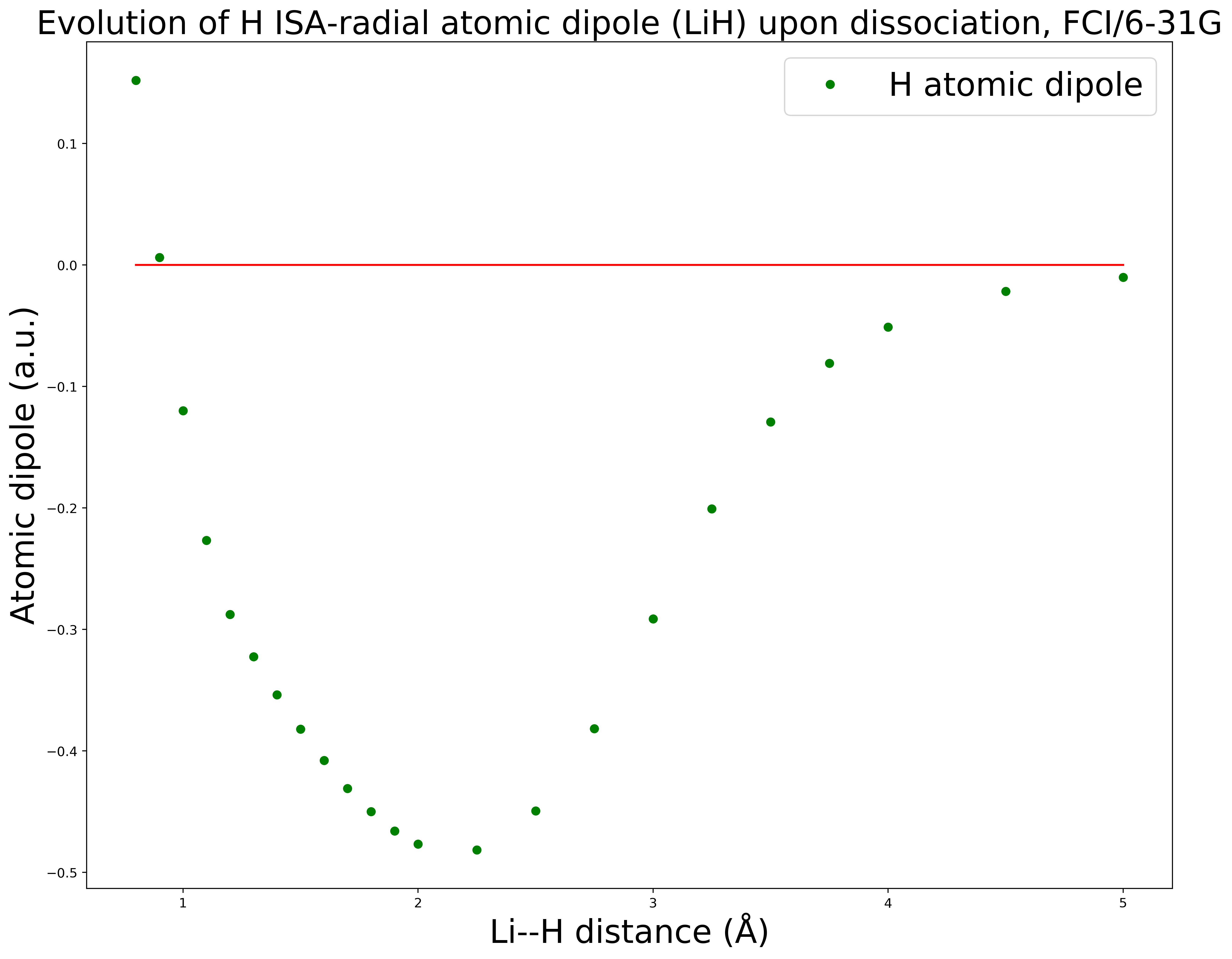}\\
  \end{tabular}
  \caption{Quantities of interest for LiH as functions of the interatomic distance, with a density obtained at the FCI/6-31G level of theory.  Top-left : ground state energy (Ha). Top-right : ISA atomic charge on Li. Bottom: ISA atomic dipoles along the dimer axis on Li (left panel) and H (right panel).}
  \label{fig:LiH_dissociation_curves}
\end{figure}
\FloatBarrier

As expected, the ISA atomic charge tends to zero in the dissociation limit, while the Li--H bond is very polarized in the covalent regime (with atomic charges of order 0.80 a.u. at the equilibrium distance 1.70 \AA). The ISA Li and H atomic dipoles also tend to zero in the dissociation limit, but in a non-monotonic manner, with a minimum (resp. maximum) for H (resp. Li) atomic dipole at about 2.2 \AA{} (resp. 3.0 \AA).

\subsection{Basis set convergence of DMA and ISA}

In this section, we compare some typical features of the DMA and ISA methods, such as the convergence of local multipole moments with increasing basis sets. Figure \ref{fig:evolution_DMA_ISA_monopoles_water_basis_set} displays the evolution of the local atomic charges on H and O atoms of H$_2$O computed with DMA (using both Stone's and Vign\'e-Maeder's redistribution strategies) and ISA (computed with Hipart code, now known as Horton \cite{HORTON}), from a molecular density computed at the DFT/PBE0/(aug)-cc-pVXZ (X $\in \left\lbrace D, T, Q, 5 \right\rbrace$) level of theory \cite{Adamo1999}. ISA atomic charges stabilize up to a few \% for basis sets larger than aug-cc-pVDZ, the DMA atomic charges do not stabilize with increasing basis set size (e.g. four-fold increase of atomic charge on the oxygen atom). This was already known for DMA atomic charges computed with Stone's redistribution strategy and attributed to the presence of diffuse functions in the basis set, and had lead to the hybrid real-space and basis-space GDMA method \cite{stone2005distributed}. We see here that the more balanced Vign\'e-Maeder redistribution strategy \cite{Vigne-Maeder1988} leads to a qualitatively similar non-converging behavior.
A similar behavior is observed for atomic dipoles (Figure \ref{fig:evolution_DMA_ISA_dipoles_water_basis_set}), with ISA dipoles stabilizing up to 10 \% beyond aug-cc-pVDZ basis.\\

\FloatBarrier
\begin{figure}[!htp]
    \centering
    \includegraphics[width=14.5cm]{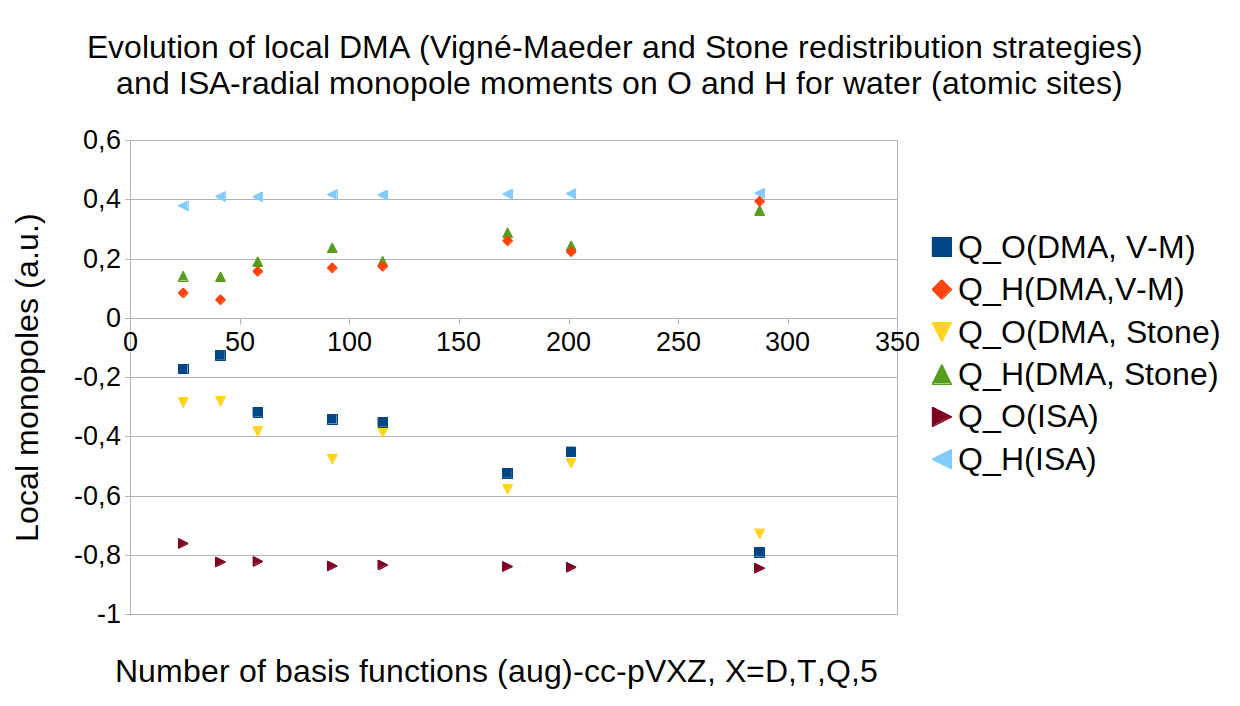}
    \caption{Local atomic charges on hydrogen and oxygen atoms of the water molecule with increasing basis set size, computed from DMA (Stone's and Vign\'e-Maeder's redistribution strategies) and ISA. The number of basis functions refers to the whole molecule (two hydrogen atoms and oxygen atoms).}
    \label{fig:evolution_DMA_ISA_monopoles_water_basis_set}
\end{figure}
\FloatBarrier

\FloatBarrier
\begin{figure}[!htp]
    \centering
    \includegraphics[width=14.5cm]{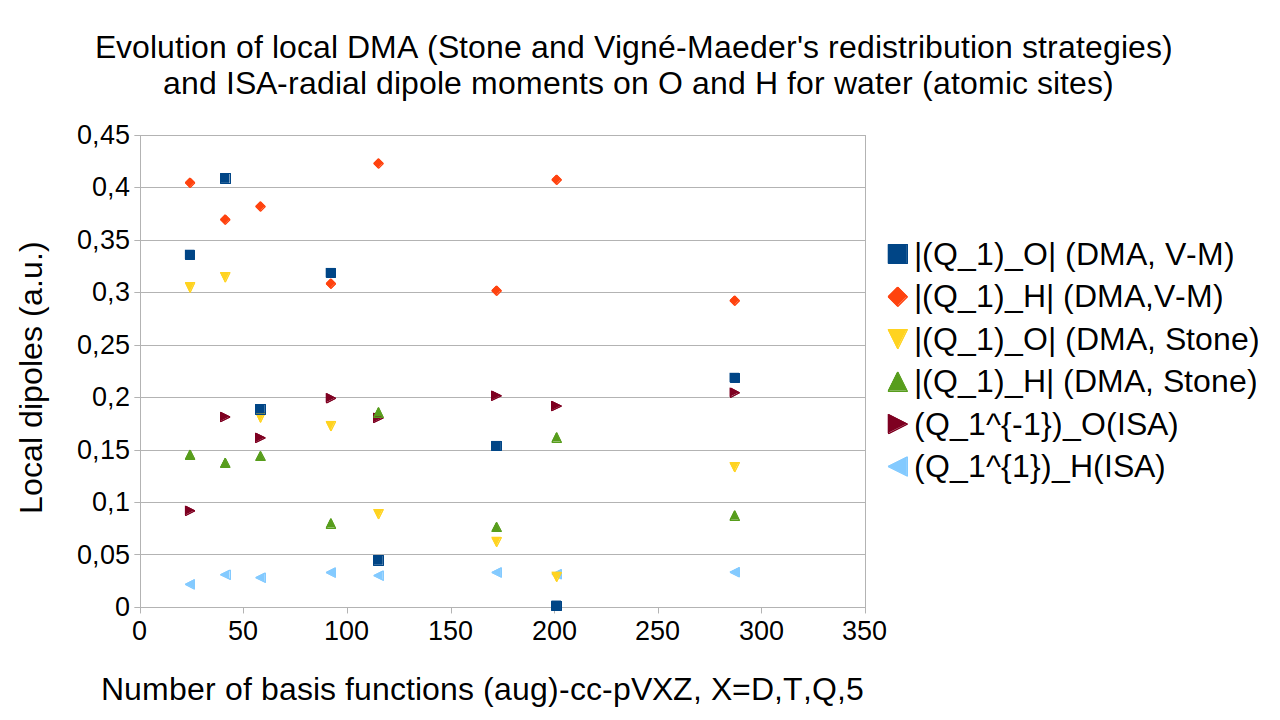}
    \caption{Local atomic dipoles on hydrogen and oxygen atoms of the water molecule with increasing basis set size, computed from DMA (Stone's and Vign\'e-Maeder's redistribution strategies) and ISA. The number of basis functions refers to the whole molecule (two hydrogen atoms and oxygen atoms).}
    \label{fig:evolution_DMA_ISA_dipoles_water_basis_set}
\end{figure}
\FloatBarrier

\section{Proofs}\label{sec:proofs}

Let us first elaborate on the assumptions on $\rho$ and the $\rho^0_a$'s made in the lemmata, propositions and theorems to be proved. As already mentioned previously, electronic densities of molecules and clusters are continuous, positive, functions on $\R^3$, decaying exponentially fast at infinity. We can therefore assume without loss of generality that $\rho \in X_+$, and this implies in particular that $\rho \log\rho \in L^1(\R^3)$. For the same reasons, it makes sense to restrict ourselves to proatom densities $\rho_a^0$ belonging to $X_+$.  In addition, if the $\rho_a$'s are nonnegative measurable functions satisfying $\sum_{a=1}^M \rho_a(\cdot-\bR_a)=\rho$ for some $\rho \in X_+$, then each $\rho_a$ satisfies $0 \le \rho_a \le \rho(\cdot + \bR_a)$, and therefore belongs to $X_+$. This justifies {\it a posteriori} the condition $\bm\rho \in X_+^M$ in the definition \eqref{eq:KrhoR} of $\bcK_{\rho,\bR}$.

\medskip

The following technical remark is important as it allows one to justify several algebraic manipulations in the proofs below.

\begin{remark} \label{rem:integration_entropy} Let $\rho \in X_+$. 
For $\bm\rho=(\rho_a)_{1 \le a \le M} \in \cK_{\rho,\bR}$, we have $\rho_a \in X_+$, and therefore $\rho_a \log \rho_a \in L^1(\R^3)$. As a consequence, we have for all $\bm\rho=(\rho_a)_{1 \le a \le M} \in \cK_{\rho,\bR}$ and $\bm\rho^0=(\rho_a^0)_{1 \le a \le M} \in (X_+^r)^M$,
\begin{align} \label{eq:justif}
S(\bm\rho|\bm\rho^0) &= \sum_{a=1}^M  \int_{\R^3} \rho_a \log \left( \frac{\rho_a}{\rho_a^0} \right)   \nonumber \\
&= \sum_{a=1}^M  \int_{\R^3} \rho_a \left( \log \rho_a - \log \rho_a^0 \right)
\qquad \mbox{\rm (in agreement with \eqref{eq:convention_entropy})}
\nonumber \\
&=
\sum_{a=1}^M  \int_{\R^3} \left( \underbrace{\rho_a \log \rho_a}_{\in L^1(\R^3)} - \underbrace{(\rho_a \log(\rho_a^0))_+}_{\in L^1(\R^3) \cap L^\infty(\R^3)} + \underbrace{(\rho_a \log(\rho_a^0))_-}_{\ge 0}  \right), 
\end{align}
where $x_+=\max(x,0)$ and $x_-=-\min(x,0)$ denote the positive and negative parts of $x \in \R$ ($x_+,x_- \ge 0$ and $x=x_+-x_-$). In addition, the function $(\rho_a \log(\rho_a^0))_+$ has compact support since $\rho_a^0$ vanishes at infinity. It follows that $S(\bm\rho|\bm\rho^0) < + \infty$ if and only if $\rho_a \log(\rho_a^0) \in L^1(\R^3)$ for all $1 \le a \le M$, and that we then have (using Fubini theorem)
\begin{align} \label{eq:rewrite_S}
S(\bm\rho|\bm\rho^0) &=\sum_{a=1}^M  \int_{\R^3} \rho_a \log \rho_a - \sum_{a=1}^M \int_{\R^3} \rho_a \log \rho_a^0 =\sum_{a=1}^M  \int_{\R^3} \rho_a \log \rho_a - \sum_{a=1}^M \int_{\R^3} \langle \rho_a \rangle_s \log \rho_a^0.
\end{align}
\end{remark}

\subsection{Proof of Lemma~\ref{lem:Step1}}

Let us define
\begin{equation}\label{eq:def_rho_a^*}
 \rho_a^*(\bm r) := \left| \begin{array}{ll} \dps
 \frac{\rho_a^0(\bm r)}{\rho^0(\bm r + \bR_a)}\rho(\bm r + \bR_a) & \quad \mbox{if } \rho_a^0(\bm r) > 0, \\
 0 & \quad  \mbox{if } \rho_a^0(\bm r) = 0 \end{array} \right.
 \end{equation}
 (recall that $\rho^0(\br):=\sum_{a=1}^M \rho^0_a(\br-\bR_a)$), and $\bm\rho^*:=(\rho_a^*)_{1\le a \le M}$ as an Ansatz for the minimizer of \eqref{eq:min_on_KrhoR}. Let 
 \begin{align*}
 &B_a^0:=\{\br \in \R^3 \; | \;  \rho_a^0(\bm r) = 0\},  \quad B^0:=\{\br \in \R^3 \; | \;  \rho^0(\bm r) = 0\},  
 \quad B=\{\br \in \R^3 \; | \;  \rho(\bm r) = 0\}.
 \end{align*}
 It follows from the hypothesis \eqref{eq:hyp_G0} (i.e. since $s_{\rm KL}(\rho|\rho^0) < +\infty$) and the conventions \eqref{eq:convention_entropy} that $B^0 \subset B$. 
 In addition, we have $B^0=\bigcap_{a=1}^M (\bR_a+B_a^0)$, so that if $\br \in B^0$, then for all $1 \le a \le M$, $\rho_a^0(\br-\bR_a)=0$, and therefore, using \eqref{eq:def_rho_a^*}, $\rho_a^*(\br-\bR_a)=0$. Thus, on $B^0$, there holds $\sum_{a=1}^M \rho_a^*(\cdot-\bR_a)=0=\rho$ as $B^0 \subset B$. Besides, we have
 $$
 \forall 1 \le a \le M, \quad \forall \br \in \R^3 \setminus B^0, \quad \rho_a^*(\br-\bR_a)  =
 \frac{\rho_a^0(\bm r-\bR_a)}{\rho^0(\bm r)}\rho(\bm r).
 $$ 
 Thus, there holds that $\sum_{a=1}^M \rho_a^*(\,\cdot\,-\bR_a)=\rho$ on $\R^3 \setminus B^0$ as well, and therefore on the entire space~$\R^3$. Since the $\rho_a^*$'s are obviously nonnegative, bounded (by $\rho(\br+\bR_a)$), integrable, with finite first moments, and vanishing at infinity, it holds that $\bm\rho^* \in \cK_{\rho,\bR}$.  Now, we have 
\begin{align*}
S(\bm\rho^*|\bm\rho^0)&=
\sum_{a=1}^M  \int_{\R^3} \rho_a^* \log \left( \frac{\rho_a^*}{\rho_a^0} \right) 
=
\sum_{a=1}^M  \int_{\R^3} \rho_a^*(\,\cdot\,-\bR_a) \log \left( \frac{\rho_a^*(\,\cdot\,-\bR_a)}{\rho_a^0(\,\cdot\,-\bR_a)} \right)  \\
&=\sum_{a=1}^M  \int_{\R^3} \rho_a^*(\,\cdot\,-\bR_a) \log \left( \frac{\rho}{\rho^0} \right) =   
\int_{\R^3} \rho \log \left( \frac{\rho}{\rho^0} \right)  = s_{\rm KL}(\rho|\rho^0) < +\infty.
\end{align*}
All these equalities can be rigorously justified by using arguments similar as in Remark~\ref{rem:integration_entropy} and the infimum in \eqref{eq:min_on_KrhoR} therefore has a finite value. For all $\bm\rho=(\rho_a)_{1 \le a \le M} \in \cK_{\rho,\bR}$ such that $S(\bm\rho|\bm\rho^0) < +\infty$, we have in addition, using the fact that $\rho = \sum_{a=1}^M \rho_a(\,\cdot\, - \bR_a)$,
\begin{align*}
S(\bm\rho^*|\bm\rho^0) & = \int_{\R^3} \rho \log \left( \frac{\rho}{\rho^0} \right)  = \sum_{a=1}^M  \int_{\R^3} \rho_a(\,\cdot\,-\bR_a) \log  \left( \frac{\rho}{\rho^0} \right) 
= \sum_{a=1}^M  \int_{\R^3} \rho_a \log \left( \frac{\rho(\,\cdot+\bR_a)}{\rho^0(\,\cdot+\bR_a)} \right).
\end{align*}
Therefore, using again arguments as in Remark~\ref{rem:integration_entropy} to justify each equality, we get
\begin{align*}
S( \bm \rho| \bm \rho^0) - S( \bm \rho^*| \bm \rho^0)&= \sum_{a=1}^M  \int_{\R^3}  \rho_a \log \left( \frac{\rho_a}{\rho_a^0} \right) -\sum_{a=1}^M  \int_{\R^3} \rho_a \log \left( \frac{\rho(\,\cdot+\bR_a)}{\rho^0(\,\cdot+\bR_a)} \right) \\
&= \sum_{a=1}^M  \int_{\R^3}  \rho_a \log \left( \frac{\rho_a \rho^0(\,\cdot+\bR_a)}{\rho_a^0 \rho(\,\cdot+\bR_a)} \right) =
\sum_{a=1}^M  \int_{\R^3}  \rho_a \log \left( \frac{\rho_a}{\rho_a^*} \right) = S(\bm\rho|\bm\rho^*) \ge 0,
\end{align*}
for all $\bm\rho \in \bcK_{\rho, \bR}$, which proves that $\rm\rho^*$ is a minimizer of \eqref{eq:min_on_KrhoR}. As this problem is strictly convex, it is the unique one.

\subsection{Proof of Lemma~\ref{lem:Step2_ISA}}

Obviously, $\langle \rho_a\rangle_s \in \cK^0_{\rm ISA}=X^r_+$ and $\int_{\R^3}  \langle \rho_a\rangle_s  = \int_{\R^3} \rho_a$, so that $\langle \rho_a\rangle_s$ is in the minimization set of \eqref{eq:pb_Lemma_1}. Reasoning as in Remark~\ref{rem:integration_entropy}, we obtain that 
$$
s_{\rm KL}(\rho_a|\langle \rho_a\rangle_s) =  \int_{\R^3} \rho_a \log(\rho_a) -    \int_{\R^3} \langle \rho_a \rangle_s  \log(\langle \rho_a \rangle_s) < +\infty,
$$
and that, for all $\rho^0_a \in \cK^0_{\rm ISA}$ such that $s_{\rm KL}(\rho_a|\rho_a^0) < + \infty$, it holds
\begin{align*}
s_{\rm KL}(\rho_a| \rho_a^0) =  \int_{\R^3} \rho_a \log(\rho_a) -    \int_{\R^3}  \langle\rho_a\rangle_s  \log(\rho_a^0).
\end{align*}
Therefore, for all $\rho^0_a \in \cK^0_{\rm ISA}$ such that $s_{\rm KL}(\rho_a|\rho_a^0) < + \infty$, we have
\begin{align*}
s_{\rm KL}(\rho_a| \rho_a^0) &= s_{\rm KL}(\rho_a|\langle \rho_a\rangle_s) +   \int_{\R^3} \langle \rho_a \rangle_s  \log(\langle \rho_a \rangle_s)-    \int_{\R^3}  \langle\rho_a\rangle_s  \log(\rho_a^0) \\
&= s_{\rm KL}(\rho_a|\langle \rho_a\rangle_s) + s_{\rm KL}(\langle \rho_a\rangle_s|\rho_a^0) \ge s_{\rm KL}(\rho_a|\langle \rho_a\rangle_s).
\end{align*}
This implies that $\langle \rho_a\rangle_s$ is a minimizer of~\eqref{eq:pb_Lemma_1} and the infimum is finite.
The set $\cK^0_{\rm ISA}$ is a non-empty closed convex subset of the vector space $X$, and the map ${\cal J}_0:=\cK^0_{\rm ISA} \ni  \rho_a^0 \mapsto s_{\rm KL}(\rho_a | \rho_a^0) \in \R \cup \{+\infty\}$ is convex, and strictly convex on the convex set on which it takes finite values. The uniqueness of the solution to~\eqref{eq:pb_Lemma_1} is a direct consequence of this strict convexity property.

\subsection{L-ISA counterpart of Lemma~\ref{lem:Step2_ISA}}
\label{ssec:L-ISA-Lemma}

For $\bm \lambda:=(\lambda_a)_{1\leq a \leq M}\in \mathbb{R}^M$, $\bm \mu:=(\mu_a)_{1\leq a \leq M} \in \mathbb{R}^M$, and $\bm \rho:=(\rho_a)_{1\leq a \leq M} \in X^M$, we denote by $\bm \lambda \odot \bm \mu \in \mathbb{R}^M$ and $\bm \lambda \odot \bm \rho \in X^M$ the Kronecker products defined by 
$$
[\bm \lambda \odot \bm \mu]_a = \lambda_a \, \mu_a \quad \mbox{and} \quad 
[\bm \lambda \odot \bm \rho]_a = \lambda_a \, \rho_a.
$$

\begin{lemma}\label{lem:prelimG} Let $\bm\rho = (\rho_a)_{1 \le a \in M} \in X_+^M$ such that for all $1 \le a \le M$, there exists $\rho_a^0 \in \cK^0_{z,\rm L-ISA}$ such that $\int_{\R^3} \rho_a^0=\int_{\R^3} \rho_a$ and $s_{\rm KL}(\rho_a|\rho_a^0) < +\infty$. Then, the problem
\begin{equation}\label{eq:pb_Lemma_1bis}
\mathop{\inf}_{\bm \rho^0\in \bcK^0_{\rm L-ISA}, \; \bcN(\bm \rho^0) = \bcN(\bm \rho)} S(\bm \rho |\bm \rho^0)
\end{equation}
has a unique minimizer $\bm G^0_{\bcK^0_{\rm L-ISA}}(\bm \rho)$. For all $\bm \lambda \in (\mathbb{R}_+)^M$, it holds
\begin{equation}\label{eq:rescaled}
\bm G^0_{\bcK^0_{\rm L-ISA}}(\bm\lambda \odot \bm \rho) = \bm\lambda \odot \bm G^0_{\bcK^0_{\rm L-ISA}}(\bm \rho).
\end{equation}
\end{lemma}

Let us point out that problem (\ref{eq:pb_Lemma_1bis}) can be decomposed as $M$ independent problems of the form
$$
\mathop{\inf}_{\rho_a^0\in \cK^0_{z_a,\rm L-ISA}, \; \int_{\mathbb{R}^3}\rho_a^0 = \int_{\mathbb{R}^3}\rho_a} s_{\rm KL}(\rho_a |\rho_a^0)
$$
for all $1\leq a\leq M$ as in (\ref{eq:comp_G}).

\begin{proof} The set $\{\bm \rho^0\in \bcK^0_{\rm L-ISA} \, | \, \bcN(\bm \rho^0) = \bcN(\bm \rho)\}$ is non-empty and convex, and is included in a finite-dimensional vector subspace of $X$.  The existence of a solution~\eqref{eq:pb_Lemma_1bis} follows from the compactness of $\bcK^0_{\rm ISA}$ (for the strong topology of $X^M$) and the strong lower-semicontinuity of the functional ${\cal J}_0$ introduced in the previous section. The second assertion of the lemma is a consequence of the fact that (i) for all $\bm \lambda \in \mathbb{R_+}^M$, $\bm \rho \in X^M$, and $\bm \rho^0 \in \cK^0_{\rm L-ISA}$, $\bm\lambda \odot \bm \rho^0 \in \cK^0_{\rm L-ISA}$, $\bcN(\bm\lambda \odot \bm \rho^0) = \bcN(\bm\lambda \odot \bm \rho)$ and (ii) $s_{\rm KL}(\lambda f|\lambda g)=\lambda s_{\rm KL}(f|g)$ for all $\lambda \in \R_+$ and $f,g \in X_+$.
\end{proof}

\subsection{Topological properties of the set of admissible AIM decompositions}
\label{ssec:topology}

The following lemma provides some information about the topology of $\bcK_{\rho,\bR}$ which will be useful for our analysis. Recall that $\cM_{\rm b}(\R^3)$ is the Banach space of the bounded (signed) Radon measures on~$\R^3$ endowed with the total variation norm and that $\cM_{\rm b}(\R^3)$ is the dual of the Banach space $C_0(\R^3)$ of the real-valued continuous functions going to zero at infinity, and that a sequence $(\mu_n)_{n \in \N}$ of elements of $\cM_{\rm b}(\R^3)$ weakly-* converge to some $\mu \in \cM_{\rm b}(\R^3)$ if and only if 
 $$
 \forall f \in C_0(\R^3), \quad \int_{\R^3} f \, d\mu_n \mathop{\longrightarrow}_{n \to +\infty} \int_{\R^3} f \, d\mu.
 $$

 \begin{lemma} \label{lem:topology_KrhoR} 
 Let $\rho \in X_+$ and $\bR=(\bR_1,\ldots,\bR_M) \in \R^{3M}$. The set $\bcK_{\rho,\bR}$ of admissible AIM decompositions  defined by~\eqref{eq:KrhoR}  is nonempty and convex. It is also bounded and closed for
the norm topology of the Banach space $X^M$, and sequentially compact in the following sense: from any sequence $(\bm\rho^n)_{n \in \N}=((\rho_a^n)_{1\le a \le M})_{n \in \N}$ of elements of $\bcK_{\rho,\bR}$, one can extract a subsequence $(\bm\rho^{n_k})_{k \in \N}$ converging toward an element $\bm\rho^*$ of $\bcK_{\rho,\bR}$ for the weak-* topologies of $\cM_{\rm b}(\R^3)^M$ and $L^\infty(\R^3)^M$, and for the weak topology of $L^p(\R^3)^M$ for any $1 < p < \infty$, and such that
\begin{align}\label{eq:masslimit}
\forall 1 \le a \le M, \quad \int_{\R^3} \rho_a^{n_k}(\br) \, d\br & \mathop{\longrightarrow}_{k \to \infty} \int_{\R^3} \rho_a^*(\br) \, d\br, \\
 \int_{\R^3} |\br| \rho_a^{n_k}(\br) \, d\br & \mathop{\longrightarrow}_{k \to \infty} \int_{\R^3} |\br| \rho_a^*(\br) \, d\br. \label{eq:momentlimit}
\end{align}
\end{lemma}

\begin{proof}
Let $(\bm\rho^n)_{n \in \N}$ denote a sequence of elements of $\bcK_{\rho,\bR}$ so that for all $n\in \mathbb{N}$, $\bm\rho^n:=\left(\rho_a^n\right)_{n\in\mathbb{N}}$.
Since for all $n\in \mathbb{N}$ and $1\leq a \leq M$, $\rho_a^n \geq 0$ and $\rho \geq 0$, the fact that
 $
 \displaystyle \sum_{a=1}^M \rho_a^n(\,\cdot\,-\bR_a)= \rho
 $
 implies that 
 $
 \|\rho_a^n\|_{L^1(\mathbb{R}^3)} \leq \left\| \rho \right\|_{L^1(\mathbb{R}^3)}
 $
 and
 $
 \left\| \rho_a^n \right\|_{L^\infty(\mathbb{R}^3)} \leq  \left\| \rho \right\|_{L^\infty(\mathbb{R}^3)}$. 
 Thus, up to the extraction of a subsequence, still denoted by $(\rho_a^n)_{n\in \mathbb{N}}$ for the sake of simplicity, there exists a bounded Radon measure $\rho_a^*$ such that 
 $$
 \rho_a^n \mathop{\rightharpoonup}_{n\to \infty} \rho_a^* \mbox{ weakly-* in } \cM_{\rm b}(\R^3).
 $$
 In particular, $\rho_a^* \ge 0$ and $\displaystyle \sum_{a=1}^N \rho_a^*(\cdot-\bR_a) = \rho$. Since $(1+|\cdot|) \rho \in L^1(\mathbb{R}^3)$, $\rho_a^*$ is necessarily absolutely continuous with respect to the Lebesgue measure and hence $(1+|\cdot|) \rho_a^* \in L^1(\mathbb{R}^3)$. 
 Besides, $(\rho_a^n)_{n \in \N}$ being bounded in every $L^p(\R^3)$ for $1 < p \le \infty$, it holds that, up to the extraction of a subsequence,
 $$
 \rho_a^n \mathop{\rightharpoonup}_{n\to +\infty} \rho_a^* \mbox{ weakly-* in }L^\infty(\mathbb{R}^3) \mbox{ and weakly in }L^p(\mathbb{R}^3),
 \; \mbox{ for all } 1 < p < \infty.
 $$
 Thus $\rho^*=(\rho_a^*)_{1 \le a \le M} \in \bcK_{\rho,\bR}$, and the sequential compactness properties listed in the lemma hold true. Moreover, let us denote by $N_a^n:= \int_{\mathbb{R}^3} \rho_a^n$ and by $N_a^*:= \int_{\mathbb{R}^3} \rho_a^*$. It holds that
 $
 \displaystyle N_a^* \leq \mathop{\liminf}_{n\to +\infty} N_a^n$.
The fact that for all $n\in \mathbb{N}^*$,
 $
 \sum_{a=1}^M N_a^n = \int_{\mathbb{R}^3} \rho =\sum_{a=1}^M N_a^*
 $
implies~\eqref{eq:masslimit}. The convergence \eqref{eq:momentlimit} is obtained in the same way.
 \end{proof}

\subsection{Proof of Theorem~\ref{thm:WPcontPb}}\label{ssec:WPcontPb}

Using the arguments in Remark~\ref{rem:assumptions_ISA} for ISA, and assumption~\eqref{eq:cond_gzak} for L-ISA, we see that the infimum in \eqref{eq:ISAopt} is finite. Let $\left( \bm\rho^n, \bm \rho^{0,n}\right)_{n\in \mathbb{N}}$ be a minimizing sequence for (\ref{eq:ISAopt}). 
 By construction, the sequence $\left( \bm\rho^n,  \bm G^0_{\bcK^0}(\bm \rho^n)\right)_{n\in \mathbb{N}}$, where 
 $$
 \bm G^0_{\bcK^0}(\bm \rho^n) := \argmin_{\widetilde{\bm\rho}^0 \in \bcK^0, \, \bcN(\widetilde{\bm\rho}^0)=\bcN(\bm\rho^n)} S(\bm\rho^n|\widetilde{\bm\rho}^0)
 $$ 
 (we use the notation $\widetilde{\bm\rho}^0$ to avoid a conflict of notation with the first term of the sequence $\left( \bm\rho^n\right)_{n\in \mathbb{N}}$),
 is uniquely defined for ISA and L-ISA by virtue of Lemmata~\ref{lem:Step2_ISA} and~\ref{lem:prelimG}, is also a minimizing sequence to (\ref{eq:ISAopt}). We can therefore assume without loss of generality that $\bm\rho^{0,n}=\bm G^0_{\bcK^0}(\bm \rho^n)$.
 
 \medskip
 
 We infer from Lemma~\ref{lem:topology_KrhoR} that there exists $\bm\rho^* \in\bcK_{\rho,\bR}$ such that, up to extraction of a subsequence, $(\bm\rho^n)_{n \in \N}$ converges to $\bm\rho^*$ for the weak-* topologies of $\cM_{\rm b}(\R^3)^M$ and $L^\infty(\R^3)^M$, and for the weak topology of $L^p(\R^3)^M$ for any $1 < p < \infty$, and 
$$
\forall 1 \le a \le M, \quad N_a^n:=\int_{\R^3} \rho_a^{n}(\br) \, dr \mathop{\longrightarrow}_{k \to \infty} \int_{\R^3} \rho_a^*(\br) \, dr =:N_a^*.
$$

\noindent\boxed{1}
We want to show that $\bm G^0_{\bcK^0}(\bm\rho^n) \rightharpoonup \bm G^0_{\bcK^0}(\bm\rho^\ast)$ for the weak-* topology of $\cM_{\rm b}(\R^3)^M$.

\medskip

\noindent\circled{i}
Let us first deal with ISA. We have $\rho_a^{0,n}:=(\bm G^0_{\bcK^0_{\rm ISA}}(\bm\rho^n))_a
= \langle\rho_a^n\rangle_s
$ (spherical average). We therefore have for all $1 \le a \le M$ and all $n \in \N$, $0 \le \rho_a^{0,n} \le \langle \rho (\,\cdot\,+\bR_a)\rangle_s \in X_+$ so that $\rho_a^{0,n} \in X_+$. Reasoning as in the proof of Lemma~\ref{lem:topology_KrhoR}, we obtain the existence of a function $\rho_a^{0,*} \in X_+$ such that
 $$
 \rho_a^{0,n} \mathop{\rightharpoonup}_{n\to +\infty} \rho_a^{0,*}, \mbox{ weakly-* in } \cM_{\rm b}(\R^3) \mbox{ and } L^\infty(\mathbb{R}^3), \mbox{ and weakly in }L^p(\mathbb{R}^3),
 $$
 for all $1 < p < \infty$.
The function $\bm G^0_{\bcK^0_{\rm ISA}}$ is linear and strongly continuous on e.g. $L^2(\R^3)^M$, hence weakly continuous on this space. We therefore have $\bm\rho^{0,n}=\bm G^0_{\bcK^0_{\rm ISA}}(\bm\rho^n) \rightharpoonup \bm G^0_{\bcK^0_{\rm ISA}}(\bm\rho^\ast)$ weakly in $L^2(\R^3)^M$.
Since $(\bm\rho^{0,n})_{n \in \N}$ also converges to $\bm\rho^{0,*}$ weakly in $L^2(\R^3)^M$, we obtain that $\bm\rho^{0,*}=G^0_{\bcK^0_{\rm ISA}}(\bm\rho^\ast)$. Thus, $\bm G^0_{\bcK^0_{\rm ISA}}(\bm\rho^n) \rightharpoonup \bm G^0_{\bcK^0_{\rm ISA}}(\bm\rho^\ast)$ for the weak-* topology of $\cM_{\rm b}(\R^3)^M$.

 \medskip
 
 \noindent\circled{ii}
 Let us now turn to L-ISA. The above argument for ISA heavily relies  on the fact that $\bm \rho \mapsto \bm G^0_{\cK^0_{\rm ISA}}(\bm \rho)$ is a linear map, which is not the case of the map $\bm\rho \mapsto \bm G^0_{\cK^0_{\rm L-ISA}}(\bm \rho)$. 

We denote by $\overline{\cK}^0_{z_a}:= \left\{ \overline\rho_a^0 \in \cK^0_{z_a,\rm L-ISA} \; | \; \int_{\mathbb{R}^3} \overline\rho_a^0 = 1\right\}$ and $F_a^n: \overline{\cK}^0_{z_a} \ni \overline\rho_a^0 \mapsto - \int_{\mathbb{R}^3} \rho_a^n \log \overline\rho_a^0$. Since $\overline{\cK}^0_{z_a}$ is a closed finite-dimensional simplex of the vector space $(X^r)^M$, all the relevant topologies on $(X^r)^M$ (the strong and weak-* topologies of $\cM_{\rm b}(\R^3)^M$ and $L^\infty(\R^3)^M$, the strong and weak topologies of $L^p(\R^3)$ for $1 < p < \infty$) are equivalent on $\overline{\cK}^0_{z_a}$. It is then easily seen using the concavity of the logarithm that the functions $F_a^n$ are convex and continuous, and strictly convex if $\rho_a^n \neq 0$. The same property holds for $F_a^*$. Let us prove that
\begin{equation} \label{eq:pre_Gamma_CV}
\overline\rho_a^{0,n} \mathop{\longrightarrow}_{n \to \infty} \overline\rho_a^{0,*} \mbox{ (for any of these topologies)} \quad \Rightarrow \quad 
F_a^n(\overline\rho_a^{0,n}) \mathop{\longrightarrow}_{n \to \infty}  F_a^*(\overline\rho_a^{0,*}).
\end{equation}
First, since for all $n\in \mathbb{N}$, $\overline{\rho}_a^{0,n}:= \sum_{k=1}^{m_{z_a}} {c}_{a,k}^n g_{z_a,k}$ where the $g_{z_a,k}$ are positive $L^1$-normalized functions of $X_+^r$ and $\sum_{k=1}^{m_{z_a}} {c}_{a,k}^n =1$, it holds that $g_a^- \leq \overline{\rho}_a^{0,n} \leq g_a^+$ so that $ \log g_a^- \leq \log \overline{\rho}_a^{0,n} \leq  \log g_a^+ $, where $g_a^-:= \mathop{\min}_{1\leq k \leq m_{z_a}} g_{z_a,k}$ and $g_a^+:= \mathop{\max}_{1\leq k \leq m_{z_a}} g_{z_a,k}$. 
Thus, we obtain that
 $$
 \rho_a^n |\log \overline{\rho}_a^{0,n}| \leq \rho(\cdot + \bm R_a)\max\left( |\log g_a^-|, |\log g_a^+|\right). 
 $$
Assumption (\ref{eq:cond_gzak}) implies that $\rho(\cdot + \bm R_a)|\log g_a^-|\in L^1(\mathbb{R}^3)$. In addition, since $g_{z_a,k} \in X_+^r$ for all $1\leq k \leq m_{z_a}$ and are thus bounded and go to 0 at infinity, it holds that $|\log g_a^+(\bm r)| \leq |\log g_a^-(\bm r)|$ for sufficiently large values of $|\bm r|$, which implies that $\rho(\cdot + \bm R_a) |\log g_a^+|\in L^1(\mathbb{R}^3$). Therefore, there exists a nonnegative function $h_a \in L^1(\R^3)$ such that for all $n \in \N$, 
$ \rho_a^n |\log \overline{\rho}_a^{0,n}| \leq h_a$. Let $\epsilon > 0$ and $R > 0$ such that $\int_{\R^3 \setminus B_R} h \le \epsilon$. Then,
\begin{align*}
|F_a^n(\overline\rho_a^{0,n}) -  F_a^*(\overline\rho_a^{0,*})| & \le 2\epsilon + \left| \int_{B_R}  \rho_a^n \log \overline\rho_a^{0,n} - \int_{B_R}  \rho_a^* \log \overline\rho_a^{0,*} \right|.
\end{align*}
Now, denoting by $B_R:=\{\br \in \R^3 \, | \, |\br| < R\}$, we have that $(\rho_a^n|_{B_R})_{n \in \N}$ weakly converges to $\rho_a^*|_{B_R}$ in $L^2(B_R)$, while $(\overline\rho_a^{0,n})_{n \in \N}$ converges strongly to  $\overline\rho_a^{0,*}$ since all the relevant topologies are equivalent on $\overline{\cK}^0_{z_a}$. The latter property also implies that $(\log \overline\rho_a^{0,n})_{n \in \N}$ converges strongly to  $\log \overline\rho_a^{0,*}$ as we have $g_a^- \le \overline\rho_a^{0,n} \le g_a^+$ with $g_a^+$ bounded and $g_a^-$ bounded away from zero on $B_R$ (this function is continuous and positive on the compact $\overline B_R$). Hence \eqref{eq:pre_Gamma_CV}. This properties implies in particular that the sequence of functionals $(F_a^n)_{n\in \mathbb{N}}$ $\Gamma$-converges to the functional $F_a^*$, and therefore that the minimizers of $F_a^n$ converge to the  minimizers of $F_a^*$, these minimizers being unique when $\rho_a^n$ and $\rho_a^*$ are not identically equal to zero. It follows from Lemma~\ref{lem:prelimG} that
$$
[\bm G^0_{\cK^0_{\rm L-ISA}}(\bm \rho^n)]_a = \rho_a^{0,n} = \left|\begin{array}{ll} N_a^n \argmin_{\overline{\cK}^0_{z_a}} F_a^n & \quad \mbox{if } N_a^n \neq 0, \\
0 & \quad \mbox{if } N_a^n = 0,
\end{array} \right.
$$
converges (in any relevant topologies) to 
$$
\rho_a^{0,*} = \left|\begin{array}{ll} N_a^* \argmin_{\overline{\cK}^0_{z_a}} F_a^* & \quad \mbox{if } N_a^* \neq 0, \\
0 & \quad \mbox{if } N_a^* = 0,
\end{array} \right. \quad = [\bm G^0_{\cK^0_{\rm L-ISA}}(\bm \rho^*)]_a.
$$
Hence the desired result.

 \medskip
 
\noindent\boxed{2}
 Let us now prove that $\left( \bm \rho^*, G^0_{\bcK^0}\left( \bm \rho^*\right)\right)$ is a minimizer to problem (\ref{eq:ISAopt}). To simplify the notation, let us set $\bm \rho^{0,*}=\left( \rho_a^{0,*}\right)_{1\leq a \leq M}:=G^0_{\bcK^0}\left( \bm \rho^*\right)$. It is sufficient to prove that for all $1 \le a \le M$, 
 $$
 s_{\rm KL}\left( \rho_a^* | \rho_a^{0,*} \right) \leq \mathop{\liminf}_{n\to +\infty} s_{\rm KL}\left( \rho_a^n | \rho_a^{0,n} \right).
 $$

\noindent\circled{i}
Let us again first assume that $N_a^* >0$. Then, for $n$ large enough, $2N_a^* > N_a^n > N_a^*/2$ and $\overline{\rho}_a^n:= \frac{\rho_a^n}{N_a^n}$ defines a probability measure on $\mathbb{R}^3$. 
 Since $\overline{\rho}_a^n \leq 2 \frac{\rho}{N_a^*}$ and $\rho \in L^1(\mathbb{R}^3)$, it can be easily checked that the sequence $(\overline{\rho}_a^n)_{n\in \mathbb{N}}$ is tight (in the sense of probability measures). 
Hence, up to the extraction of a subsequence, $(\overline{\rho}_a^n)_{n\in \mathbb{N}}$ weakly converges in the sense of probability measures to $\overline{\rho}_a
^* = \frac{\rho_a^*}{N_a^*}$, i.e. for all continuous bounded functions $f: \mathbb{R}^3 \to \mathbb{R}$, 
$$
\int_{\mathbb{R}^3} f \,d\overline{\rho}_a^n \mathop{\longrightarrow}_{n\to +\infty}\int_{\mathbb{R}^3} f \,d\overline{\rho}_a^*.
$$

\medskip

Besides, using similar arguments as above, it holds that the sequence $\left(\overline{\rho}_a^{0,n}\right)_{n\in\mathbb{N}}$ with $\overline{\rho}_a^{0,n}:=\frac{\rho_a^{0,n}}{N_a^n}$ weakly converges in the sense of probability measures to $\overline{\rho}_a^{0,*}:=\frac{\rho_a^{0,*}}{N_a^*}$. 
By the joint lower semi-continuity of the Kullback-Leibler divergence for the weak convergence of probability measures, it then holds that 
$$
\int_{\mathbb{R}^3} \rho_a^* \log(\rho_a^*/ \rho_a^{0,*}) \leq \mathop{\liminf}_{n\to +\infty} \int_{\mathbb{R}^3} \rho_a^n \log(\rho_a^n/ \rho_a^{0,n}).
$$

\medskip

\noindent\circled{ii} Let us now assume that $N_a^*=0$, which implies that $\rho_a^*= \rho_a^{0,*} =0$. Since 
$
\displaystyle N_a^n \mathop{\longrightarrow}_{n\to +\infty} N_a^*=0$, the sequences $(\rho_a^n)_{n\in\mathbb{N}}$ and $(\rho_a^{0,n})_{n\in \mathbb{N}}$ strongly converge to $0$ in $L^1(\mathbb{R}^3)$. Thus, up to the extraction of subsequences, both sequences converge almost everywhere to $0$. Using the fact that
\begin{itemize}
\item  
 $|x\log x| \leq |y\log y|$ for all $y\leq 1/e$ and $0\leq x \leq y$; 
\item 
$|x\log x|\leq e$ for all $0 \leq x \leq e$,
\item  
 $|x\log x| \leq |y\log y|$ for all $y\geq e$ and $0\leq x \leq y$; 
\end{itemize}
and the bounds $0\leq \rho_a^n\leq \rho(\cdot + R_a)$ for all $n\in\mathbb{N}$, 
we obtain that 
$$
|\rho_a^n \log \rho_a^n| \leq e \chi_{\frac{1}{e} \leq \rho(\cdot + R_a) \leq e} + |\rho(\cdot + R_a) \log \rho(\cdot + R_a)| \left(1 - \chi_{\frac{1}{e} \leq \rho(\cdot +R_a) \leq e}\right).
$$
In the case when $\bcK^0 = \bcK^0_{\rm ISA}$, it holds that $\rho_a^{0,n} =\langle \rho_a^n \rangle_s$ for all $n\in \mathbb{N}$, so that $0 \leq \rho_a^{0,n}  \leq \langle \rho(\cdot + R_a) \rangle_s$. Hence, 
$$
|\rho_a^{0,n} \log (\rho_a^{0,n})| \leq e \chi_{\frac{1}{e} \leq \langle \rho(\cdot + R_a) \rangle_s \leq e} + |\langle \rho( \cdot + R_a) \rangle_s \log \langle \rho(\cdot + R_a) \rangle_s| \left(1 - \chi_{\frac{1}{e} \leq \langle \rho( \cdot + R_a) \rangle_s \leq e}\right) ,
$$
where $\chi_E$ denotes the characteristic function of $E$.  
Note also that $\langle \rho( \cdot + R_a) \rangle_s\log \langle \rho( \cdot + R_a) \rangle_s \in L^1(\mathbb{R}^3)$ by the convexity of the function $\mathbb{R}_+ \ni x \mapsto x\log x$. Thus, since
$$
\int_{\mathbb{R}^3} \rho_a^n \log\left( \frac{\rho_a^n}{\langle \rho_a^{n}\rangle_s}\right) = \int_{\mathbb{R}^3} \rho_a^n \log\left(\rho_a^n\right) - \int_{\mathbb{R}^3} \langle\rho_a^{n}\rangle_s \log\left(\langle \rho_a^{n}\rangle_s\right),
$$
we deduce from the Lebesgue dominated convergence theorem that
$$
\int_{\mathbb{R}^3} \rho_a^n \log\left( \frac{\rho_a^n}{\langle \rho_a^{n}\rangle_s}\right) \mathop{\longrightarrow}_{n\to +\infty} 0 = \int_{\mathbb{R}^3} \rho_a^* \log\left( \frac{\rho_a^*}{\langle \rho_a^{*}\rangle_s}\right).
$$

\medskip

In the case when $\bcK^0 = \bcK^0_{\rm L-ISA}$, since for all $n\in \mathbb{N}$, $\rho_a^{0,n}:= \sum_{k=1}^{m_{z_a}} c_{a,k}^n g_{z_a,k}$ where the $g_{z_a,k}$ are positive $L^1$-normalized functions of $X_+^r$, it holds that $N_a^n g_a^- \leq \rho_a^{0,n} \leq N_a^n g_a^+$ so that $ \log g_a^- + \log N_a^n \leq \log \rho_a^{0,n} \leq  \log g_a^+ + \log N_a^n$. As a consequence, 
$$
|\log \rho_a^{0,n}|\leq |\log N_a^n| + \max\left( |\log g_a^-|, |\log g_a^+|\right),
$$
and 
$$
\rho_a^n |\log \rho_a^{0,n}| \leq \rho_a^n |\log N_a^n| + \rho_a^n \max\left( |\log g_a^-|, |\log g_a^+|\right). 
$$
On the one hand, 
$$
\int_{\mathbb{R}^3} \rho_a^n |\log N_a^n| = N_a^n |\log N_a^n| \mathop{\longrightarrow}_{n\to +\infty} 0.
$$
On the other hand, $\rho_a^n \max\left( |\log g_a^-|, |\log g_a^+|\right) \leq \rho(\cdot+\bR_a) \max\left( |\log g_a^-|, |\log g_a^+|\right)\in L^1(\mathbb{R}^3)$.
Thus, the Lebesgue dominated convergence theorem yields that 
$$
\int_{\mathbb{R}^3}\rho_a^n \max\left( |\log g_a^-|, |\log g_a^+|\right) \mathop{\longrightarrow}_{n\to +\infty} 0. 
$$
Finally, $\int_{\mathbb{R}^3} \rho_a^n |\log \rho_a^{0,n}| \mathop{\longrightarrow}_{n\to +\infty} 0$, 
and 
$$
\int_{\mathbb{R}^3} \rho_a^n \log\left( \frac{\rho_a^n}{\rho_a^{0,n}}\right) \mathop{\longrightarrow}_{n\to +\infty} 0 = \int_{\mathbb{R}^3} \rho_a^* \log\left( \frac{\rho_a^*}{ \rho_a^{0,*}}\right).
$$

We have thus proved that in any case,
$$
\int_{\mathbb{R}^3} \rho_a^* \log\left( \frac{\rho_a^*}{\langle \rho_a^{*}\rangle_s}\right) \leq \mathop{\liminf}_{n\to +\infty} \int_{\mathbb{R}^3} \rho_a^n \log\left( \frac{\rho_a^n}{\langle \rho_a^{n}\rangle_s}\right)
$$
for all $1\leq a \leq M$, which implies that $\left(\bm \rho^*, \bm\rho^{0,*}\right)$ is a minimizer to (\ref{eq:ISAopt}). 

\medskip

\noindent\boxed{3}
The uniqueness comes from the strict convexity of the functional 
$$
\bcK_{\rho, \bR} \times \bcK^0 \ni \left(\bm \rho, \bm\rho^0\right) \mapsto S(\bm \rho | \bm \rho^0 ) \in \R_+ \cup \{+\infty\}
$$
on the convex subset of $\bcK_{\rho, \bR} \times \bcK^0$ on which it is finite.

\medskip

Lastly, (\ref{eq:EL1}) is a direct consequence of the fact that 
$$
\bm \rho^*:= {\rm argmin}_{\bm \rho \in \bcK_{\rho, \bR}} S\left( \bm \rho| \bm \rho^{0,*}\right)
$$
and Lemma~\ref{lem:Step1}.

\subsection{Proof of Theorem~\ref{th:convergence}}\label{sec:proofconvergence}

It is easy to check that the assumptions $\rho > 0$ and $s_{\rm KL}(\rho|\rho^{0,(0)}) < +\infty$ ensure that for both ISA and L-ISA, the first iteration is well defined. Assume now by induction that the first $\widetilde m$ iterations are well-defined for some $\widetilde m \in \N^*$ and that $s_{\rm KL}(\rho|\rho^{0,(m-1)}) < +\infty$ for all $1 \le m \le \widetilde m$.

\medskip

For all $1 \le m \le \widetilde m$ and $1\leq a \leq M$, we denote by 
$$
N_a^{(m)}:= \int_{\mathbb{R}^3} \rho_a^{(m)} = \int_{\mathbb{R}^3} \rho_a^{0,(m)} \quad \mbox{ and } \quad \bN^{(m)} = (N_a^{(m)})_{1 \leq a \leq M}\in \mathbb{R}_+^M.
$$
From the assumptions of Theorem~\ref{th:convergence}, it can be easily proven by recursion that $N_a^{(m)}>0$ for all $1 \le m \le \widetilde m$  and all $1\leq a \leq M$. 
 
 \medskip
 
 Using Lemmata~\ref{lem:Step1} and~\ref{lem:prelimG}, we obtain that for all $1 \le m \le \widetilde m$, $\bm \rho^{(m)}$ and $\bm \rho^{0,(m)}$ are respectively minimizers of 
 \begin{equation}\label{eq:iter1}
\bm \rho^{(m)} \in \mathop{\rm argmin}_{\bm \rho \in \bcK_{\rho, \bR}} S(\bm \rho|\bm \rho^{0,(m-1)})
 \end{equation}
 and
 \begin{equation}\label{eq:iter2}
\bm \rho^{0,(m)} \in \mathop{\rm argmin}_{\bm \rho^0 \in \bcK^0, \; \bcN(\bm \rho^0) = \bN^{(m)}} S(\bm \rho^{(m)} |\bm \rho^0).      
 \end{equation}
In order to prove (\ref{eq:decrease}), we first show that
\begin{equation}\label{eq:ineq1}
 S(\bm \rho^{(m)}| \bm \rho^{0,(m-1)}) \geq S (\bm \rho^{(m)}| \bm \rho^{0,(m)}).
\end{equation}
It follows from Lemmata~\ref{lem:Step1} and~\ref{lem:prelimG} that $\bm\mu^{(m)}\odot \bm \rho^{0,(m)}$, with $\bm \mu^{(m)}:=\left( \frac{N_a^{(m-1)}}{N_a^{(m)}}\right)_{1\leq a \leq M}$, is the unique minimizer of 
  \begin{equation}\label{eq:iter2niter}
\inf_{\bm \rho^0 \in \bcK^0, \; \bcN(\bm \rho^0) = \bN^{(m-1)}} S(\bm \rho^{(m)} |\bm \rho^0).    \end{equation}
We therefore have 
  \begin{align*}
  S(\bm \rho^{(m)}| \bm \rho^{0,(m-1)})
  & \geq S(\bm \rho^{(m)}|\bm\mu^{(m)}\odot \bm \rho^{0,(m)}) \\
  & = \sum_{a=1}^M \int_{\mathbb{R}^3} \rho_a^{(m)} \log\left( \frac{\rho_a^{(m)}}{\frac{N_a^{(m-1)}}{N_a^{(m)}} \rho_a^{0,(m)}}\right) \\
  & = \sum_{a=1}^M \int_{\mathbb{R}^3} \rho_a^{(m)} \log\left( \frac{\rho_a^{(m)}}{\rho_a^{0,(m)}}\right) +  \int_{\mathbb{R}^3} \rho_a^{(m)} \log\left( \frac{N_a^{(m)}}{N_a^{(m-1)}}\right)\\
  & = \sum_{a=1}^M \int_{\mathbb{R}^3} \rho_a^{(m)} \log\left( \frac{\rho_a^{(m)}}{\rho_a^{0,(m)}}\right) +  N_a^{(m)} \log\left( \frac{N_a^{(m)}}{N_a^{(m-1)}}\right),\\
 & = S (\bm \rho^{(m)}| \bm \rho^{0,(m)}) + \sum_{a=1}^M N_a^{(m)} \log\left( \frac{N_a^{(m)}}{N_a^{(m-1)}}\right).\\
  \end{align*}
It thus remains to show that
$$
\sum_{a=1}^M N_a^{(m)} \log\left( \frac{N_a^{(m)}}{N_a^{(m-1)}}\right) \geq 0.
$$
We now use the fact that $-\log x \geq 1 -x$ (since $\log x \leq x-1$) to get
  \begin{align*}
     \sum_{a=1}^M N_a^{(m)} \log\left( \frac{N_a^{(m)}}{N_a^{(m-1)}}\right) &= \sum_{a=1}^M -N_a^{(m)} \log\left( \frac{N_a^{(m-1)}}{N_a^{(m)}}\right)\\
     & \geq\sum_{a=1}^M N_a^{(m)}\left(1 - \frac{N_a^{(m-1)}}{N_a^{(m)}} \right) 
     = \sum_{a=1}^M (N_a^{(m)} - N_a^{(m-1)})
     = 0,
  \end{align*}
  since 
  $$\sum_{a=1}^M N_a^{(m)} = \sum_{a=1}^M \int_{\mathbb{R}^3} \rho_a^{(m)}  =
  \int_{\mathbb{R}^3} \rho
  = \sum_{a=1}^M \int_{\mathbb{R}^3} \rho_a^{(m-1)}
  = \sum_{a=1}^M N_a^{(m-1)}.
  $$ 
  Hence (\ref{eq:ineq1}). Let us now show that
  \begin{equation}\label{eq:ineq2}
  S(\bm \rho^{(m-1)} | \bm \rho^{0,(m-1)} )  \geq  S(\bm \rho^{(m)}| \bm \rho^{0, (m-1)}) + \frac{1}{2\|\rho\|_{L^\infty}} \sum_{a=1}^M \left\| \rho_a^{(m)} - \rho_a^{(m-1)} \right\|_{L^2}^2. 
  \end{equation}
Note that, by definition of $\bm \rho ^{(m)}$, $S(\bm \rho^{(m-1)} | \bm \rho^{0,(m-1)} )  \geq  S(\bm \rho^{(m)}| \bm \rho^{0, (m-1)})$. 
Using a second-order Taylor expansion formula with integral remainder, it holds that for all $x,x',y >0$,
$$
x'\log\left( \frac{x'}{y} \right) \geq x \log \left( \frac{x}{y} \right) + (x'-x)\left[ \log \left(\frac{x}{y} \right) + 1\right] + \frac{1}{\max(x,x')} \frac{(x-x')^2}{2}.
$$
As a consequence, since for all $1\leq a \leq M$, 
$\frac{1}{\|\rho_a^{(m)}\|_{L^\infty}} \geq \frac{1}{\|\rho\|_{L^\infty}}$, we obtain that
\begin{align*}
   S(\bm \rho^{(m-1)} | \bm \rho^{0,(m-1)} )  & \geq  S(\bm \rho^{(m)}| \bm \rho^{0, (m-1)}) + \sum_{a=1}^M \int_{\mathbb{R}^3} (\rho_a^{(m-1)} - \rho_a^{(m)})\left[\log\left( \frac{\rho_a^{(m)}}{\rho_a^{0,(m-1)}}\right) + 1\right]\\
   & + \sum_{a=1}^M \frac{1}{2\|\rho\|_{L^\infty}} \int_{\mathbb{R}^3} (\rho_a^{(m-1)} - \rho_a^{(m)})^2.\\
\end{align*}
Using the fact that the expression
$$
\frac{\rho_a^{(m)}(\bm r)}{\rho_a^{0,(m-1)}(\bm r)} = \frac{\rho(\bm r + \bR_a)}{\sum_{b=1}^M  \rho_b^{0,(m-1)}(\bm r - \bR_b + \bR_a)}
$$
is independent of $a$ and the equality
$$
\sum_{a=1}^M \int_{\mathbb{R}^3} \rho_a^{(m-1)} = \sum_{a=1}^M \int_{\mathbb{R}^3} \rho_a^{(m)} = \int_{\mathbb{R}^3} \rho, 
$$
we thus obtain that
\begin{align*}
   S(\bm \rho^{(m-1)} | \bm \rho^{0,(m-1)} )  & \geq  S(\bm \rho^{(m)}| \bm \rho^{0, (m-1)})\\
   & + \sum_{a=1}^M \int_{\mathbb{R}^3} (\rho_a^{(m-1)}(\bm r - \bR_a) - \rho_a^{(m)}(\bm r - \bR_a))\log\left( \frac{\rho(\bm r )}{\sum_{b=1}^M  \rho_b^{0,(m-1)}(\bm r - \bR_b)}\right)\,d\bm r\\
   & + \frac{1}{2\|\rho\|_{L^\infty}}\sum_{a=1}^M  \left\|\rho_a^{(m-1)} - \rho_a^{(m)}\right\|_{L^2}^2,\\
    & = S(\bm \rho^{(m)}| \bm \rho^{0, (m-1)}) + \int_{\mathbb{R}^3} (\rho(\bm r) - \rho(\bm r))\log\left( \frac{\rho(\bm r )}{\sum_{b=1}^M  \rho_b^{0,(m-1)}(\bm r - \bR_b)}\right)\,d\bm r\\
   & + \frac{1}{2\|\rho\|_{L^\infty}}\sum_{a=1}^M  \left\|\rho_a^{(m-1)} - \rho_a^{(m)}\right\|_{L^2}^2,\\
    & = S(\bm \rho^{(m)}| \bm \rho^{0, (m-1)}) + \frac{1}{2\|\rho\|_{L^\infty}}\sum_{a=1}^M  \left\|\rho_a^{(m-1)} - \rho_a^{(m)}\right\|_{L^2}^2.
\end{align*}
Hence (\ref{eq:ineq2}). Collecting this result together with (\ref{eq:ineq1}), we obtain (\ref{eq:decrease}). It also follows from \eqref{eq:min_on_KrhoR}, \eqref{eq:iter1}, and \eqref{eq:ineq1} that 
$$
s_{\rm KL}(\rho,\rho^{0,(m-1)}) = S(\bm\rho^{(m)}|\bm\rho^{0,(m-1)}) \ge S(\bm\rho^{(m)}|\bm\rho^{0,(m)}) \ge \inf_{\bm\rho \in \bcK_{\rho,\bR}} S(\bm\rho|\bm\rho^{0,(m)}) = s_{\rm KL}(\rho,\rho^{0,(m)}).
$$
In particular $s_{\rm KL}(\rho,\rho^{0,(m)}) < +\infty$ for all $1 \le m \le \widetilde m$, so that the iteration $\widetilde m+1$ is well defined. By recursion, all the iterations are well defined and the above results are valid for all $m \ge 1$. The bound \eqref{eq:decrease_KL} is obtained by a simple combination of the above inequalities.

\medskip

Another consequence of (\ref{eq:decrease}) is that $\left( S(\bm \rho^{(m)} | \bm \rho^{0,(m)}) \right)_{m\geq 0}$ is a non-increasing sequence of nonnegative real numbers, hence converges to a nonnegative real number as $m$ tends to infinity. The bound  (\ref{eq:decrease})  also implies that
$$
\sum_{m\geq 1} \sum_{a=1}^M \left\| \rho_a^{(m)} - \rho_a^{(m-1)} \right\|_{L^2}^2 < +\infty,
$$
which, in turn, implies (\ref{eq:diffiter}). 

\medskip

 Let us finally prove the last assertion of Theorem~\ref{th:convergence}. We infer from Lemma~\ref{lem:topology_KrhoR} that there exist $\bm\rho^* \in\bcK_{\rho,\bR}$ and a subsequence $(\bm\rho^{(m_l)})_{l \in \N}$ of $(\bm\rho^{(m)})_{m \in \N}$ converging to $\bm\rho^*$ for the weak-* topologies of $\cM_{\rm b}(\R^3)^M$ and $L^\infty(\R^3)^M$, and for the weak topology of $L^p(\R^3)^M$ for any $1 < p < \infty$, and such that
$$
\forall 1 \le a \le M, \quad N_a^{(m_l)}:=\int_{\R^3} \rho_a^{(m_l)}(\br) \, dr \mathop{\longrightarrow}_{l \to \infty} \int_{\R^3} \rho_a^*(\br) \, dr =:N_a^*.
$$
In addition, we can prove that there exists $\bm \rho^{0,*}\in \bcK^0$, such that, up to extraction of a subsequence,
$(\bm\rho^{0,(m_l)})_{l \in \N}$ converges to $\bm\rho^{0,*}$ for the weak-* topologies of $\cM_{\rm b}(\R^3)^M$ and $L^\infty(\R^3)^M$, and for the weak topology of $L^p(\R^3)^M$ for any $1 < p < \infty$. 
Using similar arguments as in the proof of Theorem~\ref{thm:WPcontPb}, we obtain that 
$$
\bm \rho^{0,*} = \bm G^0_{\bcK^0_{\rm L-ISA}}(\bm \rho^*). 
$$
Moreover, since $\bcK^0_{\rm L-ISA}\subset {\rm Span}\left\{ g_{z_a,k}, 1\leq a \leq M, \; 1 \leq k \leq m_{z_a}\right\}$, which is a finite-dimensional vector space, the convergence of the sequence $(\bm\rho^{0,(m_l)})_{l \in \N}$ to $\bm \rho^{0,*}$ holds strongly in any $L^p(\mathbb{R}^3)$ for any $1\leq p \leq +\infty$. 

Using the fact that
$$
\rho_a^{(m_l+1)}(\bm r) = \frac{\rho_a^{0,(m_l)}(\bm r) }{\sum_{b=1}\rho_b^{0,(m_l)}(\bm r- \bR_b + \bR_a)} \rho(\bm r + \bR_a)
$$
together with $\eqref{eq:diffiter}$ yields that
$$
\rho_a^*(\bm r) =  \frac{\rho_a^{0,*}(\bm r) }{\sum_{b=1}\rho_b^{0,*}(\bm r- \bR_b + \bR_a)} \rho(\bm r + \bR_a). 
$$
This relationship, together with the fact that $\bm \rho^{0,*} = \bm G^0_{\bcK^0_{\rm L-ISA}}(\bm \rho^*)$ implies that $(\bm \rho^*, \bm \rho^{0,*})$ is the unique minimizer of (\ref{eq:ISAopt}). Hence the desired result.

\subsection{Proof of Proposition~\ref{prop:spheric}}\label{sec:sphericproof}

 If $\rho(\cdot + \bR_2) = \langle \rho(\cdot + \bR_2) \rangle_s$, it is clear that $\rho^{\rm opt}_1 = 0$ and $\rho^{\rm opt}_2 = \langle \rho^{\rm opt}\rangle_s$. 
 
 Conversely, let us assume that $\rho^{\rm opt}_1= 0$ (and therefore $\rho_2^{\rm opt}=\rho(\cdot + \bR_2) > 0$ on $\R^3$, which implies that $\langle \rho_2^{\rm opt} \rangle_s = \langle\rho(\cdot + \bR_2) \rangle_s  > 0$ on $\R^3$), and prove that necessarily $\rho_2^{\rm opt}= \langle \rho(\cdot + \bR_2) \rangle_s$. 
 For all $\bm \rho \in \bcK_{\rho, \bR}$, we denote by $\mathcal J(\bm \rho):= S(\bm \rho | \langle \bm \rho \rangle_s)$. 
 Let us consider radially symmetric functions $h_1 \in X^r$ and $h_2 \in X^r$ such that $\langle \rho_2^{\rm opt} \rangle_s + \epsilon h_2 > 0$ a.e. for $\epsilon >0$ small enough, and 
 consider the perturbed AIM densities
 $$
 \rho_a^\epsilon := \frac{\langle \rho_a^{\rm opt}\rangle_s + \epsilon h_a}{\sum_b (\langle \rho_b^{\rm opt}\rangle_s+\epsilon h_b)(\cdot -\bR_b + \bR_a)} \rho(\cdot + \bR_a).
 $$
  To the first order in $\epsilon$, we have, uniformly in $L^\infty_{\rm loc}(\R^3)$,
 \begin{align*}
 \rho_1^\epsilon &= \epsilon  \frac{h_1}{\langle \rho(\cdot + \bR_2) \rangle_s(\cdot+\bR_1)} \rho(\cdot + \bR_1) + o(\epsilon), \\
 \rho_2^\epsilon  &=  \rho(\cdot + \bR_2) - \epsilon  \frac{h_1(\cdot-\bR_1+\bR_2)}{\langle \rho(\cdot + \bR_2) \rangle_s} \rho(\cdot + \bR_2) + o(\epsilon),  \end{align*}
 and thus
 \begin{align*}
 \mathcal J(\bm\rho^\epsilon) - \mathcal J(\bm\rho^{\rm opt})  & = - \epsilon \int_{\mathbb{R}^3} \frac{\rho(\cdot +\bR_2)}{\langle \rho(\cdot +\bR_2)\rangle_s}\log\left( \frac{\rho(\cdot +\bR_2)}{\langle \rho(\cdot +\bR_2)\rangle_s} \right)h_1(\cdot -\bR_1 + \bR_2) + o(\epsilon).\\
 \end{align*} 
 Let us denote by 
 $$
 \rho^{s,2}(\bm r):= \langle \rho(\cdot +\bR_2)\rangle_s\left( \bm r - \bR_2\right). 
 $$
 It then holds that
 \begin{align*}
 \mathcal J(\bm\rho^\epsilon) - \mathcal J(\bm\rho^{\rm opt}) &  = - \epsilon \int_{\mathbb{R}^3} \frac{\rho(\cdot +\bR_1)}{ \rho^{s,2}(\cdot +\bR_1)}\log\left( \frac{\rho(\cdot +\bR_1)}{ \rho^{s,2}(\cdot +\bR_1)} \right)h_1 + o(\epsilon)\\
 & = - \epsilon \int_{\mathbb{R}^3} \left\langle \frac{\rho(\cdot +\bR_1)}{ \rho^{s,2}(\cdot +\bR_1)}\log\left( \frac{\rho(\cdot +\bR_1)}{ \rho^{s,2}(\cdot +\bR_1)} \right) \right\rangle_s h_1 + o(\epsilon)\\
 \end{align*}
 First, we have
 \begin{align*}
\int_{\mathbb{R}^3} \left\langle \frac{\rho(\cdot +\bR_1)}{ \rho^{s,2}(\cdot +\bR_1)}\log\left( \frac{\rho(\cdot +\bR_1)}{ \rho^{s,2}(\cdot +\bR_1)} \right) \right\rangle_s & = \int_{\mathbb{R}^3}\frac{\rho(\cdot +\bR_1)}{ \rho^{s,2}(\cdot +\bR_1)}\log\left( \frac{\rho(\cdot +\bR_1)}{ \rho^{s,2}(\cdot +\bR_1)} \right)\\
 & = \int_{\mathbb{R}^3} \frac{\rho}{\rho^{s,2}}\log\left( \frac{\rho}{\rho^{s,2}} \right)\\
 & = \int_{\mathbb{R}^3}\left( \frac{\rho}{\rho^{s,2}}\log\left( \frac{\rho}{\rho^{s,2}} \right)\right)^{s,2}\\
  & = \int_{\mathbb{R}^3}\frac{\rho(\cdot +\bR_2)}{\langle \rho(\cdot +\bR_2)\rangle_s}\log\left( \frac{\rho(\cdot +\bR_2)}{\langle \rho(\cdot +\bR_2)\rangle_s} \right)\\
    & = \int_{\mathbb{R}^3}\left\langle \frac{\rho(\cdot +\bR_2)}{\langle \rho(\cdot +\bR_2)\rangle_s}\log\left( \frac{\rho(\cdot +\bR_2)}{\langle \rho(\cdot +\bR_2)\rangle_s} \right)\right\rangle_s\\
  & \geq \int_{\mathbb{R}^3}\frac{\langle\rho(\cdot +\bR_2)\rangle_s}{\langle \rho(\cdot +\bR_2)\rangle_s}\log\left( \frac{\langle \rho(\cdot +\bR_2)\rangle_s}{\langle \rho(\cdot +\bR_2)\rangle_s} \right) = 0,
 \end{align*}
 the last inequality being a consequence of Jensen's inequality. Let us reason by contradiction and assume that
$$
 \int_{\mathbb{R}^3} \left\langle \frac{\rho(\cdot +\bR_1)}{ \rho^{s,2}(\cdot +\bR_1)}\log\left( \frac{\rho(\cdot +\bR_1)}{ \rho^{s,2}(\cdot +\bR_1)} \right) \right\rangle_s >0. 
$$
There would then exist a bounded Borel subset $A$ of $[0, +\infty)$ with positive Lebesgue measure such that $w(r) > 0$ for almost all $r\in A$ where $w\in Y_+$ is the function such that
$$
\forall r>0, \; \forall \bm\sigma \in \S^2, \quad w(r):= \left\langle \frac{\rho(\cdot +\bR_1)}{ \rho^{s,2}(\cdot +\bR_1)}\log\left( \frac{\rho(\cdot +\bR_1)}{ \rho^{s,2}(\cdot +\bR_1)} \right) \right\rangle_s(r\bm\sigma).
$$
Choosing a particular function $h_1$ such that $h_1(r\bm\sigma) = g_1(r)$ with $g_1 >0$ on $A$ and $h_1 = 0$ on $[0, +\infty) \setminus A$ would imply that 
$$
\int_{\mathbb{R}^3} \left\langle \frac{\rho(\cdot +\bR_1)}{ \rho^{s,2}(\cdot +\bR_1)}\log\left( \frac{\rho(\cdot +\bR_1)}{ \rho^{s,2}(\cdot +\bR_1)} \right) \right\rangle_s h_1  = 4\pi \int_0^{+\infty} w(r) g_1(r)r^2\,dr >0.
$$
 Letting $\epsilon$ go to $0$, we reach a contradiction since $(\bm \rho^{\rm opt}, \langle \bm \rho^{\rm opt} \rangle_s)$ is a minimizer of (\ref{eq:ISAopt}). 
 
 Hence, it necessarily holds that 
 $$
 \int_{\mathbb{R}^3} \left\langle \frac{\rho(\cdot +\bR_1)}{ \rho^{s,2}(\cdot +\bR_1)}\log\left( \frac{\rho(\cdot +\bR_1)}{ \rho^{s,2}(\cdot +\bR_1)} \right) \right\rangle_s  = \int_{\mathbb{R}^3}\left\langle \frac{\rho(\cdot +\bR_2)}{\langle \rho(\cdot +\bR_2)\rangle_s}\log\left( \frac{\rho(\cdot +\bR_2)}{\langle \rho(\cdot +\bR_2)\rangle_s} \right)\right\rangle_s =0.
 $$
 Since $\left\langle \frac{\rho(\cdot +\bR_2)}{\langle \rho(\cdot +\bR_2)\rangle_s}\log\left( \frac{\rho(\cdot +\bR_2)}{\langle \rho(\cdot +\bR_2)\rangle_s} \right)\right\rangle_s\geq 0$ almost everywhere from Jensen's inequality, we obtain
 $$
\left\langle \frac{\rho(\cdot +\bR_2)}{\langle \rho(\cdot +\bR_2)\rangle_s}\log\left( \frac{\rho(\cdot +\bR_2)}{\langle \rho(\cdot +\bR_2)\rangle_s} \right)\right\rangle_s = 0.
 $$
 This equality implies that 
 $$
 \left\langle \rho(\cdot +\bR_2) \log \rho(\cdot +\bR_2)\right\rangle_s = \langle \rho(\cdot +\bR_2)\rangle_s\log \langle \rho(\cdot +\bR_2)\rangle_s \mbox{ a.e.}
 $$
 Using again Jensen's inequality and the strict convexity of the function $x \mapsto x\log x$, we finally obtain $\rho(\cdot +\bR_2)= \langle \rho(\cdot +\bR_2)\rangle_s$. Hence the desired result.

\subsection{Proof of Proposition~\ref{prop:continuity}} \label{sec:continuity}

It is clear from their definition that the functions $w_a$ are nonnegative and bounded by $M_\rho$. Taking spherical averages in the optimality condition \eqref{eq:EL1}, and using \eqref{eq:EL11}, we see that for all $1 \le a \le 2$,
\begin{align*}
I_a(r) &=  \fint_{\S^2} \frac{\rho(\bR_a+r \bm\sigma)}{w_a(r)+w_b(|\bR_a+r\bm\sigma-\bR_b|)} \, d\bm\sigma = 1, \quad \mbox{or} \quad w_a(r)=0,
\end{align*}
where $b = 2$ if $a=1$ and $b=1$ if $a=2$. For given and fixed $w_b$, the
function $w_a$ is therefore defined implicitly on $(0,+\infty)$ by
$$
\mbox{for almost all } r \in (0,+\infty), \quad F_a(r,w_a(r))=0 \quad \mbox{or} \quad w_a(r)=0,
$$
where the function $F_a: (0,+\infty) \times \R_+ \to [0,+\infty]$ is given by
$$
F_a(r,y)=\fint_{\S^2} \frac{\rho(\bR_a+r \bm\sigma)}{y+ w_b(|\bR_a+r\bm\sigma-\bR_b|)} \, d\bm\sigma - 1.
$$
Note that the function $\S^2 \ni \bm\sigma \mapsto \rho(\bR_a+r \bm\sigma)$ is continuous on $\S^2$, nonnegative, and bounded by~$M_\rho$. Consequently, $F_a$ is well-defined on $(0,+\infty) \times (0,+\infty)$, and for each $r > 0$, the function $(0,+\infty) \ni y \mapsto F_a(r,y) \in \R$ is decreasing and converges to $-1$ at infinity, so that for each $r \in (0,+\infty)$, there exists at most one $y \in \R_+$ such that $F_a(r,y)=0$. We denote by $w_a^+(r)$ the unique solution in $\mathbb{R}_+$ to $F_a(r,w_a^+(r))=0$ if it exists, and we set $w_a^+(r)=0$ otherwise. We therefore have $w_a(r) \in \{0,w_a^+(r)\}$ for almost all $r \in \R_+$.

\medskip
We have for all $r,y,z \in (0,+\infty)$,
\begin{align} \label{eq:pre_IFT}
\left| F_a(r,y)-F_a(r,z) \right| & = \left| 
\fint_{\S^2} \rho(\bR_a+r \bm\sigma) \left( \frac{1}{y+w_b(|\bR_a+r\bm\sigma-\bR_b|)}- \frac{1}{z+ w_b(|\bR_a+r\bm\sigma-\bR_b|)} \right) \, d\bm\sigma \right| \nonumber \\
&=  K_a(r,y,z) \, |y-z|,
\end{align}
with
$$
K_a(r,y,z):= \fint_{\S^2}  \frac{\rho(\bR_a+r \bm\sigma) }{(y+ w_b(|\bR_a+r\bm\sigma-\bR_b|))(z+ w_b(|\bR_a+r\bm\sigma-\bR_b|))} \, d\bm\sigma.
$$
Using the fact that $\rho$ is continuous and positive everywhere on $\R^3$, hence bounded away from zero on any compact subset of $\R^3$, we get that for all $(r,y,z) \in (0,+\infty) \times [0,M_\rho]^2$,
\begin{align*}
K_a(r,y,z) & \ge   \frac{m_a(r)}{2M_\rho^2} > 0
\end{align*}
where
$$
m_a(r):=\min_{\overline{B}(0,r+|\bR_a|)} \rho
$$
is a continuous, positive, decreasing function of $r$.

\medskip

For $r_1,r_2 \in \R_+$ such that $r_1+r_2 \ge R=|\bR_a-\bR_b|$, we denote by $\theta_1(r_1,r_2)$ and $\theta_2(r_1,r_2)$, the unique real numbers in $[0,\pi]$ such that
$$
r_1^2+R^2-2r_1R\cos(\theta_1(r_1,r_2))=r_2^2 \quad \mbox{and} \quad 
r_2^2+R^2-2r_2R\cos(\theta_2(r_1,r_2))=r_1^2,
$$
and by $\widetilde\rho_a(r,\theta,\phi)$ the representation of the total density $\rho$ in spherical coordinates relative to $\bR_a$, with the $z$-axis aligned with $\be:= \frac{\bR_b - \bR_a}{R}$, i.e.
$$
\widetilde\rho_a(r,\theta,\phi):=\rho(\bR_a+r(\sin\theta\cos\phi \, \be_x+\sin \theta\sin\phi \, \be_y + \cos\theta \,  \be)),
$$
where $(\be_x,\be_y,\be)$ forms an orthonormal basis of $\R^3$.
Using the change of variables $r_2=(r_1^2+R^2-2r_1R\cos\theta)^{1/2}$, for which $r_1R\sin\theta \, d\theta = r_2 \, dr_2$, we have
\begin{align*}
I_1(r_1) &= \frac{1}{4\pi} \int_0^\pi d\theta \int_0^{2\pi} d\phi \, \frac{\widetilde \rho_1(r_1,\theta,\phi)}{w_1(r_1)+w_2((r_1^2+R^2-2r_1R\cos\theta)^{1/2})} \, \sin \theta \\
&= \frac{1}{4\pi} \int_{|R-r_1|}^{R+r_1} dr_2 \int_0^{2\pi} d\phi \, \frac{\widetilde \rho_1(r_1,\theta_1(r_1,r_2),\phi)}{w_1(r_1)+w_2(r_2)} \, \frac{r_2}{r_1R} \\
&= \int_{|R-r_1|}^{R+r_1} \frac{h_1(r_1,r_2)}{w_1(r_1)+w_2(r_2)} \,  dr_2, 
\end{align*}
where the function $h_1:D_1\to \R_+$ is defined by 
$$
D_1:=\{(r_1,r_2), \; r_1 \in (0,+\infty), \; r_2 \in [|R-r_1|,R+r_1]\}, \quad h_1(r_1,r_2) := \frac{r_2}{4\pi r_1R} \int_0^{2\pi} \widetilde\rho_1(r_1,\theta_1(r_1,r_2),\phi) \, d\phi.
$$
The function $w_1$ is therefore defined implicitly on $(0,+\infty)$ by 
$$
F_1(r_1,w_1(r_1))=0 \quad \mbox{or} \quad w_1(r_1)=0,
$$
where the function $F_1: (0,+\infty) \times \R_+ \to [0,+\infty]$ is given by
$$
F_1(r_1,y)=\int_{|R-r_1|}^{R+r_1} \frac{h_1(r_1,r_2)}{y+w_2(r_2)} \,  dr_2 - 1.
$$
Note that the function $h_1$ is continuous on $D_1$ and that we have for all $(r_1,r_2) \in D_1$,
\begin{align*}
    &|h_1(r_1,r_2)| \le \frac{M_\rho r_2}{2r_1R} .
\end{align*}
By dominated convergence, $F_1$ is therefore locally bounded and continuous on $(0,+\infty) \times (0,+\infty)$. 

\medskip

Recall that $0 \le w_a(r_a) \le M_\rho$ for a.a. $r_a \in (0,+\infty)$. We therefore have for all $r_1,r_1',y,z \in (0,+\infty)$,
\begin{align*}
    \left| F_1(r_1,y)-F_1(r_1',z) \right| & \le \left| F_1(r_1,y)-F_1(r_1',y) \right| +\left| F_1(r_1',y)-F_1(r_1',z) \right| \\
    &\le \frac{2M_\rho}{\min(y,z)} |r_1-r_1'| + \frac{M_\rho}{yz} |y-z|,
\end{align*}
which proves that $F_1$ is locally Lipschitz on the open set  $(0,+\infty) \times (0,+\infty)$. It follows from \eqref{eq:pre_IFT} and the implicit function theorem for Lipschitz functions (see Theorem~\ref{thm:IFT_Lip} in Appendix B) that if $w_1^+(r_{1*}) > 0$ for some $r_{1*} > 0$, then $w_1^+$ is Lipschitz in the neighborhood of $r_{1*}$. Let $r_{1*} \in (0,+\infty)$ be such that $w_1^+(r_{1*})=0$, and $(r_{1,n}) \in (0,+\infty)^\N$ a sequence converging to $r_{1*}$. We have for all $n \in \N$
$$
F_1(r_{1,n},w_1^+(r_{1,n}))=\int_{|R-r_{1,n}|}^{R+r_{1,n}} \frac{h_1(r_{1,n},r_2)}{w_1^+(r_{1,n})+w_2(r_2)} \,  dr_2 - 1 = 0.
$$
Assume that we can extract from $(r_{1,n})_{n \in \N}$ a subsequence $(r_{1,n_k})_{k \in \N}$ with positive limit, i.e., such that $w_1^+(r_{1,n_k}) \to y_* > 0$ as $k$ tends to infinity. Passing to the limit by the dominated convergence theorem, we get
$$
F_1(r_{1*},y_*)=\int_{|R-r_{1*}|}^{R+r_{1*}} \frac{h_1(r_{1*},r_2)}{y_*+w_2(r_2)} \,  dr_2 - 1=0 . 
$$
Hence, $y_*=w_1^+(r_{1*})=0$. We reach a contradiction. 
This proves that $w_1^+$ is continuous on $(0,+\infty)$. The same holds for $w_2^+$ by symmetry. This completes the proof of the first statement.

\medskip

Let us now establish the second one. If the $w_a$'s are bounded away from zero on every compact subset of $[0,+\infty)$, then $w_a=w_a^+$ almost everywhere on $[0,+\infty)$, which, in view of the first statement, implies that the $w_a$'s are Lipschitz on $(0,+\infty)$ and that for all $r_1,r_2 > 0$,
$$
F_1(r_1,w_1(r_1))=0 \quad \mbox{and} \quad F_2(r_2,w_2(r_2))=0
$$
where
\begin{align*}
F_1(r_1,y) &= \frac{1}{4\pi} \int_{\S^2} \frac{\rho(\bR_1+r_1 \bm\sigma)}{y+w_2(|r_1\bm\sigma-R\be|)} \, d\bm\sigma - 1,   \\
F_2(r_2,y) &= \frac{1}{4\pi} \int_{\S^2} \frac{\rho(\bR_2+r_2 \bm\sigma)}{w_1(|r_2\bm\sigma+R\be|)+y} \, d\bm\sigma - 1.
\end{align*}
By a simple continuity argument, we deduce from the fact that the $w_j$'s are bounded away from zero on $[0,1]$ that $w_j$ is Lipschitz on the whole interval $[0,+\infty)$.
If $\rho$ is $C^1$ away from the centers $\bR_j$, then the functions $F_j:[0,+\infty) \times (0,+\infty) \to \R$ are $C^1$ on $((0,R)\cup (R,+\infty)) \times (0,+\infty)$. By the implicit function theorem for $C^1$ functions, the functions $w_j$ are $C^1$ on $(0,R)\cup (R,+\infty)$
(we use \eqref{eq:pre_IFT} and its analogue for $w_2$ to show that the assumptions of the implicit function theorem are satisfied). By a simple bootstrap argument, if $\rho$ is $C^k$ away from the centers $\bR_j$, then the functions $w_j$ are $C^k$ on $(0,R)\cup (R,+\infty)$.

\section*{Ackowledgements}

R.B. thanks Michael Herbst, Emmanuel Giner and Laurent Vidal for useful discussions and technical help. The authors are also very grateful to Antoine Levitt and Alston Misquitta for stimulating discussions. This publication is part of a project that has received funding from the European Research Council (ERC) under the European Union’s Horizon 2020 Research and Innovation Programme (Grant Agreement $n^{\circ}$ 810367).

\section*{Appendix A: Non-uniqueness of GISA fixed-points}
\label{appendix:GISA_dependence_initial_guess}

Since GISA is covered by the framework~\eqref{eq:ISAopt}, uniqueness of a minimizer, as stated in Theorem~\ref{thm:WPcontPb} for ISA and L-ISA, can not be guaranteed. 
We have therefore investigated the uniqueness of minimizers and have observed numerically a sensitivity of the GISA solution (obtained by the fixed-point iterations described in Section \ref{sec:GISA}) to the initial guess that we report in this Appendix. 
More specifically, for a (non-dimensional) test density sum of two normalized Gaussian functions:
\begin{equation}
\rho(\mathbf{r})= \left(\frac{\alpha_1}{\pi}\right)^{\frac{3}{2}} e^{-\alpha_1 |\mathbf{r}-\mathbf{R_1}|^2} + \left(\frac{\alpha_2}{\pi}\right)^{\frac{3}{2}} e^{-\alpha_2 |\mathbf{r}-\mathbf{R_2}|^2}
\label{eq:test_density_GISA_sum_2_gaussians}
\end{equation} 
with $\alpha_1=0.1$, $\alpha_2=0.5$, $|\bR_1-\bR_2|=1.131$, using six shells on each atom in the GISA pro-atomic densities variational space, and Gaussian exponents 
$$\alpha_{1,k} \in \left\lbrace 0.01,0.1,1,2,5,10 \right\rbrace, \quad  \alpha_{2,k} \in \left\lbrace0.05,0.5,2,4,10,50 \right\rbrace, 
$$
we found two local minima $\rho_a^{0,*}$ and $\rho_a^{0,**}$, whose profiles $r \mapsto \log\left(4 \pi r^2 \rho_a^0(r) \right)$ are represented in Figure \ref{fig:GISA_local_global_minimum}. The expected solution, corresponding to one unit charge on each atom was obtained e.g. with the initial guess:
\begin{equation}\label{eq:initial_guess_1}
\left(c_{1,k}^{(0),*}\right)_{k=1..6} = ( 0,0,0,1,0,0 ), \quad \left( c_{2,k}^{(0),*}\right)_{k=1..6} = (0,0,0,0,1,0 ),
\end{equation}
(\textit{i.e.} with weights initially on more compact Gaussians than the expected solutions), while a more balanced initial guess:
\begin{equation}\label{eq:initial_guess_2}
\left(c_{1,k}^{(0),**}\right)_{k=1..6} = \left( \frac{1}{6},\frac{1}{6},\frac{1}{6},\frac{1}{6},\frac{1}{6},\frac{1}{6} \right), \quad  \left( c_{2,k}^{(0),**}\right)_{k=1..6} = \left(\frac{1}{6},\frac{1}{6},\frac{1}{6},\frac{1}{6},\frac{1}{6},\frac{1}{6} \right)
\end{equation}
leads to a fixed-point associated to slightly different pro-atomic densities (see Figure \ref{fig:GISA_local_global_minimum}, right panel) and slightly distorted atomic charges and dipoles, see Table \ref{tab:local_global_minimum_GISA_atomic_charges_dipoles}.

\FloatBarrier
\begin{table}[!htp]
    \centering
\begin{tabular}{|>{\centering}m{3.cm}|>{\centering}m{4.cm}|>{\centering}m{4.cm}|}
    \hline
     & Fixed point 1 & Fixed point 2 \tabularnewline
     & (expected solution) & 
     \tabularnewline
     \hline
     \hline
    Charge $q_1$ & 1.000 &  0.977 \tabularnewline
    \hline
    Charge $q_2$ & 1.000 &  1.023 \tabularnewline
    \hline
    Dipole $d_z^{1}$ & 0.000 &  0.020 \tabularnewline
    \hline
    Dipole $d_z^{2}$ & 0.000 &  0.006 \tabularnewline
    \hline
       \end{tabular}
    \caption{Local charges and dipoles (component along the $z$ axis) on the two atoms computed by GISA from two different initial guesses \eqref{eq:initial_guess_1} and \eqref{eq:initial_guess_2}.}
    \label{tab:local_global_minimum_GISA_atomic_charges_dipoles}
\end{table}
\FloatBarrier

\FloatBarrier
\begin{figure}[htp]
  \centering
  \begin{tabular}{cc}
    \includegraphics[width=70mm]{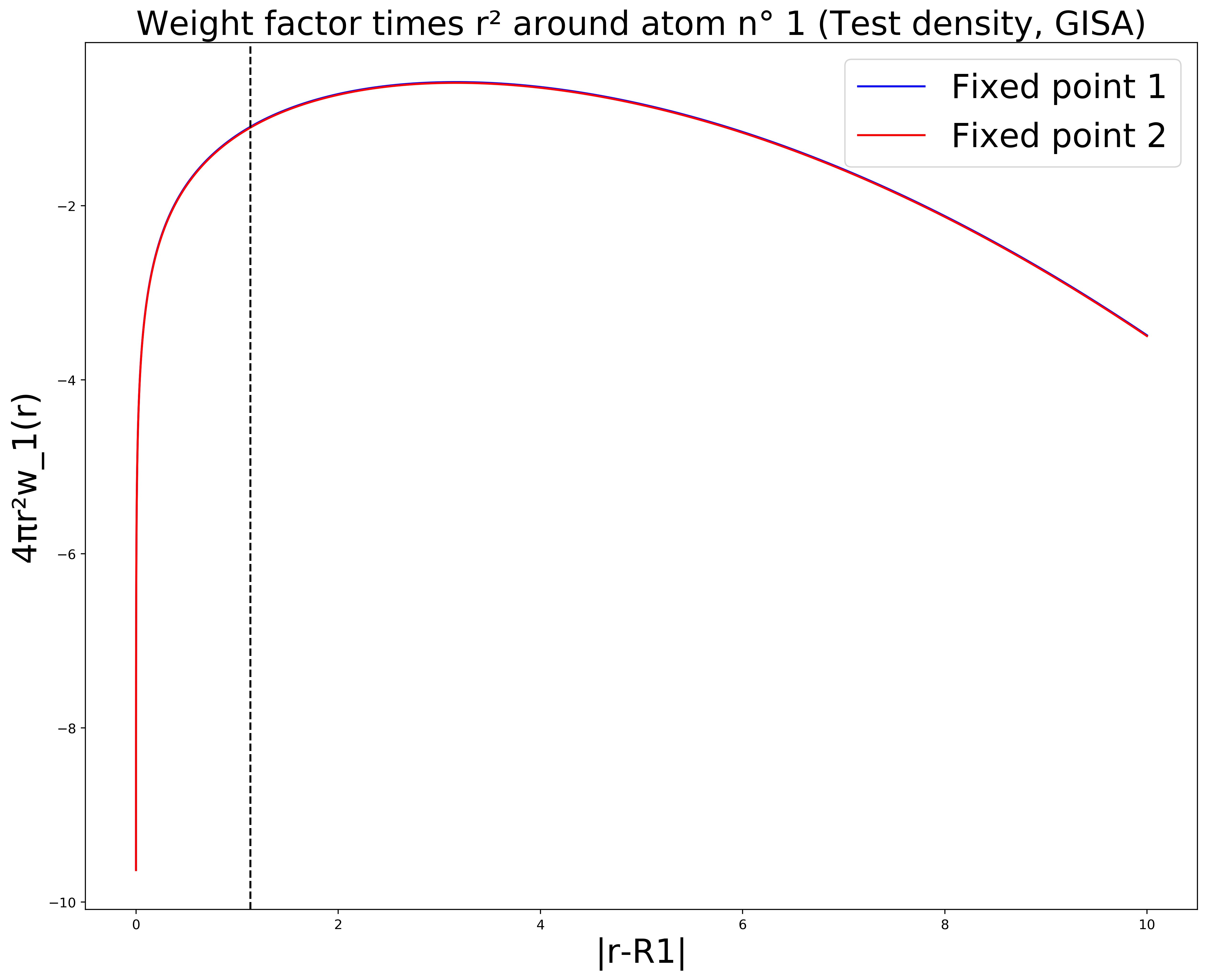}&
    \includegraphics[width=70mm]{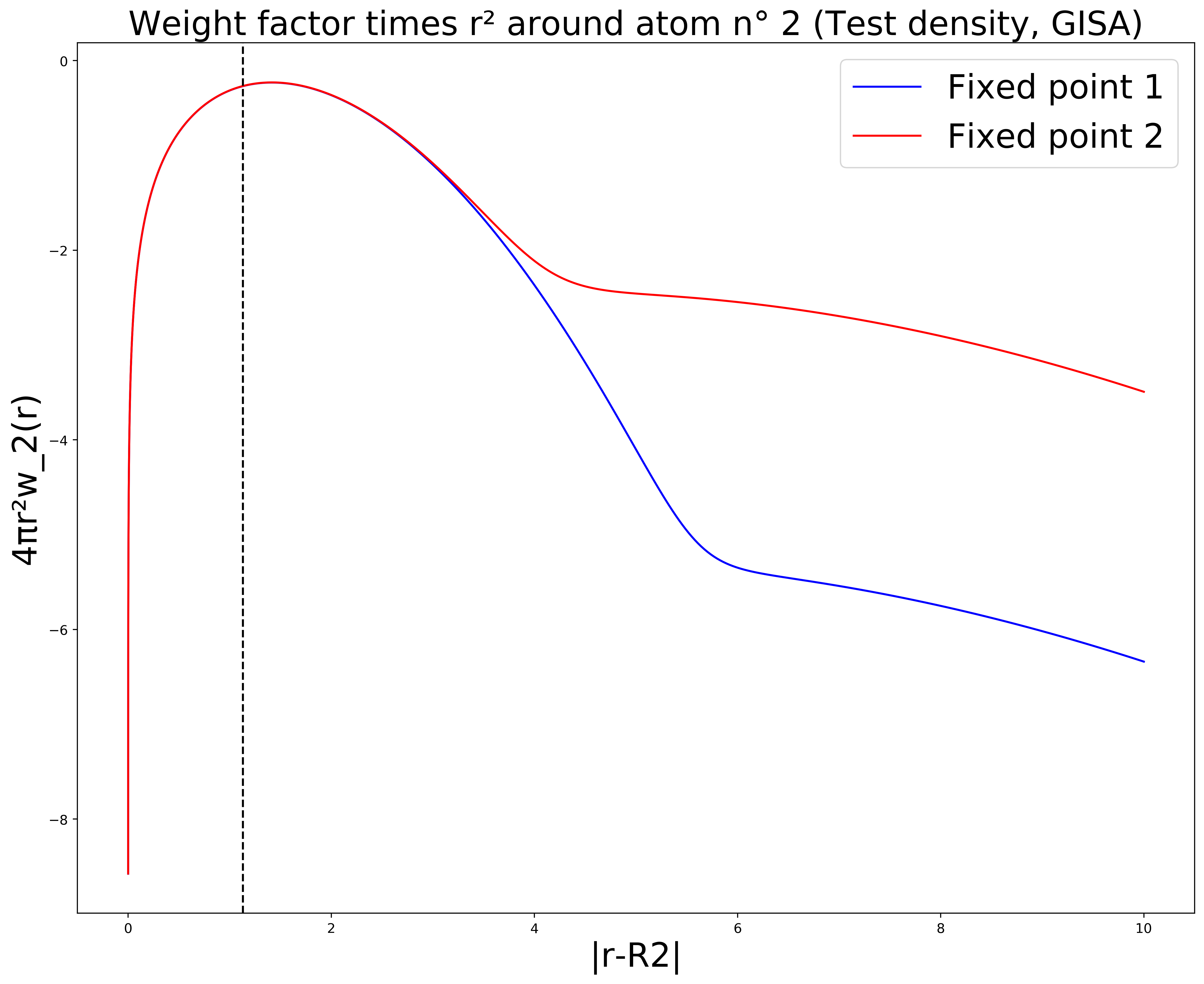}\\
  \end{tabular}
  \caption{Pro-atomic density profiles $r \mapsto \log\left(4 \pi r^2 \rho_a^{0,*}(r) \right)$ and $r \longmapsto \log\left(4 \pi r^2 \rho_a^{0,**}(r) \right)$ associated to the two unveiled fixed points (the expected one, and another one) of the GISA algorithm.}
  \label{fig:GISA_local_global_minimum}
\end{figure}
\FloatBarrier

\section*{Appendix B. Implicit function theorem for Lipschitz functions}

\begin{theorem}\label{thm:IFT_Lip}[\cite{wuertz2008implicit}] Let $U$ and $V$ open subsets of $\R^m$ and $\R^n$ respectively, $F \in C^{0,1}(U \times V;\R^n)$, and $(a,b) \in U \times V$ such that $F(a,b)=0$. Assume that there exists $K > 0$ such that
$$
\forall (x,(y_1,y_2)) \in U \times (V \times V), \quad |F(x,y_1)-F(x,y_2)| \ge K |y_1-y_2|.
$$
Then, there exists an open neighborhood $\widetilde U$ of $a$ in $U$ and a Lipschitz function $\phi:\widetilde U \to  V$ such that $\phi(a)=b$ and 
$$
\{(x,y) \in \widetilde U \times V \; | \; F(x,y)=0\} = \{(x,\phi(x)), \, x \in \widetilde U\}.
$$
In particular, $F(x,\phi(x))=0$ for all $x \in \widetilde U$.
\end{theorem}

 \bibliographystyle{plain}
\bibliography{biblio}

\end{document}